\documentclass[10pt, leqno]{amsart}
\usepackage{latexsym, amsfonts,mathrsfs, amsmath, amssymb, amscd, epsfig}
\usepackage{mathrsfs}
\usepackage[usenames,dvipsnames]{xcolor}
\usepackage{amsthm}
\usepackage{array}
\usepackage{psfrag}
\usepackage{bm}
\usepackage{setspace}
\usepackage{soul}
\usepackage{graphicx}
\usepackage{pdfpages}

\numberwithin{equation}{section}

\newtheorem{lemma}{Lemma}[section]
\newtheorem{theorem}{Theorem}
\newtheorem{proposition}[lemma]{Proposition}
\newtheorem{corollary}[lemma]{Corollary}
\newtheorem{definition}[lemma]{Definition}
\newtheorem{remark}[lemma]{Remark}
\newtheorem{problem}{Problem}
\newtheorem{condition}[lemma]{Condition}

\newtheorem*{note}{Note}
\newtheorem*{notation}{Notations}
\newtheorem*{proclaim}{Proclamation}
\theoremstyle{definition}

\def\beq#1\eeq{\begin{equation}#1\end{equation}}
\def\balign #1 #2 \ealign{\begin{aligned} #1 #2  \end{aligned} }

\def\Div{{\rm div}}

\def\sgn{{\rm sgn}}

\def\bu{\mathbf{u}}

\def\msE{\mathscr{E}}

\def \msB {\mathscr{B}}

\newcommand \alp{\alpha}
\newcommand \eps{\varepsilon}
\newcommand \vphi{\varphi}

\newcommand \Gam{\Gamma}
\newcommand \gam{\gamma}
\newcommand \om{\omega}
\newcommand \tx{\text}
\newcommand \R{\mathbb{R}}
\newcommand \til{\tilde}
\newcommand \wtil{\widetilde}

\newcommand \der{\partial}
\newcommand \mcl{\mathcal}

\newcommand \ol{\overline}
\newcommand \Om{\Omega}

\newcommand \Gamen{\Gamma_0}
\newcommand \Gamex{\Gamma_L}
\newcommand \Gamw{\Gamma_w}

\newcommand \rx{{\rm{x}}}

\newcommand \tpsi{\til{\psi}}
\newcommand \tPsi{\til{\Psi}}

\def \msB {\mathscr{B}}

\def\Div{{\rm div}}

\def\sgn{{\rm sgn}}

\newcommand \bx{\mathbf{x}}

\newcommand \us{u_s}
\newcommand{\mfrak}{\mathfrak}
\newcommand \gs{g_{\rm s}}
\newcommand \ls{l_{\rm s}}

\newcommand \barrhoi{\bar{\rho}_I}

\newcommand \barui{\bar u_{I}}
\newcommand \fsonic{\mathfrak{f}_{\rm sn}}

\newcommand \Le{L_{\rm e}}
\newcommand \Tac{\mcl{T}_{\rm acc}}

\newcommand \umax{u_{\rm max}}
\newcommand \lmax{l_{\rm max}}
\newcommand \tphi{\tilde{\phi}}
\newcommand \tilT{\tilde{T}}
\newcommand \iterV{\mcl{V}}
\newcommand \iterT{\mcl{T}}
\newcommand \iterP{\mcl{P}}
\newcommand \iterseta{\iterV(r_2)\times \iterP(r_3)}

\newcommand \bJ{\bar{J}}
\newcommand \ubJ{\underbar{J}}

\newcommand \Veps{V_{\eps}}

\begin{document}
\title[Accelerating flows with transonic $C^1$-transitions]
{The steady Euler-Poisson system and accelerating flows with transonic $C^1$-transitions }

\author{Myoungjean Bae}
\address{}
\email{mjbae@kaist.ac.kr}
\author{Ben Duan}
\address{}
\email{bduan@dlut.edu.cn}

\author{Chunjing Xie}
\address{
}
\email{cjxie@sjtu.edu.cn}

\begin{abstract}
In this paper, we prove the existence of two-dimensional solutions to the steady Euler-Poisson system with {\emph{continuous transonic transitions}} across {\emph{sonic interfaces}} of codimension 1. First, we establish the well-posedness of a boundary value problem for a linear second order system that consists of an {\emph{elliptic-hyperbolic}} mixed type equation with a degeneracy occurring on an interface of codimension 1, and an elliptic equation weakly coupled together. Then we apply the Schauder fixed point theorem to prove the existence of two-dimensional solutions to the potential flow model of the steady Euler-Poisson system with continuous transonic transitions across sonic interfaces. With the aid of Helmholtz decomposition, established in \cite{BDX3}, we extend the existence result to the full Euler-Poisson system for the case of nonzero vorticity. Most importantly, the solutions constructed in this paper are classical solutions to Euler-Poisson system, thus their sonic interfaces are not weak discontinuities in the sense that all the flow variables are $C^1$ across the interfaces. 

\end{abstract}

\keywords{accelerating flow, $C^1$-transonic transition, elliptic-hyperbolic mixed equation with a degeneracy, Euler-Poisson system, Helmholtz decomposition, singular perturbation, sonic interface}
\subjclass[2010]{
35J47, 35J57, 35J66, 35M10, 76N10}

\date{\today}

\maketitle

\tableofcontents

\section{Introduction and Main theorems}
\subsection{Introduction}
The Euler-Poisson system
\begin{equation}\label{UnsteadyEP}
\left\{
\begin{aligned}
& \rho_t+ \Div_{\rx} (\rho \bu)=0, \\
& (\rho \bu)_t+\Div_{\rx} (\rho \bu \otimes \bu) +\nabla_{\bx} p=\rho \nabla_{\rx} \Phi, \\
& (\rho \msE)_t +\Div_{\rx}(\rho\msE \bu +p\bu)=\rho \bu\cdot \nabla_{\bx}\Phi,\\
& \Delta_{\rx} \Phi=\rho-\barrhoi,%
\end{aligned}%
\right.
\end{equation}
describes the motion of plasma or electrons in semiconductor devices
where $\rho$, ${\bf u}$, $p$ and $\msE$ represent the density, velocity, pressure and the energy density of electrons, respectively. In particular, this system describes an electron fluid dynamics in a background of ions with a constant density and a constant charge (\cite{Guo99} and \cite{MarkRSbook}). The Poisson equation $\tx{\eqref{UnsteadyEP}}_4$ describes how the two fluids of electrons and ions interact through the electric field $\nabla \Phi$. In this paper, we consider ideal polytropic gas, where $p$ and $\msE$ are given by
\begin{equation}
p=S\rho^{\gam}\quad\tx{and}\quad \msE=\frac{|\bu|^2}{2}+\frac{p}{(\gam-1)\rho},
\end{equation}
respectively, for a function $S>0$ and a constant $\gam>1$.
The function $\ln S$ represents {\emph{the physical entropy}}, and the constant $\gam>1$ is called the {\emph{adiabatic exponent}}. For the Mach number $M$, defined by
$$M:=\frac{|{\bf u}|}{\sqrt{\gam p/\rho}},$$
if $M<1$, then the flow corresponding to $(\rho, {\bf u}, p, \Phi)$ is said to be subsonic. On the other hand, if $M>1$, then the flow is said to be supersonic. Finally, if $M=1$, then the flow is said to be sonic.

The goal of this work is to prove the existence of a steady classical solution to the system \eqref{UnsteadyEP} that yields a flow accelerating from a subsonic state to a supersonic state in a two-dimensional flat nozzle of a finite length. The main feature of the solution constructed in this paper is that the sonic interface on which the subsonic-supersonic transition occurs has codimension 1. Most importantly, the sonic interface is not treated as a free boundary as opposed to the case of transonic shocks (cf. \cite{bae2011transonic,  chen2004steady, chen2008trans, liu2009global, xin2009transonic, XinYin_CPAM, xin2008transonic} and references therein).

The system \eqref{UnsteadyEP} consists of two parts, a compressible Euler system with source terms and a Poisson equation, weakly coupled in a nonlinear way. One of longstanding open problems related to the compressible Euler system is to prove the existence of a solution that consists of two types of transonic flows, an accelerating transonic flow with a sonic interface and a transonic shock in a convergent-divergent nozzle, called a {\emph{de Laval nozzle}}. For the last several decades, there have been extensive studies on transonic shocks in a nozzle, see  \cite{bae2011transonic, ChenSXnozzle, LiXinYin_CMP, LiXinYin_ARMA, XinYin_CPAM, xin2008transonic}. On the other hands, there are very few known results on continuous transonic flows with sonic interfaces as it involves an analysis of degenerate PDEs. In the pioneering works \cite{Morawetz1, Morawetz2, Morawetz3, Morawetz64} by Morawetz, it is shown that two-dimensional smooth transonic potential flows around an airfoil are unstable in general. The smooth transonic solutions for transonic small disturbance equations was analyzed in \cite{KZ}. Recently, the transonic solutions for the Euler system with sonic interfaces are studied, see \cite{WangXin1, WangXin2, WangXin, WengXin}, etc. Another examples of weak solutions to the unsteady Euler system with sonic interfaces, which appear as circular arcs, are given in \cite{bae2013prandtl, CF2, CFbook, elling2008supersonic}. In \cite{BCF}, it is shown that the sonic interfaces studied in \cite{bae2013prandtl, CF2, CFbook, elling2008supersonic} are weak discontinuities, in the sense that the velocity potentials across the sonic arcs are $C^{1,1}$ but not $C^2$. This means that the flow velocity fields are continuous but their derivatives are discontinuous  across the sonic arcs. For more studies on regular shock reflection for other models, one may refer to \cite{JegdicKeyfitz, KeyfitzTesdall, Zhengregular} and references therein.

In this paper, we prove the existence of two dimensional solutions to the steady Euler-Poisson system with {\emph{continuous transonic transitions}} across {\emph{sonic interfaces}} of codimension 1. More importantly, we show that the solutions are classical solutions, thus their sonic interfaces are not weak discontinuities. Differently from the Euler system, the steady Euler-Poisson system has one-dimensional continuous transonic solutions, and this can be found by a phase plane analysis \cite{Markphase, LuoXin}. According to \cite{LuoXin}, there are two types of continuous transonic solutions to the steady Euler-Poisson system under the condition of \eqref{condition for zeta0}, an accelerating transonic solution and a decelerating solution. In this paper, we show that the one-dimensional accelerating transonic solution is in fact $C^{\infty}$ smooth, see Lemma \ref{lemma-1d-full EP}. And, we prove the existence of two-dimensional solutions to the potential flow model of the steady Euler-Poisson system with continuous transonic transitions across sonic interfaces. Furthermore, with the aid of Helmholtz decomposition, established in \cite{BDX3}, we extend the existence result to the full Euler-Poisson system for the case of nonzero vorticity. Prior to this work, the authors of this paper studied the multi-dimensional subsonic solutions and supersonic solutions for the Euler-Poisson system, see \cite{BCF2, BDX3, BDX}. The ultimate goal of this work and the works in \cite{BCF2, BDX3, BDX} is to establish the existence of a multi-dimensional solution to the steady Euler-Poisson system with a continuous transonic-transonic shock configuration in a flat nozzle. For other results on continuous transonic solutions to the one-dimensional Euler-Poisson system, one can refer to \cite{Gamba, GambaMorawetz, ChenMeiZhangZhang, LiMeiZhangZhang2,WeiMeiZhangZhang} and the references therein.
\smallskip

As we seek for a two-dimensional steady solution to \eqref{UnsteadyEP}, let us set the velocity vector field ${\bf u}$ as ${\bf u}(\rx):=u_1(\rx){\bf e}_1+u_2(\rx){\bf e}_2$ for ${\rx}=(x_1, x_2)\in \R^2$. Here, ${\bf e}_j(j=1,2)$ represents the unit vector in the positive $x_j$-direction. Next, we define the vorticity function $\om$ by
$$\om(\rx):=\der_{x_1}u_2-\der_{x_2}u_1.$$
If the solution satisfies $\rho>0$ and $u_1>0$, then $(\rho, {\bf u}, S, \Phi)$ solves the following nonlinear system:
\begin{equation}\label{steadyEP}
  \left\{
\begin{aligned}
& \nabla\cdot (\rho \bu)=0 \\
&\om=\frac{1}{u_1}\left(\frac{\rho^{\gam-1}\der_{x_2} S}{(\gam-1)}-\der_{x_2}(\msB-\Phi)\right) \\
& {\bf m}\cdot \nabla S=0\\
& {\bf m}\cdot \nabla(\msB-\Phi)=0\\
& \Delta\Phi=\rho-\barrhoi
\end{aligned}%
\right.
\end{equation}
where ${\bf m}$ and $\msB$ are given by
$$
{\bf m}:=\rho{\bf u},\quad\tx{and}\quad
\msB:=\frac{|{\bf u}|^2}{2}+\frac{\gam S\rho^{\gam-1}}{\gam-1}.
$$

For one-dimensional solutions to \eqref{steadyEP}, we  solve the ODE system
\begin{equation*}
  \left\{
\begin{aligned}
& \bar{\rho}\bar u_1=J \\
&\bar S=S_0\\
&\frac{\bar u_1^2}{2}+\frac{\gam \bar S\bar{\rho}^{\gam-1}}{\gam-1}-\bar{\Phi}=\kappa_0\\
& \bar{\Phi}''=\bar{\rho}-\barrhoi
\end{aligned}%
\right.
\end{equation*}
for some constants $J>0$, $S_0>0$ and $\kappa_0\in \R$.
By setting as $\bar E:=\bar{\Phi}'$, one can rewrite the above system as
\begin{equation*}
\bar{\rho}=\frac{J}{\bar u_1},\quad
   \bar S=S_0,\quad
  \begin{cases}\bar u_1'=\frac{\bar E \bar u_1^{\gam}}{\bar u_1^{\gam+1}-\gam S_0J^{\gam-1}},\\
  \bar E'=\frac{J}{\bar u_1}-\barrhoi.
  \end{cases}
\end{equation*}

Hereafter, the constants $\gam >1$, $J>0$ and $S_0>0$ are fixed.
Let us define two constants $\barui$ and $\us$ by
\begin{equation*}
  \barui:=\frac{J}{\barrhoi},\quad\tx{and}\quad \us:=\left(\gam S_0J^{\gam-1}\right)^{\frac{1}{\gam+1}}.
\end{equation*}
For later use, we shall define a parameter $\zeta_0$ by
\begin{equation}\label{definition of zeta0}
  \zeta_0:=\frac{\barui}{\us}.
\end{equation}
And, assume that
\begin{equation}
\label{condition for zeta0}
\zeta_0>1.
\end{equation}

Under this assumption, suppose that $(\bar u_1, \bar E)$ is a $C^1$-solution to the system
\begin{equation}
\label{EP-1d-reduced}
   \begin{cases}
  \bar u_1'=\frac{\bar E \bar u_1^{\gam}}{\bar u_1^{\gam+1}-\us^{\gam+1}},\\
  \bar E'=\frac{J}{\bar u_1}-\barrhoi.
  \end{cases}
\end{equation}
In addition, suppose that
\begin{equation*}
(\bar u_1, \bar{E})(\ls)=(\us, 0)
\end{equation*}
for some constant $\ls>0$. Next, we define a function $H:(0,\infty)\rightarrow \R$ by
\begin{equation}\label{definition of H}
  H(u)=\int_{\us}^{u} \frac{J}{\barui t^{\gam+1}}(t^{\gam+1}-\us^{\gam+1})
\left(\barui-t\right)dt.
\end{equation}
Also, define a two-variable function $\mfrak{H}:(0,\infty)\times \R\rightarrow \R$ by
\begin{equation*}
\mfrak{H}(u, E):=\frac{1}{2}E^2-H(u).
\end{equation*}
Straightforward computations show that the solution $(\bar u_1, \bar E)$ satisfies
\begin{equation}
\label{1d-Hamiltonian}
  \mfrak{H}(\bar u_1, \bar E)\equiv 0
\end{equation}
as long as it exists.
\smallskip

Define
\begin{equation*}
\mcl{T}:=\{(u, E):\mfrak{H}(u, E)=0\}.
\end{equation*}
The set $\mcl{T}$ is called the {\emph{critical trajectory}} of the system \eqref{EP-1d-reduced}(see Figure \ref{figure1}).
We define a subset $\mcl{T}_{\rm{acc}}$ by
\begin{equation}
\label{definition of T and Tpm}
  \mcl{T}_{\rm{acc}}=\{(u, E)\in\mcl{T}:  (u-u_s) E \ge 0 \}.
\end{equation}

\begin{psfrags}
\begin{figure}[htp]
\centering
\psfrag{A}[cc][][0.8][0]{$\phantom{aaaa}(u_0, E_0)$}
\psfrag{B}[cc][][0.8][0]{$u_s$}
\psfrag{C}[cc][][0.8][0]{$\barui$}
\psfrag{u}[cc][][0.8][0]{$u$}
\psfrag{E}[cc][][0.8][0]{$E$}
\includegraphics[scale=0.5]{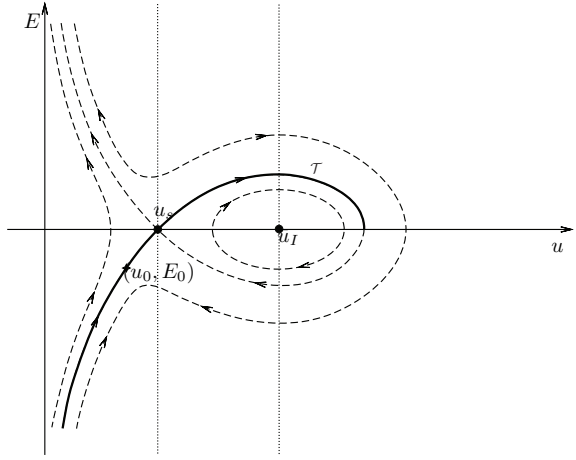}
\caption{The critical trajectory}\label{figure1}
\end{figure}
\end{psfrags}

\begin{lemma}
\label{lemma-1d-full EP}
Given constants $\gam> 1$, $J>0$ and $S_0>0$, assume that the condition \eqref{condition for zeta0} holds.
Then, the set $\mcl{T}_{\rm{acc}}$ represents the trajectory of an accelerating smooth transonic solution to \eqref{EP-1d-reduced} in the following sense: for any fixed $(u_0, E_0)\in \mcl{T}_{\rm{acc}}$ with $u_0<\us$, there exists a finite constant $l_{\rm max}>0$ depending only on $(\gam, J, S_0, \barui, u_0)$ so that the initial value problem
\begin{equation}
\label{1d-ivp}
\begin{split}
   &\begin{cases}
  \bar u_1'=\frac{\bar E \bar u_1^{\gam}}{\bar u_1^{\gam+1}-\us^{\gam+1}}\\
  \bar E'=\frac{J}{\bar u_1}-\barrhoi
  \end{cases}
\tx{for $x_1>0$},\\
&(\bar u_1, \bar E)(0)=(u_0, E_0)
\end{split}
\end{equation}
has a unique smooth solution for $x_1\in[0, l_{\rm max})$ with satisfying the following properties:
\begin{itemize}
\item[(i)] $\displaystyle{\bar u_1'(x_1)>0\,\,\tx{on $[0, l_{\rm max})$}}$.
\item[(ii)] $\displaystyle{\lim_{x_1\to l_{\rm max}-}\bar u_1'(x_1)=0}$.
\item[(iii)] $\displaystyle{\mcl{T}_{\rm acc}\cap\{(\bar u_1, \bar E)(x_1):0\le x_1\le l_{\rm max}\}=\mcl{T}_{\rm acc}\cap \{(u, E):u\ge u_0\}}$.
\item[(iv)] There exists a unique constant $\ls\in(0, l_{\rm max})$ depending only on $(\gam, J, S_0, \barui, u_0)$ so that $\bar u_1$ satisfies
    \begin{equation}
    \label{definition:ls}
      \bar u_1(x_1)\begin{cases}
      <\us\quad &\mbox{for $x_1<\ls$ {\emph{(subsonic)}}},\\
      =\us\quad &\mbox{at $x_1=\ls$ {\emph{(sonic)}}},\\
      >\us\quad &\mbox{for $x_1>\ls$ {\emph{(supersonic)}}}.
      \end{cases}
    \end{equation}
\end{itemize}

\begin{proof}
Define a function $F(t)$ by
\begin{equation}\label{definition of F in ode}
  F(t):=\frac{t^{\gam}\sqrt{2H(t)}}{|t^{\gam+1}-\us^{\gam+1}|}\quad\tx{for $t\neq \us$}
\end{equation}
where $H(u)$ is given by \eqref{definition of H}.

Clearly, $F(t)$ is smooth for $t\in(0,\infty)\setminus \{\us\}$ as long as $H(t)\ge 0$.

Note that $H(\us)=H'(\us)=0$. It follows from the assumption \eqref{condition for zeta0} that one has
$$H''(\us)=(\gam+1)J\left(\frac{1}{\us}-\frac{1}{\barui}\right)>0.$$
By applying L'H\^{o}pital's rule, it is obtained
\begin{equation*}
\lim_{t\to \us}\frac{2H(t)}{(t^{\gam+1}-\us^{\gam+1})^2}=
\lim_{t\to \us} \frac{H''(t)}{(\gam+1)^2t^{2\gam}}=\frac{J}{(\gam+1)\us^{2\gam}}
\left(\frac{1}{\us}-\frac{1}{\barui}\right)>0
\end{equation*}
so one can extend the definition of $F(t)$ up to $t=\us$ as
\begin{equation*}
  F(\us):=
  \sqrt{\frac{J}{\gam+1}\left(\frac{1}{\us}-\frac{1}{\barui}\right)}.
\end{equation*}
Therefore, $F(t)$ is continuous for all $t>0$ as long as $H(t)\ge 0$ holds.

By Taylor's theorem, there exist two smooth functions $P(t)$ and $Q(t)$ satisfying
\begin{equation*}
\begin{split}
  &H(t)=\frac 12 H''(\us)(t-\us)^2+(t-\us)^3 P(t),\\
  &t^{\gam+1}-\us^{\gam+1}=(\gam+1)\us^{\gam} (t-\us)+(t-\us)^2Q(t),
\end{split}
\end{equation*}
so that the function $F(t)$ is written as
\begin{equation*}
  F(t)=\frac{t^{\gam}\sqrt{H''(\us)+2(t-\us)P(t)}}{(\gam+1)\us^{\gam}+(t-\us)Q(t)}.
\end{equation*}
This shows that $F(t)$ is smooth and positive for all $t>0$ as long as the inequality $H(t)\ge 0$ holds.

By the definition \eqref{definition of T and Tpm}, one has
\begin{equation*}
  \Tac=\{E=\sgn (u-\us)\sqrt{H(u)}: H(u)\ge 0\}.
\end{equation*}
If $(\bar u_1, \bar E)(x_1)$ is a $C^1$-solution to \eqref{1d-ivp}, then it also solves the initial value problem:
\begin{equation}
\label{ivp-accelerating}
  \begin{cases}
  \bar u_1'=F(\bar u_1)\\
  \bar u_1(0)=u_0
  \end{cases}\quad\tx{and}\quad \bar E=\sgn(\bar u_1-\us)\sqrt{2H(\bar u_1)}.
\end{equation}

One can directly check from \eqref{definition of H} that there exists a constant $u_{\rm max}\in(\barui, \infty)$ satisfying that
\begin{equation*}
  H(t)\begin{cases}
  >0\quad&\mbox{for $0<t<u_{\rm max}$},\\
  =0\quad&\mbox{at $t=u_{\rm max}$},\\
  <0\quad&\mbox{for $t>u_{\rm max}$}.
  \end{cases}
\end{equation*}
Hence it follows from the unique existence theorem of ODEs and the method of continuation that there exists a finite constant $l_{\rm max}>0$ so that the initial value problem \eqref{ivp-accelerating} has a unique smooth solution for $0\le x_1\le l_{\rm max}$ with $(\bar u, \bar E)(l_{\rm max})=(u_{\rm max},0)$. Note that the inequality
\begin{equation*}
 0<u_0<\us<\barui<u_{\rm max}
\end{equation*}
holds. Since $F(t)>0$ for $0<t\le u_{\rm max}$, the velocity function $\bar u_1$ strictly increases with respect to $x_1\in[0, l_{\rm max})$. Therefore, there exists a unique constant $\ls\in(0, l_{\rm max})$ that satisfies \eqref{definition:ls}.

\end{proof}
\end{lemma}

\subsection{Main theorems}
\label{subsection-the main result}
The main goal of this work is to construct a two-dimensional solution of the system \eqref{steadyEP} when prescribing the boundary conditions as small perturbations of a one-dimensional accelerating smooth transonic solution introduced in Lemma \ref{lemma-1d-full EP}.

In this paper, we assume that
\begin{equation*}
  \msB-\Phi= 0\quad\tx{({\emph{the pseudo Bernoulli's law}}  )}
\end{equation*}
for technical simplicity.
Then, we can rewrite \eqref{steadyEP} as follows:
\begin{equation}\label{full EP rewritten}
  \left\{
\begin{aligned}
& \nabla\cdot {\rho{\bf u}}=0 \\
&\nabla\times {\bf u}=\frac{\rho^{\gam-1}\der_{x_2} S}{(\gam-1)u_1}\\
&  {\rho{\bf u}}\cdot \nabla S=0\\
&\frac 12|{\bf u}|^2+\frac{\gam S \rho^{\gam-1}}{\gam-1}=\Phi\\
& \Delta \Phi=\rho-\barrhoi.
\end{aligned}%
\right.
\end{equation}

For a constant $L>0$, let us define
\begin{equation}
\label{definition of Omeage L}
  \Om_L:=\{\rx=(x_1, x_2)\in \R^2: 0<x_1<L, |x_2|<1\}.
\end{equation}
The boundary $\der \Om_L$ consists of three parts: the entrance $\Gamen$, the wall $\Gamw$ and the exit $\Gamex$ given by
\begin{equation}
\label{definition of der Omeage L}
  \Gamen:=\{0\}\times [-1,1],\quad \Gamw:=(0,L)\times\{-1,1\},\quad \Gamex:=\{L\}\times [-1,1],
\end{equation}
respectively.

\begin{notation}\begin{itemize}
\item[(1)] Throughout this paper, we let ${\bf e}_j$ represent the unit vector in the positive direction of the $x_j$-axis, that is,
\begin{equation*}
 {\bf e}_1={(1,0)}\quad\tx{and}\quad  {\bf e}_2={(0,1)}.\\
\end{equation*}
\item[(2)]  For $i=1,2$, the symbols $\der_i$ and $\der_{ij}$ represent $\frac{\der}{\der x_i}$ and $\der_i\der_j$, respectively. And, any partial derivative of a higher order is denoted similarly.
\end{itemize}
\end{notation}

\begin{definition}[Background solutions]
\label{definition of background solution-full EP}
Given constants $\gam > 1$, $S_0>0$, $J>0$ and $\zeta_0>1$(see \eqref{definition of zeta0} for the definition), fix a point $(u_0, E_0)\in \Tac$ with $E_0 < 0$. Then Lemma \ref{lemma-1d-full EP} implies that the initial value problem \eqref{EP-1d-reduced} with
$$(\bar u_1, \bar E)(0)=(u_0,E_0)$$
has a unique smooth solution on the interval $[0, l_{\rm max})$. Furthermore, the flow governed by the solution is accelerating. In other words, it holds that
$$
\frac{d\bar u_1}{dx_1}>0\quad\tx{on $[0, l_{\rm max})$.}
$$
For ${\rx}=(x_1, x_2)\in \R^2$ with $x_1\in[0, l_{\rm max})$, let us define $(\bar{\rho}, {\bar{\bf u}}, \bar{\Phi})$ by
\begin{equation*}
\begin{split}
& \bar{\rho}(\rx):=\frac{J}{\bar u_1(x_1)},\quad
\bar{\bf u}(\rx):=\bar u_1(x_1){\bf e}_{1},\quad
\bar{\Phi}(\rx):=\int_0^{x_1}\bar E(t)\,dt+\frac{u_0^2}{2}+\frac{\gam S_0}{\gam-1}\left(\frac{J}{u_0}\right)^{\gam-1}.
  \end{split}
\end{equation*}
Then, $(\rho, {\bf u},S, \Phi)=(\bar{\rho}, \bar{\bf u}, S_0, \bar{\Phi})$ solves the system \eqref{full EP rewritten} in $\Om_{l_{\rm max}}$. We call $(\bar{\rho}, \bar{\bf u}, S_0, \bar{\Phi})$ the background smooth transonic solution to \eqref{full EP rewritten} associated with $(\gam, \zeta_0, J, S_0, E_0)$. Note that the initial value $u_0$ of $\bar u_1$ is uniquely determined depending on $E_0$ as the point $(u_0, E_0)$ lies on the curve $\Tac$, defined by \eqref{definition of T and Tpm}, and the function $\sgn (u-\us)\sqrt{H(u)}$ is monotone for $u\in (0, u_{\rm max}]$. Then, it follows from the inequality $E_0<0$ that $0<u<\us$.
\end{definition}

The main goal of this paper is to solve the following problem:
\begin{problem}
\label{problem-full EP} Given constants $\gam>1$, $\zeta_0>1$, $J>0$ and $S_0>0$, fix a point $(u_0, E_0)\in \Tac$ with $E_0<0$. Given $(S_{\rm en}, E_{\rm en}, \om_{\rm en}):[-1,1]\rightarrow \R^3$, define
\begin{equation}
\label{definition-perturbation of bd}
      \mcl P(S_{\rm en}, E_{\rm en}, w_{\rm en}):=\|S_{\rm en}-S_0\|_{C^4([-1,1])}+\|E_{\rm en}-E_0\|_{C^4([-1,1])}
    +\|w_{\rm en}\|_{C^5([-1,1])}.
    \end{equation}
Assuming that $\mcl P(S_{\rm en}, E_{\rm en}, w_{\rm en})$ is sufficiently small, find a solution $(\rho, {\bf u}, S, \Phi)$ to the nonlinear system \eqref{full EP rewritten} in $\Om_L$ with
the boundary conditions:
  \begin{equation}\label{BC-full EP}
  \begin{split}
  {\bf u}\cdot {\bf e}_{2}=w_{\rm en},\quad S=S_{\rm en},\quad \tx{and}
  \quad \der_{x_1}\Phi=E_{\rm en}\quad&\tx{on $\Gamen$}\\
  {\bf u}\cdot {\bf n}=0\quad\tx{and}\quad \der_{\bf n}\Phi=0\quad&\mbox{on $\Gamw$}\\
  \Phi=\bar{\Phi}\quad&\mbox{on $\Gamex$}
  \end{split}
\end{equation}
where ${\bf n}$ is the inward unit normal vector on $\Gamw$.
\end{problem}

By fixing the constants $\gam$, $\zeta_0$, $J$, and $S_0$, the term $\barrhoi$ from the equation
$\eqref{full EP rewritten}_5$ can be expressed as
\begin{equation}
    \label{barrhoi-expression-n}
      \barrhoi=\frac{J}{\zeta_0(\gam S_0J^{\gam-1})^{\frac{1}{\gam+1}}}.
    \end{equation}
    We call the system \eqref{full EP rewritten} with the expression of \eqref{barrhoi-expression-n} {\emph{the simplified steady Euler-Poisson system}} associated with $(\gam, \zeta_0, J, S_0)$. This expression indicates our choice of a background solution (in the sense of Definition \ref{definition of background solution-full EP}) from which we shall build a two-dimensional solution to \eqref{full EP rewritten} with a transonic $C^1$-transition across an interface of codimension 1.

Before stating the main theorem, we list compatibility conditions required for $(S_{\rm en}, E_{\rm en}, w_{\rm en})$.
\begin{condition}
\label{conditon:1}
\begin{itemize}
\item[(i)] For $k=1,3$,
\begin{equation*}
\frac{d^k S_{\rm en}}{dx_2^k}(x_2)=
                  \frac{d^k E_{\rm en}}{dx_2^k}(x_2)=0\quad\tx{at $|x_2|=1$};
\end{equation*}

    \item[(ii)] For $l=0,2,4$
    \begin{equation*}
     \frac{d^l w_{\rm en}}{dx_2^l}(x_2)=0 \quad\tx{at $|x_2|=1$}.
    \end{equation*}
\end{itemize}
\end{condition}

\begin{theorem}
\label{theorem-smooth transonic-full EP}
Let us fix constants $(\gam, \zeta_0, J, S_0, E_0)$ satisfying
\begin{equation*}
\gam>1,\quad \zeta_0>1,\quad  S_0>0.
\end{equation*}
Suppose that $(u_0, E_0)\in \Tac$ with $E_0<0$, which is equivalent to $0<u_0<\us$.
And, let $(\bar{\rho}, {\bar{\bf u}}, \bar{\Phi})$ be the background smooth transonic solution to \eqref{full EP rewritten} associated with $(\gam, \zeta_0, J, S_0, E_0)$. In addition, suppose that the boundary data $(S_{\rm en}, E_{\rm en}, w_{\rm en})$ satisfy Condition \ref{conditon:1}. Then one can fix two constants $\bJ$ and $\ubJ$ depending only on $(\gam, \zeta_0, S_0)$, which satisfy
\begin{equation*}
  0<\bJ<1<\ubJ<\infty
\end{equation*}
so that whenever the background momentum density $J(=\bar{\rho}\bar u_1)$ satisfies
\begin{center}
either $0<J\le \bJ$ or $\ubJ\le J<\infty$,
\end{center}
there exists a constant $d\in(0,1)$ depending on $(\gam, \zeta_0, S_0,J)$ so that if the two constants $u_0$ and $L$ satisfy the condition
\begin{equation}
\label{almost sonic condition1 full EP}
  1-d\le \frac{u_0}{\us}<1 <\frac{\bar u_1(L)}{\us}\le 1+d,
\end{equation}
then $(\bar{\rho}, \bar{\bf u},\bar{\Phi})$ is structurally stable in the following sense: one can fix a constant $\bar{\sigma}>0$ sufficiently small depending only on $(\gam, \zeta_0, S_0, E_0, J, L)$ so that if the inequality
\begin{equation}
\label{smallness of bd}
  \mcl P(S_{\rm en}, E_{\rm en}, w_{\rm en})\le \bar{\sigma}
\end{equation}
holds, then Problem \ref{problem-full EP} has a unique solution $(\rho, {\bf u}, S, \Phi)\in H^3(\Om_L)\times H^3(\Om_L;\R^2)\times H^4(\Om_L)\times H^4(\Om_L)$ that satisfies the estimate
    \begin{equation}
    \label{solution estimate full EP}
    \begin{split}
    &\|\rho-\bar{\rho}\|_{H^3(\Om_L)}+\|{\bf u}-\bar{\bf u}\|_{H^3(\Om_L)}+\|S-S_0\|_{H^4(\Om_L)}+\|\Phi-\bar{\Phi}\|_{H^4(\Om_L)}\\
    &\le
    C\mcl P(S_{\rm en}, E_{\rm en}, w_{\rm en})
    \end{split}
    \end{equation}
    for some constant $C>0$ depending only on $(\gam, \zeta_0, S_0, E_0, J, L)$.
Moreover, there exists a function $\fsonic:[-1,1]\rightarrow (0, L)$ satisfying that
    \begin{equation}
    \label{sonic boundary is a graph pt}
  M\begin{cases}
  <1\quad&\mbox{for $x_1<\fsonic(x_2)$}\\
  =1\quad&\mbox{for $x_1=\fsonic(x_2)$}\\
  >1\quad&\mbox{for $x_1>\fsonic(x_2)$}
  \end{cases}
    \end{equation}
where the {\emph{Mach number}} $M$ of the system \eqref{full EP rewritten} is defined by
    \begin{equation*}
      M:=\frac{|{\bf u}|}{\sqrt{\gam S\rho^{\gam-1}}}.
    \end{equation*}
  And, the function $\fsonic$ satisfies the estimate
    \begin{equation}
    \label{estimate of sonic boundary pt}
      \|\fsonic-\ls\|_{H^2((-1,1))}+\|\fsonic-\ls\|_{C^1([-1,1])}\le C\mcl P(S_{\rm en}, E_{\rm en}, w_{\rm en})
    \end{equation}
    for some constant $C>0$ depending on $(\gam, \zeta_0, S_0, E_0, J, L)$.

\end{theorem}

\begin{definition}[Sonic interface]
For the function $\fsonic$ given in the above theorem, we shall call the curve
\begin{equation*}
\{x_1=\fsonic(x_1):-1\le x_2\le 1\}
\end{equation*}
{\emph{the sonic interface}} of the flow governed by $(\rho, {\bf u}, S, \Phi)$.
\end{definition}

\begin{remark}
\label{remark about pt main theorem-1}
It directly follows from a standard extension method of functions in Sobolev spaces and the generalized Sobolev inequality that $\rho$, ${\bf u}$ and $S$ are in $C^1$, and $\Phi$ is in $C^2$ up to the boundary. This implies that Theorem \ref{theorem-smooth transonic-full EP} yields a classical solution to Problem \ref{problem-full EP}.
\end{remark}

\begin{remark}
\label{remark about pt main theorem-2}
The property \eqref{sonic boundary is a graph pt}
implies that the solution $(\rho, {\bf u}, S, \Phi)$ given in Theorem \ref{theorem-smooth transonic-full EP}  yields a transonic  $C^1$-transition across the sonic interface $x_1=\fsonic(x_2)$.

\end{remark}

There is another example of a sonic interface with a different feature. According to \cite[Theorem 4.1]{BCF}, the function $D{\bf u}$ is discontinuous on the sonic boundaries occurring in the self-similar regular shock reflection or in the self-similar Prandtl-Meyer reflection of potential flows(see \cite{bae2013prandtl, BCF2, CF2, CF3}). The sonic interface across which $D{\bf u}$ is discontinuous is called a {\emph{weak discontinuity}}. From Theorem \ref{theorem-smooth transonic-full EP} in this paper, and \cite[Theorem 4.1]{BCF}, we are given with two types of sonic interfaces: (i) a regular interface which is succeedingly given as a level set of a function defined in terms of a classical solution, and (ii) a weak discontinuity which is determined by only one side of a flow. This leads to the following question:
    \begin{center}
    \emph{What is a general criterion to determine the type of a sonic interface?
    }\end{center}
    A general classification of sonic interfaces should be a fascinating subject to be investigated in the future.

\begin{remark}
\label{remark about pt main theorem-3}
In Theorem \ref{theorem-smooth transonic-full EP}, we require for the background momentum density $J(=\bar{\rho}\bar u_1)$ to be either small or large, and for the background Mach number $\bar{M}$ to be close to 1 (see \eqref{almost sonic condition1 full EP}) to establish the well-posedness of Problem \ref{problem-full EP}. As we shall see later in Section \ref{section:essential part}, the biggest challenge is to establish a priori $H^1$-estimate for a solution  to a linear system consisting of {\emph{a degenerately elliptic-hyperbolic mixed type PDE of second order}} and a second order elliptic equation, which are weakly coupled together via lower order derivative terms. To the best of our knowledge, this is the first paper to deal with a boundary value problem of such a PDE system.
\end{remark}

Suppose that $(\rho, {\bf u}, S, \Phi)\in H^3(\Om_L)\times H^3(\Om_L;\R^2)\times H^4(\Om_L)\times H^4(\Om_L)$ is a solution to Problem \ref{problem-full EP}. Since ${\bf u}$ is in $C^1(\ol{\Om_L})$, there exists a unique $C^2$-function $\phi:\ol{\Om_L}\rightarrow \R$ that solves the linear boundary value problem:
\begin{equation}
\label{bvp for phi}
\begin{cases}
  -\Delta\phi=\nabla\times{\bf u}\quad&\mbox{in $\Om_L$},\\
  \der_{x_1}\phi=0\quad&\mbox{on $\Gam_0$},\\
  \phi=0\quad&\mbox{on $\Gam_w\cup\Gam_L$}.
  \end{cases}
\end{equation}

Setting as
\begin{equation*}
\nabla^{\perp}:={\bf e}_{1}\der_{x_2}-{\bf e}_{2}\der_{x_1},
\end{equation*}
one has $\nabla\times (\nabla^{\perp}\phi)=-\Delta \phi$, and this yields that $\nabla\times({\bf u}-\nabla^{\perp}\phi)=0$ in $\Om_L$. So there exists a function $\vphi:\ol{\Om_L}\rightarrow \R$ such that
\begin{equation}
\label{H-decomposition}
  {\bf u}=\nabla\vphi+\nabla^{\perp}\phi\quad\tx{in $\ol{\Om_L}$}.
\end{equation}
This yields a {\emph{Helmholtz decomposition}} of a two-dimensional vector field ${\bf u}$. The representation \eqref{H-decomposition} indicates that $\vphi$ and $\phi$ are concerned with the compressibility and the vorticity of a flow, respectively, in the sense that $\nabla\cdot {\bf u}=\Delta \vphi$ and $\nabla\times {\bf u}=-\Delta \phi$. In \cite{BDXX}--\cite{BDX}, it is shown that if ${\bf u}\cdot{\bf e}_1=\vphi_{x_1}+\phi_{x_2}\neq 0$, then \eqref{full EP rewritten} can be rewritten in
terms of $(\vphi, \Phi, \phi, S)$ as follows:
\begin{align}
\label{new system with H-decomp2}
&
{\rm div}\left(\varrho(S, \Phi, \nabla\vphi, \nabla^{\perp}\phi)(\nabla\vphi+\nabla^{\perp}\phi)\right)=0,\\
\label{new system with H-decomp3}
&\Delta\Phi=\varrho(S, \Phi,\nabla\vphi, \nabla^{\perp}\phi)-\barrhoi,\\
\label{new system with H-decomp4}
&\Delta\phi=-\frac{ \varrho^{\gam-1}(S, \Phi,\nabla\vphi, \nabla^{\perp}\phi)S_{x_2}}{(\gam-1)(\vphi_{x_1}+\phi_{x_2})},\\
\label{new system with H-decomp5}
&\varrho(S, \Phi,\nabla\vphi, \nabla^{\perp}\phi) (\nabla\vphi+\nabla^{\perp}\phi)\cdot \nabla S=0
\end{align}
for
\begin{equation}
  \label{new system with H-decomp1}
\varrho(S, \Phi,\nabla\vphi, \nabla^{\perp}\phi):=
\left(\frac{\gam-1}{\gam S}\left(\Phi-\frac 12|\nabla\vphi+\nabla^{\perp}\phi|^2\right)\right)^{\frac{1}{\gam-1}}.
\end{equation}

So we can restate Problem \ref{problem-full EP} in terms of $(\vphi, \Phi, \phi, S)$ as follows:
\begin{problem}
\label{problem-HD} Given constants $\gam>1$, $\zeta_0>1$, $J>0$, $S_0>0$, and $(u_0, E_0)\in \Tac$ with $E_0<0$, fix a constant $L\in(0, l_{\rm max})$. Suppose that functions $S_{\rm en}:[-1,1]\rightarrow \R$, $E_{\rm en}:[-1,1]\rightarrow \R$ and $w_{\rm en}:[-1,1]\rightarrow \R$ are fixed so that $(S_{\rm en}, E_{\rm en}, w_{\rm en})$ is sufficiently close to $(S_0, E_0, 0)$ in the sense that $\mcl{P}(S_{\rm en}, E_{\rm en}, w_{\rm en})$, given by \eqref{definition-perturbation of bd}, is small. Find a solution $(\vphi, \phi, \Phi, S)$ to the nonlinear system \eqref{new system with H-decomp2}--\eqref{new system with H-decomp5} in $\Om_L$ with satisfying the properties
\begin{equation*}
 (\vphi_{x_1}+\phi_{x_2})>0\quad\tx{and}\quad  \varrho(S, \Phi, \nabla\vphi, \nabla^{\perp}\phi)>0 \quad\tx{in $\ol{\Om_L}$},
\end{equation*}
and the boundary conditions:
  \begin{equation}
\label{BC for new system with H-decomp}
  \begin{split}
  \der_{x_2}\vphi=w_{\rm en},\quad  \der_{x_1}\phi=0,\quad S=S_{\rm en},\quad
  \der_{x_1}\Phi=E_{\rm en}\quad&\tx{on $\Gamen$},\\
  \der_{x_2}\vphi=0,\quad \phi=0,\quad \der_{x_2}\Phi=0\quad&\mbox{on $\Gamw$},\\
  \phi=0,\quad \Phi=\bar{\Phi}\quad&\mbox{on $\Gamex$}.
  \end{split}
\end{equation}

\end{problem}

\begin{definition}\label{definition of background solution-HD}
For the background solution $(\bar{\rho}, \bar{\bf u}, \bar{\Phi})$ associated with $(\gam, \zeta_0, J, S_0, E_0)$ in the sense of Definition \ref{definition of background solution-full EP}, define a function $\bar{\vphi}:[0, l_{\rm max}]\times [-1,1]\rightarrow \R$ by
\begin{equation*}
  \bar{\vphi}(x_1, x_2):=\int_{0}^{x_1}\bar u_1(t)\,dt.
\end{equation*}
Then, $(\vphi, \phi, \Phi, S)=(\bar{\vphi}, 0, \bar{\Phi}, S_0)$ solves Problem \ref{problem-HD} for $(S_{\rm en}, E_{\rm en}, w_{\rm en})=(S_0, E_0, 0)$. Later on, we call $(\bar{\vphi}, 0, \bar{\Phi}, S_0)$ the {\emph{background solution}} to the system \eqref{new system with H-decomp2}--\eqref{new system with H-decomp5} associated with $(\gam, \zeta_0, J, S_0, E_0)$.

\end{definition}

\begin{notation} Let $\|\cdot\|_X$ be a norm in a linear space $X$. Given a vector field ${\bf V}=(V_1,\cdots, V_n)$ with $V_j\in X$ for all $j=1,\cdots, n$, we define $\|{\bf V}\|_X$ by
\begin{equation*}
  \|{\bf V}\|_X:=\sum_{j=1}^n \|V_j\|_{X}.
\end{equation*}
\end{notation}

\begin{theorem}
\label{theorem-HD}
Given constants $(\gam, \zeta_0, J, S_0, E_0)$ satisfying
\begin{equation*}
\gam>1,\quad \zeta_0>1,\quad S_0>0,\quad E_0<0,
\end{equation*}
suppose that $(u_0, E_0)\in \Tac$, which is equivalent to $0<u_0<\us$.
Let $(\bar{\vphi}, 0, \bar{\Phi}, S_0)$ be the background solution to the system \eqref{new system with H-decomp2}--\eqref{new system with H-decomp5} associated with $(\gam, \zeta_0, J, S_0, E_0)$. In addition, suppose that the boundary data $(S_{\rm en}, E_{\rm en}, w_{\rm en})$ satisfy {\emph{Condition \ref{conditon:1}}}.
Then one can fix two constants $\bJ$ and $\ubJ$ depending only on $(\gam, \zeta_0, S_0)$ with satisfying that
\begin{equation*}
  0<\bJ<1<\ubJ<\infty
\end{equation*}
so that whenever the background momentum density $J(=\bar{\rho}\bar u_1)$ satisfies
\begin{center}
either $0<J\le \bJ$ or $\ubJ\le J<\infty$,
\end{center}
there exists a constant $d\in(0,1)$ depending on $(\gam, \zeta_0, S_0,J)$ so that if
$(E_0, L)$ are fixed to satisfy the condition \eqref{almost sonic condition1 full EP}, then $(\bar{\vphi}, 0, \bar{\Phi}, S_0)$ is structurally stable in the following sense: one can fix a constant $\bar{\sigma}>0$ sufficiently small depending only on $(\gam, \zeta_0, S_0, E_0, J, L)$ so that if the inequality \eqref{smallness of bd} holds, then Problem \ref{problem-HD} has a unique solution $(\vphi, \phi, \Phi, S )$ that satisfies the estimate
    \begin{equation}
    \label{solution estimate HD}
    \begin{split}
    \|(\vphi, \phi, \Phi, S)-(\bar{\vphi},0, \bar{\Phi}, S_0)\|_{H^4(\Om_L)}
    \le
    C\mcl P(S_{\rm en}, E_{\rm en}, w_{\rm en})
    \end{split}
    \end{equation}
    for some constant $C>0$ depending on $(\gam, \zeta_0, S_0, E_0, J, L)$.
Furthermore, there exists a function $\fsonic:[-1,1]\rightarrow (0, L)$ satisfying the property \eqref{sonic boundary is a graph pt} with ${\bf u}=\nabla\vphi+\nabla^{\perp}\phi$ and $\rho=\varrho(S, \Phi, \nabla\vphi, \nabla\phi)$. Furthermore, the function $\fsonic$ satisfies the estimate \eqref{estimate of sonic boundary pt}
    for some constant $C>0$ depending on $(\gam, \zeta_0, S_0, E_0, J, L)$.

\end{theorem}

Similarly to the works in \cite{BDXX, BDX3}, Theorem \ref{theorem-smooth transonic-full EP} can easily follow from Theorem \ref{theorem-HD}.
\begin{proof}
[Proof of Theorem \ref{theorem-smooth transonic-full EP}]
If $(\vphi, \phi, \Phi, S)$ solves Problem \ref{problem-HD} and satisfies the estimate \eqref{solution estimate HD}, then $(\rho, {\bf u}, S, \Phi)$ with $\rho=\varrho(S, \Phi, \nabla\vphi, \nabla^{\perp}\phi)$ and ${\bf u}=\nabla\vphi+\nabla^{\perp}\phi$ solves Problem \ref{problem-full EP} and satisfies the estimate \eqref{solution estimate full EP}.
\medskip

Suppose that $(\rho, {\bf u}, \Phi)$ is a solution to Problem \ref{problem-full EP}, and that it satisfies the estimate \eqref{solution estimate full EP}. First of all, we obtain $\phi$ as a solution to the boundary value problem:
\begin{equation*}
  \begin{cases}
  -\Delta \phi=\nabla\times {\bf u}\quad&\mbox{in $\Om_L$}\\
  \der_{x_1}\phi=0\quad&\mbox{on $\Gamen$}\\
  \phi=0\quad&\mbox{on $\Gam_w\cup\Gamex$}.
  \end{cases}
\end{equation*}
Next, we find $\vphi$ to satisfy
\begin{equation}
\label{recovery of vphi}
\begin{cases}
\nabla\vphi={\bf u}-\nabla^{\perp}\phi\quad&\mbox{in $\Om_L$}\\
\der_{x_2}\vphi=w_{\rm en}\quad&\mbox{on $\Gamen$}.
\end{cases}
\end{equation}
We define
\begin{equation*}
  \vphi(x_1, x_2):=\int_{0}^{x_1}({\bf u}\cdot {\bf e}_1-\der_{x_2}\phi)(t,x_2)\,dt+g(x_2)\quad\tx{in $\Om_L$}
\end{equation*}
for a function $g:[-1,1]\rightarrow \R$ to be determined. By a straightforward computation with using the equation $-\Delta\phi=\nabla \times {\bf u}$, we get
\begin{equation*}
  \der_{x_2}\vphi=(\der_{x_1}\phi+u_2)(x_1, x_2)-u_2(0, x_2)+g'(x_2).
\end{equation*}
The function $\vphi$ satisfies \eqref{recovery of vphi} if we choose the function $g$ to be
\begin{equation*}
  g(x_2)=\int_0^{x_2} w_{\rm en}(s)\,ds.
\end{equation*}
Clearly, ${\bf V}:=(\vphi, \phi, \Phi, S)$ solves Problem \ref{problem-HD}. Furthermore, it follows from the estimate \eqref{solution estimate full EP} given in Theorem \ref{theorem-smooth transonic-full EP} that ${\bf V}$ satisfies
\begin{equation}
\label{estimate from full to HD}
  \|(\vphi, \phi, \Phi, S)-(\bar{\vphi},0, \bar{\Phi}, S_0)\|_{H^4(\Om_L)}\le C_*\mcl{P}(S_{\rm en}, E_{\rm en}, \om_{\rm en})
\end{equation}
for some constant $C_*>0$ depending on $(\gam, \zeta_0, S_0, E_0, J, L)$.

 Suppose that Problem \ref{problem-full EP} has two solutions ${\bf U}^{(i)}=(\rho^{(i)}, {\bf u}^{(i)}, S^{(i)}, \Phi^{(i)})\,(i=1,2)$ that satisfy the estimate \eqref{solution estimate full EP}. For each $j=1,2$, let ${\bf{V}}^{(j)}=(\vphi^{(j)}, \phi^{(j)}, \Phi^{(j)}, S^{(j)})$ be given by the procedure described in the above. Then it follows from Theorem \ref{theorem-HD} that ${\bf V}^{(1)}={\bf V}^{(2)}$ in $\Om_L$, which yields that ${\bf U}^{(1)}={\bf U}^{(2)}$ in $\Om_L$.
\end{proof}

The rest of the paper is devoted to prove Theorem \ref{theorem-HD}.

\section{Preparations to prove Theorem \ref{theorem-HD}}

Fix constants $\gam>1$, $\zeta_0>1$, $S_0>0$ and $E_0<0$. Let $u_0\in(0, \us)$ be fixed  such that $(u_0, E_0)\in \Tac$, and fix a constant $L\in(0, l_{\rm max})$. Let $(\bar{\vphi}, 0, {\bar{\Phi}}, S_0)$ be the background solution to the system \eqref{new system with H-decomp2}--\eqref{new system with H-decomp5} associated with $(\gam, \zeta_0, J, S_0, E_0)$ in the sense of Definition \ref{definition of background solution-HD}.

\subsection{Iteration sets and linear boundary value problems}
\label{subsection:framework}
\begin{definition}
\label{definition:coefficints-iter}

(1) For $z\in \R$ and ${\bf p}, {\bf q}\in \R^2$, denote
\begin{equation}
\label{definition of v}
  {\bf v}(\rx, {\bf p}, {\bf q} ):=\nabla\bar{\vphi}(\rx)+{\bf p}+{\bf q}.\end{equation}
For ${\bf v}={\bf v}(\rx, {\bf p}, {\bf q})$, define
\begin{equation}
\label{definition of pre coeff}
  A_{ij}(\rx, z, {\bf p}, {\bf q} ):=(\gam-1)
      \left(\bar{\Phi}(\rx)+z-\frac 12|{\bf v}|^2\right)\delta_{ij}-
       {\bf v}\cdot {\bf e}_i {\bf v}\cdot {\bf e}_j.
\end{equation}
For $s\in\R$ satisfying $|s|\le \frac{S_0}{2}$, define
\begin{equation}
        \label{definition of rho tilde}
          \til{\rho}(\rx, s, z, {\bf p}, {\bf q}):=\left(\frac{\gam-1}{\gam(S_0+s)}\left(\bar{\Phi}(\rx)+z-\frac 12|{\bf v(\rx, {\bf p}, {\bf q})}|^2\right)\right)^{\frac{1}{\gam-1}}.
        \end{equation}
\quad\\
(2) For $\rx=(x_1, x_2)\in \ol{\Om_L}$, ${\bf p}=(p_1,p_2),\, {\bf q}=(q_1,q_2)$, let us regard ${\bf v}={\bf v}(\rx, {\bf p}, {\bf q})$ as a $2 \times 1$ matrix. For ${\bf r}\in \R^2$ and $\mathbb{M}=(m_{ij})_{i,j=1}^2\in \R^{2\times 2}$, define
\begin{equation}
\begin{split}
      Q_1(\rx, {\bf p}, {\bf q},  {\bf r})&:=\frac{\gam+1}{2}\bar u_1'(x_1)\left(p_1+q_1\right)^2
-{\bf r}\cdot({\bf p}+{\bf q}),\\
      R_1(\rx, {\bf p}, {\bf q}, \mathbb{M})&:={\bf v}^{T}{\mathbb{M}}{\bf v}-\left(\bar E(x_1)-(\gam+1)\bar u_1(x_1)\bar u_1'(x_1)\right)q_1.
      \end{split}
\end{equation}

\end{definition}

Let us set
\begin{equation*}
  \umax:=\bar{u}_1(\lmax).
\end{equation*}
By Lemma \ref{lemma-1d-full EP}, it holds that $\umax>\bar u_1(x_1)$ for $x_1\in[0, \lmax)$. Since $(\umax,0)\in \Tac$, it holds that $H(\umax)=0$  thus we have $u_0<\umax<\infty$, for the function $H$ defined by \eqref{definition of H}. This implies that $ 0<\bar u_1(x_1)<\infty$ for all $x_1\in [0,L]$. Therefore, we have
\begin{equation*}
A_{22}(\rx, 0, {\bf 0}, {\bf 0})=(\gam-1)\left(\bar{\Phi}-\frac 12|\nabla\bar{\vphi}|^2\right)=\frac{\gam S_0J^{\gam-1}}{\bar u_1^{\gam-1}(x_1)}>0\quad\tx{in $\ol{\Om_L}$.}
\end{equation*}
Since $\bar u_1$ is smooth on $[0,L]$, and $A_{22}(\rx, z, {\bf p}, {\bf q})$ is smooth for $(\rx,z, {\bf p}, {\bf q})\in \ol{\Om_L}\times \R\times \R^2\times \R^2$, one can fix a small constant $d_0>0$ depending only on $(\gam, \zeta_0, J, S_0, E_0)$ such that if
\begin{equation}
\label{small zpq condition}
  \max\{|z|, |{\bf p}|, |{\bf q}|\}\le d_0\quad\tx{with $|{\bf p}|:=\sqrt{p_1^2+p_2^2}$,\quad $|{\bf q}|:=\sqrt{q_1^2+q_2^2}$},
\end{equation}
then it holds that
\begin{equation}
\label{strict positivity of A22}
  A_{22}(\rx, z, {\bf p}, {\bf q})\ge \frac{\gam S_0J^{\gam-1}}{2\umax^{\gam-1}}\quad\tx{in $\ol{\Om_L}$}.
\end{equation}

\begin{definition} Let us set
\begin{equation*}
  \mcl{R}_{d_0}:=\{(z, {\bf p}, {\bf q})\in \R\times \R^2\times \R^2: {\tx{the condition \eqref{small zpq condition} holds.}}\}.
\end{equation*}
Let $(\rx, z, {\bf p}, {\bf q})\in \ol{\Om_L}\times \mcl{R}_{d_0}$ with $\rx=(x_1,x_2)$. We define coefficients $\til a_{ij}(i,j=1,2)$, $\til a$, $\til b_1$, and $\til b_0$ by
    \begin{equation}
\label{coefficients of L_1}
\begin{split}
&\til a_{ij}(\rx, z,{\bf p}, {\bf q}):=\frac{A_{ij}(\rx, z,{\bf p}, {\bf q})}{A_{22}(\rx, z,{\bf p}, {\bf q})},\\
&\til a(\rx, z,{\bf p}, {\bf q}):=\frac{\bar E(x_1)-(\gam+1)\bar u_1'(x_1)\bar u_1(x_1)}{A_{22}(\rx, z,{\bf p}, {\bf q})},\\
&\til b_1(\rx, z,{\bf p}, {\bf q}):=\frac{\bar u_1(x_1)}{A_{22}(\rx, z,{\bf p}, {\bf q})},\quad \til b_0(\rx, z,{\bf p}, {\bf q}):=\frac{(\gam-1)\bar u_1'(x_1)}{A_{22}(\rx, z,{\bf p}, {\bf q})}.
\end{split}
\end{equation}
\end{definition}

\medskip

Suppose that $(\vphi, \phi, \Phi, S)$ solves Problem \ref{problem-HD}.
And, let us set $(\psi, \Psi, T)$ as
\begin{equation*}
  (\psi, \Psi, T):=(\vphi-\bar{\vphi}, \Phi-\bar{\Phi}, S-S_0).
\end{equation*}
\begin{definition}
\label{definition:coefficients-nonlinear}
(1) Suppose that $(\psi, \Psi, \phi)$ satisfies
\begin{equation}
\label{smallness condition-1}
\max_{\ol{\Om_L}} \{|\Psi|, |D\psi|, |D\phi|\}\le d_0,
\end{equation}
and define
\begin{equation*}
  \begin{split}
  a_{ij}^{(\psi, \phi, \Psi)}(\rx)&:=\til a_{ij}(\rx,\Psi(\rx), \nabla\psi(\rx), \nabla^{\perp}\phi(\rx)),\\
   a^{(\psi, \phi, \Psi)}(\rx)&:=\til a(\rx,\Psi(\rx), \nabla\psi(\rx), \nabla^{\perp}\phi(\rx)),\\
   b_1^{(\psi, \phi, \Psi)}(\rx)&:=\til b_1(\rx,\Psi(\rx), \nabla\psi(\rx), \nabla^{\perp}\phi(\rx)),\\
  b_0^{(\psi, \phi, \Psi)}(\rx)&:=\til b_0(\rx,\Psi(\rx), \nabla\psi(\rx), \nabla^{\perp}\phi(\rx)),\\
 {f}_1^{(\psi, \phi, \Psi)}(\rx)&:=\frac{Q_1(\rx, \nabla\psi(\rx), \nabla^{\perp}\phi(\rx), \nabla\Psi(\rx))+R_1(\rx, \nabla\psi(\rx), \nabla^{\perp}\phi(\rx), D(\nabla^{\perp}\phi)(\rx))}{ A_{22}(\rx, \Psi(\rx), \nabla\psi(\rx), \nabla^{\perp}\phi(\rx))},
  \end{split}
\end{equation*}
for
\begin{equation*}
D(\nabla^{\perp}\phi):=\begin{pmatrix}\der_1\\ \der_2\end{pmatrix}\begin{pmatrix}\der_2&-\der_1\end{pmatrix}\phi=
\begin{pmatrix}\der_{12}\phi&-\der_{11}\phi\\
\der_{22}\phi &-\der_{21}\phi\end{pmatrix}.
\end{equation*}
\quad\\
(2) In addition to \eqref{smallness condition-1}, suppose that $T$ satisfies
\begin{equation}
\label{smallness condition-2}
  \max_{\Om_L}|T|\le \frac{S_0}{2}.
\end{equation}
Let us define
\begin{equation*}
\begin{split}
  &c_0(\rx):=\der_z\til{\rho}(\rx, 0,0, {\bf 0}, {\bf 0}),\quad c_1(\rx):=\der_{\bf p}\til{\rho}(\rx, 0,0, {\bf 0}, {\bf 0})\cdot{\bf e}_1,\\
 &{f}_2^{(T,\psi, \phi, \Psi)}(\rx):= \til{\rho}(\rx, T(\rx), \Psi(\rx), \nabla\psi(\rx), \nabla^{\perp}\phi(\rx))-\til{\rho}(\rx,0, 0, {\bf 0}, {\bf 0})-\mfrak c_0\Psi(\rx)
  -\mfrak c_1 \der_1\psi(\rx).
 \end{split}
\end{equation*}
A direct computation yields
\begin{equation*}
c_0(\rx)=\frac{1}{\gam S_0\bar{\rho}^{\gam-2}(\rx)},\quad
c_1(\rx)=\frac{-\bar u_1(\rx)}{\gam S_0\bar{\rho}^{\gam-2}(\rx)}.
\end{equation*}

\quad\\
(3) Under the condition
\begin{equation}
\label{smallness condition-3}
 \min_{\ol{\Om_L}} {\bf v}(\rx, \nabla\psi(\rx), \nabla^{\perp}\phi(\rx))\cdot {\bf e}_1\ge \frac{u_0}{2},
\end{equation}
define
\begin{equation*}
      {f}_3^{(T,\psi, \phi, \Psi)}(\rx):=\frac{\left(\bar{\Phi}+\Psi(\rx)-\frac 12|{\bf v}(\rx, \nabla\psi(\rx), \nabla^{\perp}\phi(\rx))|^2\right)\der_2T(\rx)}{\gam(S_0+T(\rx)){\bf v}(\rx, \nabla\psi, \nabla^{\perp}\phi)\cdot {\bf e}_1}.
    \end{equation*}

\quad\\
(4) Finally, let us define a pseudo momentum density field ${\bf m}^{(\psi, \phi, \Psi)}$ by
    \begin{equation}
    \label{definition: vector field m}
      {\bf m}^{(\psi, \phi, \Psi)}(\rx):=\left(\bar{\Phi}(\rx)+\Psi(\rx)-\frac 12|{\bf v(\rx, \nabla\psi(\rx), \nabla^{\perp}\phi(\rx))}|^2\right)^{\frac{1}{\gam-1}}\bf v(\rx, \nabla\psi(\rx), \nabla^{\perp}\phi(\rx)).
    \end{equation}
\end{definition}

If all the conditions of \eqref{smallness condition-1}--\eqref{smallness condition-3} are satisfied, the nonlinear boundary value problem consisting of \eqref{new system with H-decomp2}--\eqref{new system with H-decomp5} and the boundary conditions stated in \eqref{BC for new system with H-decomp} is equivalent to the following problem for $(\psi, \phi, \Psi, T)$:
\begin{itemize}
\item[(i)] {\emph{Equations for $(\psi, \phi, \Psi, T)$ in $\Om_L$}}
\begin{align}
\label{equation for psi}
 \sum_{i,j=1}^2 a_{ij}^{(\psi, \phi, \Psi)}\der_{ij}\psi+  a^{(\psi, \phi, \Psi)}\der_1\psi+
    b_1^{(\psi, \phi, \Psi)}\der_1\Psi
    +b_0^{(\psi, \phi, \Psi)}\Psi={f}_1^{(\psi, \phi, \Psi)},&\\
\label{equation for Psi}
\Delta \Psi-c_0\Psi-c_1\der_1\psi={f}_2^{(T,\psi, \phi, \Psi)},&\\
\label{equation for phi}
  -\Delta \phi= {f}_3^{(T,\psi, \phi, \Psi)},&\\
    \label{equation for T}
 {\bf m}^{(\psi, \phi, \Psi)}\cdot \nabla T=0.
    \end{align}

\item[(ii)]{\emph{Boundary conditions for $\psi$}}
\begin{equation}\label{BC for psi}
  \psi(0, x_2)=\int_{-1}^{x_2}w_{\rm en}(t)\,dt\quad\tx{on $\Gamen$},\quad \der_2 \psi=0\quad\tx{on $\Gamw$}.
\end{equation}

\item[(iii)]{\emph{Boundary conditions for $\phi$}}
\begin{equation}\label{BC for phi}
  \der_1\phi=0\quad\tx{on $\Gamen$},\quad \phi=0\quad\tx{on $\der\Om_L\setminus \Gamen$}.
\end{equation}

\item[(iv)]{\emph{Boundary conditions for $\Psi$}}
\begin{equation}\label{BC for Psi}
  \der_1\Psi=E_{\rm en}-E_0\quad\tx{on $\Gamen$},\quad \der_2\Psi=0\quad\tx{on $\Gamw$},\quad \Psi=0\quad\tx{on $\Gamex$}.
\end{equation}

\item[(v)]{\emph{Boundary condition for $T$}}
\begin{equation}\label{BC for T}
T=S_{\rm en}-S_0\quad\tx{on $\Gamen$}.
\end{equation}

\end{itemize}

If the boundary value problem \eqref{equation for psi}--\eqref{BC for T} has a solution $(\psi, \phi, \Psi, T)$ that has a sufficiently small $C^2$-norm in $\ol{\Om_L}$, then $(\vphi, \phi, \Phi, S):=(\bar{\vphi}, 0, \bar{\Phi}, S_0)+(\psi, \phi, \Psi, T)$ satisfies all the conditions \eqref{smallness condition-1}--\eqref{smallness condition-3}, thus it becomes a solution of Problem \ref{problem-HD}. Therefore, we shall prove Theorem \ref{theorem-HD} by solving \eqref{equation for psi}--\eqref{BC for T}. This is a strategy used in the works \cite{BDXX, BDX3, BDX} and many others. The new and most important feature of this paper is that we seek for a solution $(\psi, \phi, \Psi, T)$ to \eqref{equation for psi}--\eqref{equation for T}, in which the equation \eqref{equation for psi}, as a second order differential equation for $\psi$, changes its type from being elliptic to being hyperbolic continuously with the degeneracy occurring on an interface of codimension 1, that is called {\emph{a sonic interface}}. We shall address more analytic properties of the equation \eqref{equation for psi} in Lemma \ref{lemma on L_1} after the iteration sets and related linear boundary value problems are defined.
\medskip

\begin{definition}
\label{definition: iteration sets}
For sufficiently small positive constants $r_1$, $r_2$, and $r_3$ to be fixed later, we define sets $\iterT(r_1)$, $\iterV(r_2)$ and $\iterP(r_3)$ as follows:
\begin{equation*}
\begin{split}
&\iterT{(r_1)}:=\left\{\til{T}\in H^4(\Om_L)\;\middle\vert\;
\begin{split}&\|\til{T}\|_{H^4(\Om_L)}\le r_1,\\
&\der_2^k\til{T}=0\quad\tx{on $\Gamw$ for $k=1,3$}\end{split}\right\},\\
&\iterV{(r_2)}:=\left\{\til{\phi}\in H^5(\Om_L)\;\middle\vert\;
\begin{split}&\|\til{\phi}\|_{H^5(\Om_L)}\le r_2,\\
& \der_2^{k-1}\til{\phi}=0\quad\tx{on $\Gamw$ for $k=1,3,5$}\end{split}\right\},\\
&\iterP{(r_3)}:=\left\{(\tpsi, \tPsi)\in H^4(\Om_L)\times H^4(\Om_L)\;\middle\vert\;
\begin{split}&\|\tpsi\|_{H^4(\Om_L)}+ \|\tPsi\|_{H^4(\Om_L)}\le r_3,\\
&\der_2^{k}\tpsi=\der_2^k\tPsi=0\quad\tx{on $\Gamw$ for $k=1,3$}\end{split}\right\}.
\end{split}
\end{equation*}
The sets $\iterT(r_1)$ and $\iterV(r_2)$ represent approximated entropy perturbations, and approximated vorticities, respectively. And, the set $\iterP(r_3)$ is a collection of approximated potential perturbations for electric force and velocity.

\end{definition}

\begin{definition}
\label{definition:approx coeff and fs}
Fix $\til{T}\in \iterT{(r_1)}$ and $P=(\tphi, \tpsi, \tPsi)\in \iterV{(r_2)}\times \iterP{(r_3)}$.
\quad\\
(1) Define a bilinear differential operator $\mfrak{L}_1^P(\cdot,\cdot)$ associated with $P$ by
\begin{equation*}
  \mfrak{L}_1^P(\psi, \Psi):=\sum_{i,j=1}^2 a_{ij}^P\der_{ij}\psi+a^P\der_1\psi+b_1^P\der_1\Psi+b_0^P\Psi
\end{equation*}
for the coefficients $a_{ij}^P$, $a^P$, $b_1^P$ and $b_0^P$ given by Definition \ref{definition:coefficients-nonlinear}.
\quad\\
(2) For the coefficients ${c}_0$ and ${c}_1$ given in Definition \ref{definition:coefficients-nonlinear}, define another bilinear differential operator $\mfrak{L}_2(\cdot,\cdot)$ by
\begin{equation*}
      \mfrak{L}_2(\psi, \Psi):=\Delta \Psi-{c}_0\Psi-{c}_1\der_1\psi.
    \end{equation*}
    Note that $\mfrak{L}_2$ is fixed independently of $\til{T}\in \iterT_{r_1}$ and $P=(\tphi, \tpsi, \tPsi)\in \iterV{(r_2)}\times \iterP{(r_3)}$.

\end{definition}

\begin{problem}
\label{LBVP1 for iteration}
Fix $\tilT\in \iterT(r_1)$ and $P\in \iterV(r_2)\times \iterP(r_3)$.

Find $\phi\in H^5(\Om_L)$ that solves
\begin{equation}
\label{lbvp for phi}
  \begin{cases}
  -\Delta\phi=f_0^{(\tilT, P)}\quad&\tx{in $\Om_L$},\\
  \der_1\phi=0\quad&\tx{on $\Gamen$},\\
  \phi=0\quad&\tx{on $\Gamw\cup\Gamex$}.
  \end{cases}
\end{equation}

And, find $(\psi, \Psi)\in [H^4(\Om_L)]^2$ that solves
    \begin{equation}
    \label{lbvp main}
      \begin{split}
      \begin{cases}
      \mfrak{L}_1^P(\psi, \Psi)=f_1^P\quad&\tx{in $\Om_L$},\\
      \mfrak{L}_2(\psi, \Psi)=f_2^{(\tilT, P)}\quad&\tx{in $\Om_L$},
      \end{cases}\\
      \psi(0, x_2)=\int_{-1}^{x_2}w_{\rm en}(t)\,dt,\quad \der_1\Psi=E_{\rm en}-E_0\quad&\tx{on $\Gamen$},\\
      \der_2\psi=0,\quad \der_2\Psi=0\quad&\tx{on $\Gamw$},\\
      \Psi=0\quad&\tx{on $\Gamex$}.
      \end{split}
    \end{equation}

\end{problem}
Later, we shall fix constants $r_1$, $r_2$ and $r_3$ so that Problem \ref{LBVP1 for iteration} is well-posed for each $\tilT\in \iterT(r_1)$ and $P\in \iterV(r_2)\times \iterP(r_3)$.
\medskip

    In order to solve a single degenerately mixed type PDE, one can employ an idea from \cite[Chapter 1]{KZ}. But we should point out that the method developed in \cite{KZ} is applicable only if several technical conditions hold. To our surprise, all the technical conditions described in \cite[Chapter 1]{KZ} are satisfied by the equation \eqref{equation for psi} or its variations, which are derived from \eqref{new system with H-decomp2}(see Lemmas \ref{lemma-coefficient at bg} and \ref{corollary-coeff extension}), when the background flow is accelerating (see Lemma \ref{lemma-1d-full EP}). But, there are more difficulties to overcome. The linear system in \eqref{lbvp main} consists of two PDEs of different types, and they are weakly coupled by lower order derivative terms. So, in order to establish an a priori $H^1$-estimate of weak solutions to the system \eqref{lbvp main}, it was inevitable for us to add an assumption on the background solutions. Clearly, the condition \eqref{almost sonic condition1 full EP} holds if $L$ is sufficiently small. But the smallness of $L$ is not a necessary condition to guarantee the almost sonic condition. In Appendix \ref{appendix:nozzle length}, we show examples of the parameters $(\gam, J)$ for which $L$ can be large.
\medskip

Given $\tilT\in \iterT(r_1)$ and $P\in \iterV(r_2)\times \iterP(r_3)$, let $(\phi, \psi, \Psi)$ be a solution to Problem \ref{LBVP1 for iteration}, and define a vector field ${\bf m}(\Psi, \nabla\psi, \nabla^{\perp}\phi)$ by \eqref{definition: vector field m}. It is clear that if $(\phi, \psi, \Phi)=P$ holds in $\Om_L$, then the vector field ${\bf m}$ satisfies the equation
\begin{equation}
\label{div-free}
  \nabla\cdot {\bf m}(\Psi, \nabla\psi, \nabla^{\perp}\phi)=0\quad\tx{in $\Om_L$}.
\end{equation}
\begin{definition}
\label{definition: momentum density field} Fix $\tilT \in \iterT(r_1)$. Assume that the constants $r_1$, $r_2$, and $r_3$ are appropriately fixed so that Problem \ref{LBVP1 for iteration} is well-defined, and that it acquires a unique solution $(\phi, \psi, \Phi)$ for each $\tilT\in \iterT(r_1)$ and $P\in \iterV(r_2)\times \iterP(r_3)$. In addition, suppose that there exists a unique element $P\in \iterV(r_2)\times \iterP(r_3)$ that satisfies the equation $P=(\phi, \psi, \Phi)$ in $\Om_L$. For such an element $P$,  we call the vector field ${\bf m}(\Psi, \nabla\psi, \nabla^{\perp}\phi)$ by {\emph{the approximated momentum density field associated with $\tilT$}}.
\end{definition}

\begin{problem}
\label{problem-transport equation}
For a fixed $\tilT\in \iterT(r_1)$, let ${\bf m}(\Psi, \nabla\psi, \nabla^{\perp}\phi)$ be the {\emph{approximated momentum density field}} associated with $\tilT$ in the sense of Definition \ref{definition: momentum density field}. Solve the following boundary value problem for $T$:
\begin{equation*}
  \begin{split}
  {\bf m}(\Psi, \nabla\psi, \nabla^{\perp}\phi)\cdot \nabla T=0\quad&\tx{in $\Om_L$},\\
  T=S_{\rm en}-S_0\quad&\tx{on $\Gamen$}.
  \end{split}
\end{equation*}
\end{problem}

\begin{remark}
\label{remark-iteration}
\begin{itemize}
\item[(i)] Clearly, Problem \ref{problem-transport equation} is well defined if we prove that Problem \ref{LBVP1 for iteration} is well-posed for each $\tilT\in \iterT(r_1)$ and $P\in \iterV(r_2)\times \iterP(r_3)$, and show that the approximated momentum density field is well defined for each $\tilT\in \iterT(r_1)$.

\item[(ii)] Suppose that $T$ is a solution to Problem \ref{problem-transport equation}, and that it satisfies $T=\tilT$ in $\Om_L$. Let us define $(\vphi, \Phi, S)$ by
    \begin{equation*}
      (\vphi, \Phi, S):=(\bar{\vphi}, \bar{\Phi}, S_0)+(\psi, \Psi, T)\quad\tx{in $\Om_L$}.
    \end{equation*}
    Then $(\vphi, \Phi, \phi, S)$ solves Problem \ref{problem-HD}.
\end{itemize}
\end{remark}

\subsection{Approximated sonic interfaces}
\label{subsection: main issues}
As mentioned earlier, the biggest challenge in proving Theorem \ref{theorem-HD} is to establish the well-posedness of the boundary value problem \eqref{lbvp main}. For further discussion, we first need to understand analytic properties of the differential operator $\mfrak{L}_1^{P}$.

\begin{lemma}
\label{lemma on L_1}
Given constants $(\gam, \zeta_0, J, S_0, E_0)$, let $(\bar{\vphi}, \bar{\Phi})$ be the associated background solution to the system \eqref{new system with H-decomp1}--\eqref{new system with H-decomp5} in the sense of Definition \ref{definition of background solution-HD}. For any $L\in (0, l_{\rm max}]$, one can fix a constant $\bar{\delta}>0$ sufficiently small depending only on $(\gam, \zeta_0, J, S_0, E_0, L)$ so that if we fix the constants $r_1$, $r_2$, and $r_3$ to satisfy the inequality
\begin{equation}
\label{condition:r}
\max\{r_1,r_2, r_3\}\le 2\bar{\delta},
\end{equation}
then, for each $\tilT\in \iterT(r_1)$  and $P=(\tphi, \tpsi, \tPsi)\in
\iterV(r_2)\times \iterP(r_3)$, the coefficients $(a_{ij}^P, a^P, b_1^P, b_0^P)$ with $i,j=1,2$, and the functions $(f_0^{(\tilT, P)}, f_1^P, f_2^{(\tilT, P)})$ are well defined by Definition \ref{definition:approx coeff and fs}. Furthermore, they satisfy the following properties:

\begin{itemize}
\item[(a)] $\displaystyle{a_{ij}^P, a^P, b_1^P, b_0^P, f_0^{(\tilT, P)}, f_1^P, f_2^{(\tilT, P)}\in H^3(\Om_L)}$.
    \medskip
\item[(b)] $\displaystyle{a_{22}^P= 1}$ and $\displaystyle{a_{12}^P=a_{21}^P}$ in $\Om_L$.
    \medskip

\item[(c)] $\displaystyle{\der_2 (a_{11}^P, a^P, b_1^P, b_0^P)=0}$ and $\displaystyle{\der_2^ka_{12}^P=0}$ for $k=0,2$ on $\Gamw$.
   \medskip

\item[(d)] For
$$P_0=(0,0,0)\in \iterV(r_2)\times \iterP(r_3),$$
let us set $(\bar a_{11}, \bar a, \bar b_1, \bar b_0)$ as
    \begin{equation}
    \label{coefficient at bg}
      (\bar a_{11}, \bar a, \bar b_1, \bar b_0):=(a_{11}^{P_0}, a^{P_0}, b_1^{P_0}, b_0^{P_0}).
    \end{equation}
    Then it holds that
    \begin{equation}
\label{estimate-coefficient-difference}
\begin{split}
  &\|(a^P_{11}, a^P, b^P_1, b^P_0)-(\bar a_{11}, \bar a, \bar b_1, \bar b_0)\|_{H^3(\Om_L)}+\|a^P_{12}\|_{H^3(\Om_L)}\\
  &\le C \|(\tphi, \tpsi, \tPsi)\|_{H^4(\Om_L)}
\end{split}
\end{equation}
for a constant $C>0$ depending only on $(\gam, \zeta_0, J, S_, E_0)$.
\medskip

\item[(e)] There exists a constant $C>0$ depending only on $(\gam, \zeta_0, J, S_0, E_0)$ to satisfy the following estimates:
    \begin{align*}
    &\|f_0^{(\tilT, P)}\|_{H^3(\Om_L)}\le C\|\tilT\|_{H^4(\Om_L)},\\
    &\|f_1^{P}\|_{H^3(\Om_L)}\le C\left(\left(1+\|\tPsi\|_{H^4(\Om_L)}\right)\|\tphi\|_{H^5(\Om_L)}+
        \|(\tpsi, \tPsi)\|^2_{H^4(\Om_L)}
        \right),\\
        &\|f_2^{(\tilT. P)}\|_{H^3(\Om_L)}\le C\left(\|\tilT\|_{H^3(\Om_L)}
        +\|\tphi\|_{H^4(\Om_L)}
        +\|(\tpsi, \tPsi)\|^2_{H^4(\Om_L)}
        \right).
    \end{align*}
\medskip

\item[(f)] The functions $f_0^{(\tilT, P)}$, $f_1^P$ and $f_2^{(\tilT, P)}$ satisfy the following compatibility conditions on $\Gamw$:
    \begin{equation*}
      \der_2^{k} f_0^{(\tilT, P)}=0\,\,\tx{for $k=0,2$}, \quad
       \der_2f_1^P=0,\quad
       \der_2f_2^{(\tilT, P)}=0.
    \end{equation*}
    \medskip
\item[(g)] For the differential operator $\mfrak{L}_1^P(\cdot,\cdot)$ given by Definition \ref{definition:approx coeff and fs}, it is a second order differential operator with respect to the first component of a variable (which corresponds to $\psi$ in Definition \ref{definition:approx coeff and fs}(1)).
    \begin{itemize}
\item[($\tx{g}_1$)]
   The operator $\mfrak{L}_1^{P_0}$, as a second order differential operator with respect to its first component of an argument, is of {\emph{mixed type in $\Om_L$ with a degeneracy occurring on $\displaystyle{\Om_L\cap\{x_1=\ls\}}$}} for the constant $\ls$ given from Lemma \ref{lemma-1d-full EP}. More precisely,
     \begin{equation*}
       \tx{the operator $\mfrak L_1^{P_0}$ is}\,\,\begin{cases}
       \tx{elliptic}\quad&\mbox{for $x_1<\ls$},\\
       \tx{degenerate}\quad&\mbox{for $x_1=\ls$},\\
       \tx{hyperbolic}\quad&\mbox{for $x_1>\ls$}.
       \end{cases}
     \end{equation*}
\item[($\tx{g}_2$)] For each $P=(\tphi, \tpsi, \tPsi)\in
\iterV(r_2)\times \iterP(r_3)$, there exists a function  $\gs^P:[-1,1]\rightarrow (0,L)$ so that
\begin{equation*}
       \tx{the operator $\mfrak L_1^{P}$ is}\,\,\begin{cases}
       \tx{elliptic}\quad&\mbox{for $x_1<\gs^P(x_2)$},\\
       \tx{degenerate}\quad&\mbox{for $x_1=\gs^P(x_2)$},\\
       \tx{hyperbolic}\quad&\mbox{$x_1>\gs^P(x_2)$}.
       \end{cases}
     \end{equation*}
In addition, the function $\gs^P$ satisfies the estimate
    \begin{equation}
    \label{estimate of gs}
      \|\gs^P-\ls\|_{C^1([-1,1])}+\|\gs^P-\ls\|_{H^2((-1,1))}\le C  \|(\tphi, \tpsi, \tPsi)\|_{H^4(\Om_L)}
    \end{equation}
for a constant $C>0$ depending only on $(\gam, \zeta_0, J, S_, E_0)$. Furthermore, the function $\gs^P$ satisfies the estimate
\begin{equation}
\label{boundes of gs}
   \frac{15}{16} \ls\le \gs^P(x_2)\le \ls+\frac{1}{16}\min\{\ls, L-\ls\} \quad\tx{for $|x_2|\le 1$}.
\end{equation}

\item[($\tx{g}_3$)] For $\lambda_0:=\bar{a}_{11}\left(\frac{5\ls}{8}\right)(>0)$, it holds that
\begin{equation*}
  \begin{pmatrix}
  a_{11}^P & a_{12}^P\\
  a_{12}^P & 1
  \end{pmatrix}\ge \frac{\lambda_0}{2}\quad\tx{in $\ol{\Om_L}\cap\left\{x_1\le \frac{7\ls}{8}\right\}$}.
\end{equation*}

\item[($\tx{g}_4$)] There exists a constant $\lambda_1>0$ depending only on $(\gam, \zeta_0, J, S_, E_0)$ such that
\begin{equation*}
a_{11}^P\le -\lambda_1\quad\tx{in $\ol{\Om_L}\cap\left\{x_1\ge L-\frac{L-\ls}{10}\right\}$. }
\end{equation*}
\end{itemize}

\end{itemize}

\begin{proof} {\textbf{Step 1.}}
By the generalized Sobolev inequality, one can fix a constant $\delta_1>0$ sufficiently small such that if the inequality
\begin{equation*}
  \max\{r_2, r_3\}\le 2\delta_1
\end{equation*}
holds, then for any $P=(\tphi, \tpsi, \tPsi)\in \iterV(r_2)\times \iterP(r_3)$, the coefficients $(a_{ij}^P, a^P, b_1^P, b_0^P)$ with $i,j=1,2$ and the function $f_1^P$ given by Definition \ref{definition:approx coeff and fs} are well defined. And, one can fix a constant $\delta_2>0$ sufficiently small such that if the inequality
\begin{equation*}
  \max\{r_1, r_2, r_3\}\le 2\delta_2
\end{equation*}
holds, then any $\tilT\in \iterT(r_1)$ satisfies \eqref{smallness condition-2}. Moreover, for any $P=(\tphi,\tpsi, \tPsi)\in \iterV(r_2)\times \iterP(r_3)$, the vector field ${\bf v}(\rx, \nabla\tpsi(\rx), \nabla^{\perp}\tphi(\rx))$ given by \eqref{definition of v} satisfies \eqref{smallness condition-3}.
So the functions $f_2^{(\tilT, P)}$ and $f_3^{(\tilT, P)}$ given by Definition \ref{definition:approx coeff and fs} are well defined. We set $\bar{\delta}$ as
\begin{equation}
\label{delta-choice1}
  \bar{\delta}:=\min\{\delta_1, \delta_2\}.
\end{equation}
Then all the properties stated in (a)--(f) can be directly checked.
\medskip

{\textbf{Step 2.}} A direct computation with using $\us$ given by \eqref{EP-1d-reduced} yields $$\bar a_{11}=1-\frac{\bar u_1^{\gam+1}}{\us^{\gam+1}},$$
so we have
\begin{equation*}
  \mfrak{L}_1^{P_0}(v,0)=\left(1-\frac{\bar u_1^{\gam+1}}{\us^{\gam+1}}\right)\der_{11}v+\der_{22}v+\bar a\der_1 v.
\end{equation*}
Then Lemma \ref{lemma-1d-full EP} directly implies the statement ($\tx{g}_1$).
\medskip

{\textbf{Step 3.}} Since $a_{12}^P=a_{21}^P$, we can write $\mfrak{L}_1^P(v,0)$ as
\begin{equation*}
  \mfrak{L}_1^P(v,0)=a_{11}^P\der_{11}v+2a_{12}^P\der_{12}v+\der_{22}v+a^{P}\der_1 v.
\end{equation*}
Due to Lemma \ref{lemma-1d-full EP}, the coefficient $\bar a_{11}$ satisfies
$$
\bar{a}_{11}(L)<0<\bar a_{11}(0),\quad\tx{and}\quad \bar{a}_{11}'=-\frac{(\gam+1)\bar u_1^{\gam}}{\us^{\gam+1}}\bar u_1'.$$

The principal coefficients of the operator $\mfrak{L}_1^P$ form a symmetric matrix $${\mathbb A}^{P}=\begin{pmatrix}a_{11}^{P}&a_{12}^{P}\\
a_{12}^{P}&1\end{pmatrix},$$
so all the eigenvalues of ${\mathbb  A}^{P}$ are real. Hence the operator $\mfrak{L}_1^P$ is hyperbolic if and only if $a^P_{11}-(a^P_{12})^2<0$, and elliptic if and only if $a^P_{11}-(a^P_{12})^2>0$.
By the estimate \eqref{estimate-coefficient-difference} and the generalized Sobolev inequality, it holds that
\begin{equation}
\label{estimate of C1 difference of a_{11}}
\|\det {\mathbb A}^{P}-\det {\mathbb A}^{P_0}\|_{C^1(\ol{\Om_L})}=\|\det {\mathbb A}^{P}-\bar a_{11}\|_{C^1(\ol{\Om_L})}
\le C\bar{\delta}
\end{equation}
for some constant $C>0$ fixed depending only on $(\gam, \zeta_0, J, S_0, E_0,L)$.
Therefore we can reduce the constant $\bar{\delta}>0$ from the one given in \eqref{delta-choice1} so that if the condition \eqref{condition:r} holds, then the following two properties hold:
\begin{equation*}
  \begin{split}
&\max_{\ol{\Gamex}}\det {\mathbb A}^{P}<0<\min_{\ol{\Gamen}}\det {\mathbb A}^{P},\\
\tx{and}\,\,&\max_{\ol{\Om_L}}\der_{x_1}\det {\mathbb A}^{P}\le
    -\frac 12 \min_{\ol{\Om_L}} \frac{(\gam+1)\bar u_1^{\gam+1}}{\us^{\gam+1}}\bar u'<0.
  \end{split}
\end{equation*}
Then the implicit function theorem yields a unique $C^1$-function $\gs^P:[-1,1]\rightarrow \R$ satisfying
\begin{equation*}
\det {\mathbb A}^{P}(x_1, x_2)
\begin{cases}
>0\quad&\mbox{for $x_1<\gs^P(x_2)$},\\
=0\quad&\mbox{for $x_{1}=\gs^P(x_2)$},\\
<0\quad&\mbox{for $x_1>\gs^P(x_2)$}.
\end{cases}
\end{equation*}
Note that
\begin{equation*}
\det {\mathbb A}^{P}(\gs^P(x_2), x_2)-\bar a_{11}(\ls)= 0,  \quad\tx{for $-1\le x_2\le 1$}.
\end{equation*}
which can be rewritten as
\begin{equation*}
 \bar a_{11}(\gs(x_1))-\bar a_{11}(\ls)=\bar a_{11}(\gs(x_2))-a_{11}(\gs(x_2),x_2)+a_{12}^2(\gs(x_2),x_2)\quad\tx{for $|x_1|\le 1$.}
\end{equation*}
By using this equation and the trace inequality, we can easily derive the estimate \eqref{estimate of gs}. And, we can directly check the estimate \eqref{boundes of gs} by using \eqref{estimate of gs} and the generalized Sobolev inequality. This proves the statement (${\tx g}_2$). Finally, the statements (${\tx g}_3$) and (${\tx g}_4$) can be verified by using \eqref{estimate-coefficient-difference}.

\end{proof}
\end{lemma}

The main difficulty of this work is to establish the well-posedness of the boundary value problem \eqref{lbvp main} for
$(\psi, \Psi)$. According to the statement (g) of Lemma \ref{lemma on L_1}, the system
$$
\begin{cases}
      \mfrak{L}_1^P(\psi, \Psi)=f_1^P\\
      \mfrak{L}_2(\psi, \Psi)=f_2^{(\tilT, P)}
      \end{cases}$$
consists of a second order mixed-type equation with a degeneracy and a second order elliptic equation, which are weakly coupled together by lower order derivative terms. Note that the type of the equation $\mfrak{L}_1^P(\psi, \Psi)=f_1^P$ changes from being elliptic to being hyperbolic across an interface of codimension 1. To our best knowledge, there is no general theory that guarantees the well-posedness of a boundary value problem of this type.
      Our strategy to solve the boundary value problem \eqref{lbvp main} is as follows:
      \begin{itemize}
      \item[(1)] First, we shall apply the method developed in \cite[Chapter 1]{KZ} to handle the degeneracy of the operator $\mfrak{L}_1^P$ occurring on $x_1=\gs^P(x_2)$. Namely, we set up a new boundary value problem with a singular perturbation by introducing an auxiliary differential operator $\mfrak{L}_1^P(v,w)+\eps\der_{111}v$ for a constant $\eps>0$. The crucial step is to establish a priori $H^1$ estimate of weak solutions to an auxiliary boundary value problem that includes $$\mfrak{L}_1^P(v,w)+\eps\der_{111}v=f_1.$$
          More importantly, the a priori $H^1$ estimate is to be achieved uniformly with respect to the parameter $\eps>0$. Then the existence of a weak solution to the problem \eqref{lbvp main} is established by passing to the limit as $\eps$ tends to $0+$.
      \item[(2)] Once the existence of a weak solution to \eqref{lbvp main} is achieved, then we establish a priori $H^4$ estimate of the weak solution by applying bootstrap arguments.
      \end{itemize}

But there are still obstacles to overcome.
\smallskip

Our plan is to apply \cite[Theorems 1.1, 1.5 and 1.7]{KZ}. But these theorems can be applied only to a class of equations that satisfy several technical conditions. Nevertheless, thanks to the monotonicity property of $\bar u_1$ stated in Lemma \ref{lemma-1d-full EP}, we discover that the operator $\mfrak{L}_1^{P}(\psi,\Psi)$, by regarding as a second order linear differential operator for $\psi$, satisfies all the technical conditions required to apply \cite[Theorems 1.1, 1.5 and 1.7]{KZ} as long as the constants $r_2$ and $r_3$ in Definition \ref{definition:approx coeff and fs} are fixed properly.

But, the boundary value problem  \eqref{lbvp main} contains a weakly coupled PDE system. And, this makes the problem even more difficult. If the two operators $\mfrak{L}_1^P(\psi,\Psi)$ and $\mfrak{L}_2(\psi,\Psi)$ did not contain the coupling terms such as $b_1^P\der_1\Psi+b_0^P\Psi$ from $\mfrak{L}_1^P(\psi,\Psi)$ and $c_1\der_1\psi$ from $\mfrak{L}_2(\psi,\Psi)$, then we could have established a priori $H^1$ estimate of a weak solution right away by applying \cite[Theorems 1.1, 1.5 and 1.7]{KZ} for $\mfrak{L}_1^P$ and applying the standard elliptic theory for $\mfrak{L}_2$ separately. Unfortunately, this is not the case. The coefficients $b_1^P$, $b_0^P$ and $c_1$ are not negligible, moreover it is unclear whether there is any particular relation among the coefficients so that the coupling terms would disappear by cancellations in an energy estimate as in the study of subsonic flows(see \cite{BDX3, BDX}). In order to resolve this issue, we add the assumptions \eqref{almost sonic condition1 full EP}. In the next section, we explain how these assumptions are used to handle the coupling terms in a priori $H^1$ energy estimate.

\section{The key estimate}
\label{section:essential part}

\subsection{Preliminary}
\label{subsection-lbvp-preliminary}
Given constants $(\gam, \zeta_0, J, S_0, E_0)$ with
\begin{equation*}
  \gam>1,\quad \zeta_0>1,\quad J>0,\quad S_0>0,\quad E_0<0
\end{equation*}
let $(\bar{\vphi}, \bar{\Phi})$ be the associated background solution to the system \eqref{new system with H-decomp2}--\eqref{new system with H-decomp5} in the sense of Definition \ref{definition of background solution-HD}. For a fixed constant $L\in(0, l_{\rm max}]$, let $\bar{\delta}$ be given to satisfy Lemma \ref{lemma on L_1}. Assuming that the condition
\begin{equation}
\label{condition: r in sec4}
  \max\{r_1, r_2, r_3\}\le 2\bar{\delta}
\end{equation}
holds, fix
\begin{equation*}
\tilT\in \iterT(r_1)\quad\tx{and}\quad  P=(\tphi, \tpsi, \tPsi)\in \iterV(r_2)\times \iterP(r_3)
\end{equation*}
for $\iterT(r_1)$ and $\iterV(r_2)\times \iterP(r_3)$ given by Definition \ref{definition: iteration sets}.

In this section, we shall establish the well-posedness of the boundary value problem \eqref{lbvp main}.

\begin{problem}
\label{problem-lbvp for iteration}
Given functions $f_1\in H^3(\Om_L)$ and $f_2\in H^2(\Om_L)$ satisfying the compatibility conditions
\begin{equation}
\label{compatibility conditions for f1 and f2}
  \der_2f_1=0\quad\tx{and}\quad \der_2 f_2=0\quad \tx{on $\Gamw$},
\end{equation}
find $(v,w)$ that solves the following problem:
\begin{equation}
\label{lbvp-main general}
\begin{split}
\begin{cases}
\mfrak{L}_1^P(v, w)=f_1\quad&\tx{in $\Om_L$},\\
\mfrak L_2(v,w)=f_2\quad&\tx{in $\Om_L$},
\end{cases}\\
v=0,\quad \der_1 w=0\quad&\tx{on $\Gamen$},\\
\der_2v=0,\quad \der_2 w=0\quad&\tx{on $\Gamw$},\\
w=0\quad&\tx{on $\Gamex$}.
\end{split}
\end{equation}
\end{problem}

\begin{proposition}
\label{theorem-wp of lbvp for system with sm coeff}
Suppose that $(u_0, E_0)\in \Tac$ with $E_0<0$, which is equivalent to $0<u_0<\us$, and let $(\bar{\vphi}, 0, {\bar{\Phi}}, S_0)$ be the background solution associated with $(\gam, \zeta_0, J, S_0, E_0)$ in the sense of Definition \ref{definition of background solution-HD}. Then, one can fix two constants $\bJ$ and $\ubJ$ satisfying
\begin{equation*}
  0<\bJ<1<\ubJ<\infty
\end{equation*}
so that whenever $J\in(0, \bJ]\cup[\ubJ,\infty)$,
there exists a constant $d\in(0,1)$ such that if
$(E_0, L)$ are fixed to satisfy the condition \eqref{almost sonic condition1 full EP}, then Problem \ref{problem-lbvp for iteration} is well posed for any $P=(\tphi, \tpsi, \tPsi)\in \iterV(r_2)\times \iterP(r_3)$ provided that the constant $\bar{\delta}>0$ from Lemma \ref{lemma on L_1} is adjusted appropriately. In other words, one can reduce the constant $\bar{\delta}>0$ so that, under the condition of \eqref{condition: r in sec4}, for each $P=(\tphi, \tpsi, \tPsi)\in \iterV(r_2)\times \iterP(r_3)$, Problem \ref{problem-lbvp for iteration} acquires a unique solution $(v,w)\in H^4(\Om_L)\times H^4(\Om_L)$. Furthermore, the solution satisfies the estimate
\begin{equation}
\label{estimate of v and w}
      \|v\|_{H^4(\Om_L)}+\|w\|_{H^4(\Om_L)}
      \le C\left(\|f_1\|_{H^3(\Om_L)}
      +\|f_2\|_{H^2(\Om_L)}\right)
    \end{equation}
    for some constant $C>0$.
    \smallskip

    In the above, the parameters $\ubJ$ and $\bJ$ are fixed depending only on $(\gam, \zeta_0, S_0)$. And, the constants $d\in(0,1)$ and $\bar{\delta}$ are fixed depending only on $(\gam, \zeta_0, S_0,J)$ and $(\gam,\zeta_0, S_0, E_0, J, L)$, respectively. Finally, the estimate constant $C$ is fixed depending only on $(\gam,\zeta_0, S_0, E_0, J, L)$.

\end{proposition}

\begin{corollary}
\label{corollary:wp of mixed system in iteration}
Under the assumptions same as Proposition \ref{theorem-wp of lbvp for system with sm coeff}, for each
\begin{equation*}
\tilT\in \iterT(r_1)\quad\tx{and}\quad   P=(\tphi, \tpsi, \tPsi)\in \iterV(r_2)\times \iterP(r_3),
\end{equation*}
the boundary value problem \eqref{lbvp main} has a unique solution $(\psi, \Psi)\in [H^4(\Om_L)]^2$ that satisfies the estimate
\begin{equation}
\label{estimate-linear-pot}
\begin{split}
 & \|\psi\|_{H^4(\Om_L)}+\|\Psi\|_{H^4(\Om_L)}\\
 &\le C\left(\|\tilT\|_{H^3(\Om_L)}+\|\tphi\|_{H^4(\Om_L)}+\|(\tpsi, \tPsi)\|_{H^4(\Om_L)}^2+\mcl{P}(S_0, E_{\rm en}, \om_{\rm en})\right)
  \end{split}
\end{equation}
for some constant $C>0$ depending only on $(\gam,\zeta_0, S_0, E_0, J, L)$. In the above estimate, the term $\mcl{P}(S_0, E_{\rm en}, \om_{\rm en})$ is defined by \eqref{definition-perturbation of bd}.

\begin{proof}
For two functions $v$ and $w$ defined in $\ol{\Om_L}$, define $\psi$ and $\Psi$ by
\begin{equation*}
\begin{cases}
\psi(x_1,x_2):=v(x_1, x_2)+\int_0^{x_2}w_{\rm en}(t)\,dt,\\
\Psi(x_1, x_2):=w(x_1, x_2)+(x_1-L)(E_{\rm en}(x_2)-E_0).
\end{cases}
\end{equation*}
It can be directly checked that $(\psi, \Psi)$ solves \eqref{lbvp main} if and only if $(v,w)$ solves
\begin{equation}
\label{lbvp-mod}
\begin{split}
\begin{cases}
\mfrak{L}_1^P(v, w)=f_1^*\quad&\tx{in $\Om_L$},\\
\mfrak L_2(v,w)=f_{2}^*\quad&\tx{in $\Om_L$},
\end{cases}\\
v=0,\quad \der_1 w=0\quad&\tx{on $\Gamen$},\\
\der_2v=0,\quad \der_2 w=0\quad&\tx{on $\Gamw$},\\
w=0\quad&\tx{on $\Gamex$},
\end{split}
\end{equation}
for $f_1^*$ and $f_2^*$ given by
\begin{equation*}
  \begin{cases}
  f_1^*:=f_1^P-\mfrak{L}_1^P(\int_0^{x_2}w_{\rm en}(t)\,dt, (x_1-L)(E_{\rm en}(x_2)-E_0)),\\
  f_2^*:=f_2^{(\tilT, P)}-\mfrak{L}_2(\int_0^{x_2}w_{\rm en}(t)\,dt, (x_1-L)(E_{\rm en}(x_2)-E_0)).
  \end{cases}
\end{equation*}
If $w_{\rm en}$ and $E_{\rm en}$ satisfy Condition \ref{conditon:1}, then it directly follows from Lemma \ref{lemma on L_1}(f) that
\begin{equation*}
  \der_2f_1^*=\der_2f_2^*=0\quad\tx{on $\Gamw$}.
\end{equation*}
Therefore, Proposition \ref{theorem-wp of lbvp for system with sm coeff} yields a unique solution $(v,w)\in[H^4(\Om_L)]^2$ to \eqref{lbvp-mod}, so the boundary value problem \eqref{lbvp main} has a unique solution $(\psi, \Psi)\in [H^4(\Om_L)]^2$. And, the estimate \eqref{estimate-linear-pot} is a result directly following from the estimate \eqref{estimate of v and w} and Lemma \ref{lemma on L_1}(e).
\end{proof}
\end{corollary}

The main part of this paper is devoted to prove Proposition \ref{theorem-wp of lbvp for system with sm coeff}. Once this proposition is proved, Theorem \ref{theorem-HD} can be established by the standard iteration method.
\smallskip

First of all, we briefly explain the main strategies to prove Proposition \ref{theorem-wp of lbvp for system with sm coeff}.
\smallskip

{\emph{\textbf{Step 1:}}} We find sufficient conditions for the parameters $(J, L)$ and adjust the constant $\bar{\delta}$ from Lemma \ref{lemma on L_1} to establish a priori $H^1$-estimate of a smooth solution $(v, w)$ to \eqref{lbvp-main general}(see Proposition \ref{proposition-H1-apriori-estimate} in \S \ref{subsection: a priori H^1 estimate}). Once the a priori $H^1$-estimate is established, the uniqueness of a classical solution can easily follow. Once again, we emphasize that this is the most important and challenging step.
\smallskip

{\emph{\textbf{Step 2:}}} To prove the existence of a solution, we introduce an auxiliary boundary value problem with {\emph{a singular perturbation}} as follows:
\begin{equation}
\label{bvp-sing-pert-pre}
  \begin{split}
  &\begin{cases}
  \eps \der_{111}v+\mfrak L_1^{P}(v, w)=f_1\quad&\tx{in $\Om_L$},\\
  \mfrak L_2(v, w)=f_2\quad&\tx{in $\Om_L$},
  \end{cases}\\
  &\begin{cases}
  v=0,\quad \der_1v=0\quad&\tx{on $\Gamen$},\\
  \der_2 v=0\quad&\tx{on $\Gamw$},\\
 \der_{11}v=0\quad&\tx{on $\Gamex$},
  \end{cases}\\
  &\begin{cases}
  \der_1 w=0\quad&\tx{on $\Gamen$},\\
  \der_2w=0\quad&\tx{on $\Gamw$},\\
  w=0\quad&\tx{on $\Gamex$}.
  \end{cases}
  \end{split}
\end{equation}
By applying
\begin{itemize}
\item[-]the a priori $H^1$-estimate given in the previous step,
\item[-]the method of Galerkin's approximations,
\item[-]and the Fredholm alternative theorem,
\end{itemize}
we shall prove the following statements:
{\emph{
there exists a constant $\bar{\eps}>0$ sufficiently small such that if $\tilT\in \iterT(r_1)$ and $P\in \iterV(r_2)\times \iterP(r_3)$ are smooth in $\ol{\Om_L}$, then, for each $\eps\in(0, \bar{\eps}]$, the boundary value problem \eqref{bvp-sing-pert} has a unique weak solution $(v_{\eps},w_{\eps})$ in $[H^1(\Om_L)]^2$. Furthermore, the set $\{(v_{\eps},w_{\eps})\}_{\eps\in(0, \bar{\eps}]}$ is bounded in $[H^1(\Om_L)]^2$
(see Lemma \ref{lemma:wp of singular pert prob-main}).}}

Then, the weak compactness of the set $\{(v_{\eps},w_{\eps})\}$ in $H^1$ yields a weak solution to \eqref{lbvp-main general} by taking the limit as $\eps$ tends to $0+$(see Proposition \ref{proposition-wp of bvp with approx coeff}).
The precise definition of a weak solution shall be given in the next section.
\smallskip

{\emph{{\textbf{Step 3:}} }} For a weak solution to \eqref{lbvp-main general}, we establish a global $H^4$-estimate of $(v,w)$ by a bootstrap argument. The main challenge is to establish a priori estimates of higher order derivatives of $v$. According to Lemma \ref{lemma on L_1}(g), $\mfrak{L}_1^P(\cdot, w)$ is a second order linear elliptic-hyperbolic mixed type differential operator of $v$ with a degeneracy on the approximated sonic boundary $x_1=\gs^P(x_2)$. Therefore, there is no known general theory by which we can immediately establish a global $H^k$-estimate of $v$ in $\Om_L$. To overcome this difficulty, we employ the idea from \cite[Chapter 1]{KZ} to establish a global $H^4$-estimate of $(v,w)$ by a bootstrap argument and by using an auxiliary boundary value problem in an extended domain that contains $\Om_L$. It requires several technical procedures so we shall not mention any further details here. All the details are given in \S \ref{subsection:H4 estimate}.
\smallskip

{\emph{{\textbf{Step 4:}}}}
Finally, for any given $P\in \iterV(r_2)\times \iterP(r_3)$, $f_1\in H^3(\Om_L)$ and $f_2\in H^2(\Om_L)$, the existence of a solution $(v,w)$ to \eqref{lbvp-main general} can be established in three steps:
\begin{itemize}
\item[(1)] Approximate $P$ by smooth functions $\{P_n\}$ so that $P_n$ converges to $P$ in appropriately chosen norms.
\item[(2)] For each $n\in\mathbb{N}$, find a solution $(v_n, w_n)$ to the problem \eqref{lbvp-main general} with $P=P_n$ by the method of Galerkin's apprroximations and applying all the results obtained in the previous steps.
\item[(3)] Take the limit of $(v_n, w_n)$ in $H^4$ as $n$ tends to $\infty$ so that the limit $(v_{\infty}, w_{\infty})$ solves \eqref{lbvp-main general} for the originally fixed $P\in \iterV(r_2)\times \iterP(r_3).$
\end{itemize}

\subsection{A priori $H^1$-estimate}
\label{subsection: a priori H^1 estimate}
Throughout this section, let $P=(\tphi, \tpsi, \tPsi)\in \iterV(r_2)\times \iterP(r_3)$ be fixed unless otherwise specified.
\begin{notation}
(1) For simplicity, we denote $(a_{ij}^P, a^P, b_1^P, b_0^P)$, and its approximated sonic interface $\gs^P$(see Lemma \ref{lemma on L_1}(g)) by $(a_{ij}, a, b_1, b_0)$ and $\gs$, respectively.\\
(2) The terms $(a_{ij}, a, b_1, b_0)$ and $\gs$ for $P=(0,0,0)(=:P_0)$ are denoted by $(\bar a_{ij}, \bar a, \bar b_1, \bar b_0)$ and $\ls$(which is a constant), respectively(see Lemma \ref{lemma on L_1}(d) and (g)).
\end{notation}

\begin{proclaim}
(1) For the rest of the paper, we shall state an estimate constant to be fixed depending only on the data if it is fixed depending only on $(\gam, \zeta_0, S_0, E_0, J, L)$.\\
(2) Unless otherwise specified, any estimate constant appearing hereafter is presumed to be fixed depending only on the data.
\end{proclaim}

\begin{lemma}
\label{proposition-H1-apriori-estimate}
Suppose that $(v,w)$ is a smooth solution to the boundary value problem \eqref{lbvp-main general} associated with $P=(\tphi, \tpsi, \tPsi)\in \iterV(r_2)\times \iterP(r_3)$.
There exist two constants $\bJ$ and $\ubJ$ satisfying
\begin{equation*}
  0<\bJ<1<\ubJ<\infty
\end{equation*}
so that, for any $J\in (0, \bJ]\cup [\ubJ, \infty)$, one can fix a constant $d\in(0,1)$, and reduce the constant $\bar{\delta}>0$ further from the one given in Lemma \ref{lemma on L_1}  so that if
\begin{itemize}
\item[(i)]$(E_0, L)$ are fixed to satisfy the condition \eqref{almost sonic condition1 full EP},
\item[(ii)] and if
\begin{equation*}
  \max\{r_2, r_3\}\le 2\bar{\delta},
\end{equation*}
\end{itemize}
then $(v,w)$
satisfies the estimate
    \begin{equation}
    \label{A priori H1 estimate}
    \begin{split}
      &\|\der_{1}v\|_{L^2(\Gamen)}+\|Dv\|_{L^2(\Gamex)}+\|v\|_{H^1(\Om_L)}+\|w\|_{H^1(\Om_L)}\\
      &\le C\left(\|f_1\|_{L^2(\Om_L)}+\|f_2\|_{L^2(\Om_L)}\right)
      \end{split}
    \end{equation}
    for some constant $C>0$.

    In the above, the constants $\bJ$ and $\ubJ$ are fixed depending only on $(\gam, \zeta_0, S_0)$, and, for each $J\in (0, \bJ]\cup [\ubJ, \infty)$, the constant $d\in(0,1)$ is depending on $(\gam, \zeta_0, S_0,J)$. Finally, the constant $\bar{\delta}>0$ is adjusted depending only on the data (in the sense of Proclamation (1)).

\end{lemma}

\begin{proof}

{\textbf{Step 1.}}
For  $i,j=1,2$, let us set
\begin{equation}
\label{def of coeff pert}
(d a_{ij}, d a, d b_1, d b_0)
:=(a_{ij}, a, b_1, b_0)-(\bar a_{ij}, \bar a, \bar b_1, \bar b_0)\quad\tx{in $\Om_L$}.
\end{equation}
Next, let us set
\begin{equation*}
 \mfrak{L}_1:=\mfrak{L}_1^{P_0}\quad\tx{and}\quad
 d \mfrak{L}_1^{P}:=\mfrak{L}_1^{P}-\mfrak{L}_1.
\end{equation*}
Then we have
\begin{equation*}
 d \mfrak{L}_1^{P}(v,w)=d a_{11}\der_{11}v+2a_{12}\der_{12}v+d a\der_1 v+d b_1\der_1 w+d b_0w\quad\tx{in $\Om_L$}.
\end{equation*}
And, we rewrite the boundary value problem \eqref{lbvp-main general} as
\begin{equation*}
\begin{split}
\begin{cases}
\mfrak L_1(v, w)=f_1-d \mfrak{L}_1^{P}(v,w)=:F_1\quad&\tx{in $\Om_L$},\\
\mfrak L_2(v,w)=f_2\quad&\tx{in $\Om_L$},
\end{cases}\\
v=0,\quad \der_1 w=0\quad&\tx{on $\Gamen$},\\
\der_2v=0,\quad \der_2 w=0\quad&\tx{on $\Gamw$},\\
w=0\quad&\tx{on $\Gamex$}.
\end{split}
\end{equation*}

For a function $G=G(x_1)$ to be determined later, denote
\begin{equation}
\label{definition of I1 and I2}
 I_1:=\int_{\Om_L}  G\der_1v{\mfrak L}_1(v,w)\,d\rx,\quad
 I_2:=\int_{\Om_L} w{\mfrak L}_2(v,w)\,d\rx.
\end{equation}
Clearly, it holds that
\begin{equation}
\label{I1+I2 in H1 estimate}
  I_1+I_2=\int_{\Om_L} F_1G\der_1v+f_2w\,d\rx.
\end{equation}
By using Definition \ref{definition:coefficients-nonlinear}, and integrating by parts with using the boundary conditions for $(v,w)$, one can check that
\begin{equation*}
\begin{split}
  I_1=&\frac 12\int_{\Gamex}(\bar a_{11}(\der_1v)^2-(\der_2v)^2)G\,dx_2-\frac 12\int_{\Gam_0} \bar a_{11}(\der_1v)^2 G\,dx_2\\
  &+\int_{\Om_L}\left(\left(-\frac 12\bar a_{11}'+\bar a\right)G-\frac 12 \bar a_{11}G'\right) (\der_1 v)^2+G'\frac{(\der_2v)^2}{2}\,d\rx\\
  &+\int_{\Om_L} (\bar b_1\der_1v\der_1w+\bar b_0w\der_1v)G\,d\rx,\\
 \tx{and}\quad I_2
  =&\int_{\Om_L}\left(-|\nabla w|^2-\frac{\bar{\rho}}{\gam S_0\bar{\rho}^{\gam-1}} w^2\right)+\frac{J}{\gam S_0\bar{\rho}^{\gam-1}}w\der_1v \,d\rx.
  \end{split}
\end{equation*}
\smallskip

The key of the proof is to  choose the function $G$ as
\begin{equation}
\label{definition of G}
  G(x_1)=\bar{\rho}^{\eta}(x_1)
\end{equation}
for a constant $\eta>0$ to be determined later.
\smallskip

Define a function $\alp_1$ by
\begin{equation}
\label{main-coeff}
\begin{split}
\alp_1:=&  -\frac 12(\bar a_{11}G)'+\bar aG=\frac{\bar{\rho}'}{\bar{\rho}} \frac G2 \left((\gam-1+\eta)\frac{\bar u_1^{\gam+1}}{\us^{\gam+1}}+(2-\eta)\right).
  \end{split}
\end{equation}
We rearrange the terms in $-(I_1+I_2)$ as follows:
\begin{equation*}
-(I_1+I_2)=T_{\rm bd}+T_{\rm coer}+T_{\rm mix}
\end{equation*}
for
\begin{align}
\label{definition-Jbd}
T_{\rm bd}:=&\frac{G(0)}{2} \int_{\Gam_0}\bar a_{11}(\der_1v)^2 \,dx_2-\frac{G(L)}{2}\int_{\Gamex}(\bar a_{11}(\der_1v)^2-(\der_2v)^2)\,dx_2\\ 
\label{definition-Jp}
T_{\rm coer}:=&-\int_{\Om_L}\alp_1(\der_1v)^2
  +\eta \frac{\bar{\rho}'}{\bar{\rho}}G\frac{(\der_2 v)^2}{2}\,d\rx+
  \int_{\Om_L}|\nabla w|^2+\frac{\bar{\rho}}{\bar c^2} w^2\,d\rx\\
  \label{definition-Jc}
T_{\rm mix}:=&  \int_{\Om_L}\frac{(\gam-1)\bar u}{\gam S_0\bar{\rho}^{\gam-1}}\frac{\bar \rho'}{\bar\rho}G w\der_1v
-G\frac{\bar u}{\gam S_0\bar{\rho}^{\gam-1}}\der_1w\der_1v-\frac{J}{\gam S_0\bar{\rho}^{\gam-1}}\der_1v w\,d\rx.
\end{align}
Note that the proof of Lemma \ref{lemma on L_1}(see Step 3 in the proof) gives that
\begin{equation*}
  \bar a_{11}>0\,\,\tx{on $\Gamen$}\quad\tx{and}\quad \bar a_{11}<0\,\,\tx{on $\Gamex$},
\end{equation*}
and this implies that
 \begin{equation*}
  T_{\rm bd}\ge 0.
 \end{equation*}
In Step 5, we give an improved estimate for a lower bound of $T_{\rm bd}$, so that it can be used to handle the singular perturbation problem \eqref{bvp-sing-pert}.
\smallskip

In the following steps, we shall fix $\eta>0$ and find $(\bJ, \ubJ)$ with
\begin{equation*}
  0<\bJ<1<\ubJ<\infty
\end{equation*}
so that whenever $J\in(0, \bJ]\cup[\ubJ,\infty)$, there exists a constant $d\in(0,1)$ depending on $(\gam, \zeta_0, S_0,J)$ such that if
$(E_0, L)$ are fixed to satisfy the condition \eqref{almost sonic condition1 full EP}, then it holds that
\begin{equation*}
T_{\rm coer}+T_{\rm mix}\ge \lambda \int_{\Om_L} |Dv|^2+|Dw|^2\,d\rx
\end{equation*}
for some constant $\lambda>0$.
\medskip

{\textbf{Step 2.}} Define a function $\beta$ by
\begin{equation}
\label{definition of omega}
  \beta:=\frac{(\gam-1)\bar u_1}{\gam S_0\bar{\rho}^{\gam-1}}\frac{\bar \rho'}{\bar\rho}G-\frac{J}{\gam S_0\bar{\rho}^{\gam-1}},
\end{equation}
and rewrite $T_{\rm mix}$ as
\begin{equation*}
T_{\rm mix}=\int_{\Om_L}\beta w \der_1v-G\frac{\bar u_1}{\gam S_0\bar{\rho}^{\gam-1}}\der_1w\der_1v\,d\rx.
\end{equation*}
Owing to the boundary condition $w=0$ on $\Gamex$, Poincar\'e inequality gives
\begin{equation*}
\int_{\Om_L}w^2\,d\rx\le L^2 \int_{\Om_L}(\der_1w)^2\,d\rx.
\end{equation*}
By using this estimate and the Cauchy-Schwarz inequality, one has
\begin{equation}
\label{estimate-Tmix}
\begin{split}
|T_{\rm mix}|
  &\le \frac 14 \int_{\Om_L}(\der_1 w)^2\,d\rx+
  \int_{\Om_L}
  2\left(\left(G\frac{\bar u_1}{\gam S_0\bar{\rho}^{\gam-1}}\right)^2+L^2\beta^2\right)(\der_1 v)^2\,d\rx.
  \end{split}
\end{equation}
Define
\begin{equation}
\label{definition-alp2}
  \alp_2:=2\left(\left(G\frac{\bar u_1}{\gam S_0\bar{\rho}^{\gam-1}}\right)^2+L^2\beta^2\right).
\end{equation}
For $\alp_1$ given by \eqref{main-coeff}, let us set $\alp$ as
\begin{equation}
\label{definition of beta*}
\alp:=-\alp_1-\alp_2.
\end{equation}
Then it follows from \eqref{definition-Jp} and \eqref{estimate-Tmix} that
\begin{equation}
\label{inequality 1 of Jp+JC}
  T_{\rm coer}+T_{\rm mix}\ge
  \int_{\Om_L} \alp (\der_1 v)^2-\frac{\eta}{2}\frac{\bar \rho'}{\bar \rho}G (\der_2v)^2\,d\rx
  +\frac 34\int_{\Om_L}|\nabla w|^2+\frac{\bar{\rho}}{\gam S_0\bar{\rho}^{\gam-1}} w^2\,d\rx.
\end{equation}
As long as $\eta$ is fixed as a positive constant, it follows from Lemma \ref{lemma-1d-full EP} that
\begin{equation*}
  \min_{\ol{\Om_L}}\frac{\eta}{2}\frac{\bar u_1'}{\bar u_1}G>0.
\end{equation*}
\medskip

{\textbf{Step 3.}} Note that $(\bar u_1, \bar E)(x_1)$ satisfies
\begin{equation*}
\bar{E}^2=2H(\bar u_1)\quad\tx{for $x_1\in (0, l_{\rm max}]$}
\end{equation*}
where the function $H$ given by \eqref{definition of H}. Define
\begin{equation}\label{definition of kappa}
  \kappa:=\frac{\bar u_1}{\us}\quad E(\kappa):=\bar{E}(\bar u_1),\quad \tx{and}\quad \mcl{F}(\kappa):=  \int_1^{\kappa}
\left(1-\frac{t}{\zeta_0}\right)
\left(1-\frac{1}{t^{\gam+1}}\right)
\,
dt.
\end{equation}
Then, the function $E(\kappa)$ satisfies
\begin{equation}
\label{definition of F}
\begin{split}
  E^2(\kappa)=2\us J\mcl{F}(\kappa).
\end{split}
\end{equation}
By the definition of $\Tac$(see \eqref{definition of T and Tpm}), it holds that
\begin{equation*}
  (\kappa-1)E(\kappa)\ge 0,
\end{equation*}
so we obtain that
\begin{equation}
\label{expression of E in tau}
   E(\kappa)=\begin{cases}
  -\sqrt{2\us J\mcl{F}(\kappa)}&\quad\mbox{for $\kappa<1$},\\
  \sqrt{2\us J\mcl{F}(\kappa)}&\quad\mbox{for $\kappa\ge 1$}.
  \end{cases}
\end{equation}
Set $\mcl{H}(\kappa)$ as
\begin{equation}
\label{definition of mcl H}
    \mcl{H}(\kappa):=
  \left|\frac{ \kappa^{\gam-1}\sqrt{\mcl{F}(\kappa)}}{\kappa^{\gam+1}-1}\right|.
\end{equation}
Substituting \eqref{expression of E in tau} into the differential equation for $\bar u_1$ given in \eqref{EP-1d-reduced} yields
\begin{equation}
\label{ode of rho in kappa}
\begin{split}
  \frac{\bar u'_1}{\bar u_1}
   &=\frac{\sqrt{2\us J}}{\us^2}\mcl{H}(\kappa),\quad\tx{and}\quad
   \frac{\bar{\rho}'}{\bar{\rho}}=
    -\frac{\sqrt{2\us J}}{\us^2}
  \mcl{H}(\kappa).
 \end{split}
\end{equation}

By using the definition of $\us$ given in \eqref{EP-1d-reduced}, one can express $\us$ as
\begin{equation}
\label{definition of us}
\us=h_0J^{\frac{\gam-1}{\gam+1}}
\quad\tx{with $h_0:=(\gam S_0)^{\frac{1}{\gam+1}}$}.
\end{equation}
Substitute this expression into \eqref{ode of rho in kappa} to get
\begin{equation}
\label{ode for rho in kappa}
    \frac{\bar{\rho}'}{\bar{\rho}}=-\sqrt 2 h_0^{-\frac 32} J^{\frac{2-\gam}{\gam+1}}\mcl{H}(\kappa).
\end{equation}
Then, the function $\alp_1$(see \eqref{main-coeff}) can be expressed as
\begin{equation}
\label{alp1-Jkappa-expre}
\begin{split}
  -\alp_1= \sqrt 2 h_0^{-\frac 32} J^{\frac{2-\gam}{\gam+1}} \mcl{H}(\kappa) \frac G2
  \left((\gam-1+\eta)\kappa^{\gam+1}+2-\eta\right).
  \end{split}
\end{equation}

Using \eqref{definition of kappa}, \eqref{definition of us} and the expression $\displaystyle{\bar{\rho}=\frac{J}{\bar u_1}}$, one can explicitly check that
\begin{equation*}
\begin{split}
  &\left(G\frac{\bar u_1}{\gam S_0\bar{\rho}^{\gam-1}}\right)^2=\frac{G^2\kappa^{2\gam}
  J^{\frac{-2(\gam-1)}{\gam+1}}}{h_0^2},\\
  \tx{and}\quad &\beta=-\sqrt 2 (\gam-1)h_0^{-\frac 12}\kappa^{\gam}\mcl{H}(\kappa)J^{\frac{3-2\gam}{\gam+1}}G-
h_0^{-2}\kappa^{\gam-1}J^{\frac{3-\gam}{\gam+1}}.
  \end{split}
\end{equation*}

Since $\displaystyle{\min_{x_1\in[0, L]}\bar{u}'_1(x_1)>0}$ holds by Lemma \ref{lemma-1d-full EP}, the function $\bar u_1:[0, L]\rightarrow (0, \infty)$ is invertible. So we use the first equation in \eqref{ode of rho in kappa} to get
\begin{equation}
\label{L-computation-n}
  L=\int_{u_0}^{\bar u_1(L)}\frac{\bar u_1}{\bar u'_1}\frac{d\bar u_1}{\bar u_1}=\sqrt{\frac{h_0^3}{2}}
  J^{\frac{\gam-2}{\gam+1}}\int_{\kappa_0}^{\kappa_L}\frac{1}{\kappa \mcl{H}(\kappa)}\,d\kappa
\end{equation}
for
\begin{equation*}
 \kappa_0:=\frac{u_0}{\us},\quad \kappa_L:=\frac{\bar u_1(L)}{\us}.
\end{equation*}
Define
\begin{equation}
\label{definition of lambda}
\lambda(\kappa_0, \kappa_L):=\left(\int_{\kappa_0}^{\kappa_L} \frac{1}{\tau\mcl{H}(\kappa)}\,d\kappa\right)^2.
\end{equation}
The straightforward computation gives
\begin{equation*}
  (L\beta)^2=\frac{1}{2h_0}\kappa^{2(\gam-1)}
  \lambda(\kappa_0, \kappa_L)J^{\frac{-2(\gam-1)}{\gam+1}}
  \left(\sqrt 2(\gam-1)h_0^{\frac 32} \kappa\mcl{H}(\kappa)G+J^{\frac{\gam}{\gam+1}}\right)^2.
\end{equation*}
In terms of $(J,\kappa)$, $\alp_2$ can be expressed as
\begin{equation}
\label{alp2-Jkappa-expre}
\begin{split}
  {\alp}_2=&\frac{2}{h_0^2}\kappa^{2\gam} J^{\frac{-2(\gam-1)}{\gam+1}}G^2\\
  &+
  \frac{1}{h_0}\kappa^{2(\gam-1)}\lambda(\kappa_0, \kappa_L)J^{\frac{-2(\gam-1)}{\gam+1}}
  \left(\sqrt 2(\gam-1)h_0^{\frac 32} \kappa\mcl{H}(\kappa)G+J^{\frac{\gam}{\gam+1}}\right)^2.
  \end{split}
\end{equation}

Define
\begin{equation*}
G_*:=J^{\frac{2-\gam}{\gam+1}}G.
\end{equation*}
By using \eqref{definition of G} and \eqref{definition of us}, the term $G_*$ can be expressed in terms of $(J, \kappa)$ as
\begin{equation*}
  G_*=\kappa^{-\eta}h_0^{-\eta}J^{\frac{2-\gam+2\eta}{\gam+1}}.
\end{equation*}
By using \eqref{alp1-Jkappa-expre}, \eqref{alp2-Jkappa-expre} and the expression right in the above, we represent $-\alp_1$ and $\alp_2$ as
\begin{equation*}
\begin{split}
-\alp_1&=\frac{\sqrt 2}{2}h_0^{-\frac 32} G_*\mcl{H}(\kappa),\\
  \alp_2&=\frac{2}{h_0^{2+\eta}}\kappa^{2\gam-\eta}J^{\frac{-\gam+2\eta}{\gam+1}}G_*+
  \frac{1}{h_0}\kappa^{2(\gam-1)}\lambda(\kappa_0, \kappa_L)J^{\frac{2}{\gam+1}}
  \left(\sqrt{2}(\gam-1)h_0^{\frac 32-\eta}\kappa^{1-\eta}\mcl{H}(\kappa)J^{\frac{2\eta-\gam}{\gam+1}}+1\right)^2.
\end{split}
\end{equation*}
Then the function $\alp$ given by \eqref{definition of beta*} can be expressed as
\begin{equation}
\label{new expression of alp}
  \alp(\kappa;\kappa_0, \kappa_L, J)=\om_1(\kappa; J, \eta)G_*-\om_2(\kappa;\kappa_0, \kappa_L, J, \eta)
\end{equation}
for
\begin{equation}
\label{expression of alp}
  \begin{split}
  &\om_1(\kappa;J, \eta):=\frac{\sqrt 2}{2}h_0^{-\frac 32}\mcl{H}(\kappa)((\gam-1)\kappa^{\gam+1}+\eta(\kappa^{\gam+1}-1)+2)-
  \frac{2}{h_0^{2+\eta}}\kappa^{2\gam-\eta}J^{\frac{2\eta-\gam}{\gam+1}},\\
  &\om_2(\kappa;\kappa_0, \kappa_L, J, \eta):=\frac{1}{h_0}\kappa^{2(\gam-1)}\lambda(\kappa_0, \kappa_L)J^{\frac{2}{\gam+1}}
  \left(\sqrt 2(\gam-1)h_0^{\frac 32-\eta}\kappa^{1-\eta}\mcl{H}(\kappa)J^{\frac{2\eta-\gam}{\gam+1}}+1\right)^2.
  \end{split}
\end{equation}

\medskip

{\textbf{Step 4.}}
From \eqref{definition of F}, it can be easily checked that $\mcl{F}(1)=0$ and $\mcl{F}'(1)=0$. By applying L$'$h\^{o}pital's rule, one can directly check that
\begin{equation}
\label{H at 1}
 \mcl{H}(1)= \lim_{\kappa\to 1} \mcl{H}(\tau)=\lim_{\kappa\to 1}\frac{\sqrt{\frac 12 \mcl{F}''(1)(\kappa-1)^2}}{\kappa^{\gam+1}-1}=\frac{\sqrt{\frac 12 \mcl{F}''(1)}}{\gam+1}=\frac{\sqrt{1-\frac{1}{\zeta_0}}}{\sqrt{2(\gam+1)}}.
 \end{equation}
 This result, combined with \eqref{condition for zeta0} implies that $\mcl{H}(1)>0$. Therefore, $\mcl{H}(\kappa)$, given by \eqref{definition of mcl H}, satisfies that
\begin{equation}\label{positivity of mcl H}
  \mcl{H}(\kappa)>0\quad\tx{for all $\kappa>0$.}
\end{equation}
It follows from \eqref{definition of lambda} that
\begin{equation*}
\lim_{{\kappa_0\to 1-}\atop{\kappa_L\to 1+}}\lambda(\kappa_0, \kappa_L)=0.
\end{equation*}
Hence we obtain the following important result:
\begin{equation}
\label{alp limit}
  \lim_{{\kappa_0\to 1-}\atop{\kappa_L\to 1+}}{\alp}(1 ;\kappa_0, \kappa_L, J,\eta)=
  h_0^{-\eta}J^\frac{2-\gam+2\eta}{\gam+1}
  \left(\frac 12h_0^{-\frac 32}\sqrt{1-
  \frac{1}{\zeta_0}}\sqrt{\gam+1}
  -\frac{2}{h_0^{2+\eta}}J^{\frac{2\eta-\gam}{\gam+1}}\right).
\end{equation}

Now, we consider two cases:\\
\phantom{a}\quad (Case 1)\,\,$\eta>0$ is fixed to satisfy $2\eta-\gam>0$,\\
\phantom{a}\quad (Case 2)\,\,$\eta>0$ is fixed to satisfy $2\eta-\gam<0$.
\medskip

 Back to \eqref{definition of G}, fix the constant $\eta$ as $\eta=\frac 34\gam$ which implies $2\eta-\gam>0$. This corresponds to (Case 1). Then, one can fix a constant $\bar{J}>0$ sufficiently small depending only on $(\gam, \zeta_0, S_0)$ so that whenever the background momentum density $J$ satisfies the inequality
\begin{equation*}
  0<J\le \bar J,
\end{equation*}
we obtain from \eqref{alp limit} that
\begin{equation*}
 \lim_{{\kappa_0\to 1-}\atop{\kappa_L\to 1+}}{\alp}(1 ;\kappa_0, \kappa_L, J, \frac 34 \gam)\ge \frac 14 h_0^{-\frac 34(\gam+2)}J^{\frac{4+\gam}{2(\gam+1)}}\sqrt{(\gam+1)\left(1-\frac{1}{\zeta_0}\right)}.
\end{equation*}
Note that ${\alp}(\kappa; \kappa_0, \kappa_L, J, \frac 34 \gam)$ is continuous with respect to $(\kappa, \kappa_0, \kappa_L)$. Therefore, one can fix a small constant $d$ with $d\in (0,1)$ depending only on $(\gam, \zeta_0, S_0, J)$ so that if the inequality
\begin{equation*}
  1-d\le\kappa_0< 1< \kappa_L \le 1+d
\end{equation*}
holds,
then we have
\begin{equation}
\label{positive lower bound 1 of alp}
   \min_{\kappa\in[\kappa_0, \kappa_L]}{\alp}(\kappa ;\kappa_0, \kappa_L, J, \frac 34\gam)\ge \frac 18 h_0^{-\frac 34(\gam+2)}J^{\frac{4+\gam}{2(\gam+1)}}\sqrt{(\gam+1)\left(1-\frac{1}{\zeta_0}\right)}.
\end{equation}
\medskip

 Next, fix $\eta$ as $\eta=\frac{\gam}{4}$ which implies $2\eta-\gam<0$. This corresponds to (Case 2).  Then, one can fix a constant $\underline{J}>1$ sufficiently large depending only on $(\gam, \zeta_0, S_0)$ so that whenever the background momentum density $J$ satisfies the inequality
\begin{equation*}
  J\ge \underbar{J},
\end{equation*}
then we obtain that
\begin{equation*}
  \lim_{{\kappa_0\to 1-}\atop{\kappa_L\to 1+}}{\alp}(1 ;\kappa_0, \kappa_L, J, \frac{\gam}{4})\ge \frac 14 h_0^{-\frac 14(\gam+6)}
  J^{\frac{4-\gam}{\gam+1}}
  \sqrt{(\gam+1)\left(1-\frac{1}{\zeta_0}\right)}>0.
\end{equation*}
Therefore, one can fix a small constant $d$ with $d\in(0,1)$ depending only on $(\gam, \zeta_0, S_0, J)$ so that if the inequality
\begin{equation*}
  1-d\le\kappa_0< 1< \kappa_L \le 1+d
\end{equation*}
holds, then we have
\begin{equation}
\label{positive lower bound 2 of alp}
 \min_{\kappa\in[\kappa_0, \kappa_L]} {\alp}(\kappa ;\kappa_0, \kappa_L,J, \frac{\gam}{4})\ge \frac 18 h_0^{-\frac 14(\gam+6)}
  J^{\frac{4-\gam}{\gam+1}}
  \sqrt{(\gam+1)\left(1-\frac{1}{\zeta_0}\right)}.
\end{equation}

\medskip

{\textbf{Step 5.}} Back to \eqref{inequality 1 of Jp+JC}, we shall use \eqref{definition of kappa}, \eqref{definition of us} and \eqref{ode for rho in kappa} to express the coefficients $-\frac{\eta}{2}\frac{\bar{\rho}'}{\bar{\rho}}G$ and $\frac{\bar{\rho}}{\gam S_0 \bar{\rho}^{\gam-1}}$ given on the right-hand side in terms of $(\kappa, J)$ to get
\begin{equation*}
\begin{split}
  T_{\rm coer}+T_{\rm mix}\ge&
  \int_{\Om_L} \alp (\der_1 v)^2+
  \frac{\eta}{\sqrt 2}
 h_0^{-(\frac 32+\eta)}\kappa^{-\eta}\mcl{H}(\kappa)
  J^{\frac{2+2\eta-\gam}{\gam+1}} (\der_2v)^2\,d\rx\\
  &+\frac 34\int_{\Om_L}|\nabla w|^2+\frac{1}{h_0^3}\kappa^{\gam-2}J^{\frac{2(2-\gam)}{\gam+1}} w^2\,d\rx.
  \end{split}
\end{equation*}
For the constant $\eta$ fixed as
\begin{equation}
\label{choice of eta}
  \eta=\begin{cases}
  \frac{3\gam}{4}\quad&\mbox{for $J\le \bJ$},\\
  \frac{\gam}{4}\quad&\mbox{for $J\ge \ubJ$},
  \end{cases}
\end{equation}
define
\begin{equation*}
  \lambda_0:=\min_{ \kappa\in[\kappa_0, \kappa_L]} \left\{{\alp}(\kappa ;\kappa_0, \kappa_L, J), \,\, \frac{\eta}{\sqrt 2}
 h_0^{-(\frac 32+\eta)}\kappa^{-\eta}\mcl{H}(\kappa)
  J^{\frac{2+2\eta-\gam}{\gam+1}},\,\, \frac{1}{h_0^3}\kappa^{\gam-2}J^{\frac{2(2-\gam)}{\gam+1}},\,\,
  \frac 34 \right\}.
\end{equation*}
By \eqref{positivity of mcl H}, \eqref{positive lower bound 1 of alp} and \eqref{positive lower bound 2 of alp}, it is clear that $\lambda_0$ is positive, and we obtain that
\begin{equation}
\label{estimate of Jp+Jc}
T_{\rm coer}+T_{\rm mix}\ge \lambda_0\int_{\Om_L}|Dv|^2+|Dw|^2+w^2\,d\rx.
\end{equation}
Note that the choice of $\lambda_0$ depends only on $(\gam, \zeta_0, S_0, J)$.
\medskip

\quad\\
In treating the boundary value problem \eqref{bvp-sing-pert-pre} with a singular perturbation, it is important to have a coercivity of $T_{\rm bd}$, given by \eqref{definition-Jbd}, as a functional of $v$.
By rewriting $G(0)$, $G(L)$, $\bar{a}_{11}(0)$ and $\bar{a}_{11}(L)$ in terms of $(\kappa_0, \kappa_L, J)$, we can express $T_{\rm bd}$ as
\begin{equation*}
  T_{\rm bd}=\frac{ h_0^{-\eta}J^{\frac{2\eta}{\gam+1}}}{2}
  \left(\kappa_0^{-\eta}(1-\kappa_0^{\gam+1})\int_{\Gamen} (\der_1 v)^2\,dx_2
  +\kappa_L^{-\eta}\int_{\Gamex}(\kappa_L^{\gam+1}-1) (\der_1 v)^2+ (\der_2 v)^2\,dx_2\right).
\end{equation*}
Since we have $\kappa_0<1<\kappa_L$, one can fix a constant $\lambda_{\rm bd}>0$ satisfying
\begin{equation}\label{estimate of Jbd}
 T_{\rm bd}\ge \lambda_{\rm bd}\left(\int_{\Gamen} (\der_1 v)^2\,dx_2+
\int_{\Gamex} |Dv|^2\,dx_2\right).
\end{equation}

Finally, we combine \eqref{estimate of Jp+Jc} and \eqref{estimate of Jbd} to obtain that
\begin{equation}
\label{lower bound of -(I1+I2)}
\tx{LHS of \eqref{I1+I2 in H1 estimate}}
  \le -\lambda_0\int_{\Om_L}|Dv|^2+|Dw|^2+w^2\,d\rx
  -\lambda_{\rm bd}
  \left(\int_{\Gamen}(\der_1v)^2\,dx_2
  +\int_{\Gamex}|Dv|^2\,dx_2\right).
\end{equation}

\medskip
{\textbf{Step 6.}} By using the definition of $a_{12}$ given by \eqref{coefficients of L_1} and the compatibility condition $\der_2\tpsi=0$ on $\Gamw$, it can be directly checked that
\begin{equation}
\label{compatibility condition of a12}
  a_{12}=0\quad\tx{on $\Gamw$}.
\end{equation}
Then we integrate by parts and use \eqref{compatibility condition of a12} to directly get
\begin{equation*}
\begin{split}
&\int_{\Om_L} d\mfrak{L}_1^{P}(v,w) G\der_1 v\,d\rx\\
&=\left(\int_{\Gamex}-\int_{\Gamen}\right)Gd a_{11}\frac{(\der_1 v)^2}{2}\,dx_2
-\int_{\Om_L}\left(\frac 12 \der_1(G\,da_{11})+G(\der_2a_{12}-d a)\right)(\der_1 v)^2\,d\rx\\
&\phantom{==}
+\int_{\Om_L} G(d b_1\der_1 w+d b_0 w)\der_1 v\,d\rx.
\end{split}
\end{equation*}
If the condition \eqref{condition: r in sec4} holds, then Lemma \ref{lemma on L_1}(d) combined with the generalized Sobolev inequality yields
\begin{equation}
\label{estimate of delta coefficients}
  \|(d a_{11}, a_{12}, d a, d b_1, \delta b_0)\|_{C^1(\ol{\Om_L})}\le C(r_2+r_3)
\end{equation}
for some constant $C>0$. So we obtain that
\begin{equation}
\begin{split}
\label{weak estimate of approx problem remainder}
  &\left|\int_{\Om_L} d\mfrak{L}_1^P(v,w) G\der_1 v\,d\rx\right|\\
  &\le C_*(r_2+r_3)\left(\|\der_1v\|^2_{L^2(\Gamen)}+\|\der_1v\|^2_{L^2(\Gamex)}+\|\der_1 v\|^2_{L^2(\Om_L)}\right)
  \end{split}
\end{equation}
for some constant $C_*>0$.

Notice that the function $G$ is already fixed in the previous steps, and we know that the maximum of $|G|$ over $\ol{\Om_L}$ depends only on $(\gam, \zeta_0, S_0, J, E_0)$. So the estimate constant $C_*$ can be fixed depending only on $(\gam, \zeta_0, S_0, J, E_0,L)$.

Therefore, we can estimate the right-hand side of \eqref{I1+I2 in H1 estimate} as
\begin{equation*}
\begin{split}
  &|{\tx{RHS of \eqref{I1+I2 in H1 estimate}}}|\\
  &\le
  C_*\bar{\delta}\left(\|\der_1v\|^2_{L^2(\Gamen)}+\|\der_1v\|^2_{L^2(\Gamex)}+\|\der_1 v\|^2_{L^2(\Om_L)}\right)+\int_{\Om_L}|f_1 G\der_1 v+f_2 w|\,d\rx.
  \end{split}
\end{equation*}
Finally, the proof of Lemma \ref{proposition-H1-apriori-estimate} can be completed by reducing $\bar{\delta}>0$ and applying the Cauchy-Schwarz inequality.

\end{proof}

\begin{remark}
\label{remark: nozzle length L}
In Lemma \ref{proposition-H1-apriori-estimate}, the essential condition to achieve an a priori $H^1$-estimate of $(v,w)$ is that the function $\kappa(=\frac{\bar u_1}{\us})$ needs to be close to 1(see \eqref{almost sonic condition1 full EP}). This condition is satisfied if $\kappa_0$ is fixed in $[1-d,1)$ for a sufficiently small constant $d>0$, and if the nozzle length $L$ is fixed to be sufficiently small. But, we should point out that this condition does not necessarily imply that the nozzle length $L$ is small. In Appendix \ref{appendix:nozzle length}, we  give examples of $(\gam, J)$ for which the nozzle length $L$ can be large.

\end{remark}

\section{The unique existence of a strong solution to Problem \ref{problem-lbvp for iteration}}
\label{section:singular perturbation}

\subsection{Extensions of coefficients}
\label{subsection:coeff ext}

Given constants $\gam>1$, $\zeta_0>1$, $J>0$, $S_0>0$ and $E_0<0$, the constant $l_{\rm max}$, given in Lemma \ref{lemma-1d-full EP}, represents the maximal length of the nozzle $\Om_L$ through which the associated smooth transonic solution $(\bar u_1, \bar E)$ to \eqref{1d-ivp} has a positive acceleration of the flow speed. This monotonicity property is the key for the following lemma, which plays crucial roles in this and the next sections.
\begin{lemma}
\label{lemma-coefficient at bg}
For each $L\in(0, l_{\rm max})$, there exists a constant $\lambda_L>0$ so that the coefficients $\bar a_{11}$ and $\bar a$ given by \eqref{coefficient at bg} satisfy the following properties:

\begin{itemize}
\item[(a)] $\displaystyle{-\der_{1}\bar a_{11}\ge \lambda_L}$ in $\ol{\Om_L}$
\item[(b)] $\displaystyle{-\bar a\ge \lambda_L}$ in $\ol{\Om_L}$
\item[(c)] $\displaystyle{-2\bar a-(2m-1)\der_{1}\bar a_{11}\ge \lambda_L}$ in $\ol{\Om_L}$ for $m=0, 1, 2, 3$
\item[(d)] For each $L\in(0, l_{\rm max})$, define
\begin{equation*}
\begin{split}
 &\mfrak{a}_1(L):= \min_{\ol{\Om_L}}(-\der_{1}\bar a_{11}),\quad
 \mfrak{a}_2(L):=\min_{\ol{\Om_L}}(-\bar a),\\
 &\mfrak{c}_m(L):=\min_{\ol{\Om_L}}(-2\bar a-(2m-1)\der_{1}\bar a_{11})\quad\tx{for $m=0,1,2,3$}.
 \end{split}
\end{equation*}
Then, we have
    $$
    \lim_{L\to l_{\rm max}-}\mfrak{a}_1(L)=
    \lim_{L\to l_{\rm max}-}\mfrak{a}_2(L)=
    \lim_{L\to l_{\rm max}-}\mfrak{c}_m(L)=0
    \quad\tx{for $m=0,1,2,3$}.
    $$
\end{itemize}
\begin{proof}
First of all, a direct computation gives
\begin{equation*}
  -\der_1\bar a_{11}=(\gam+1)\frac{\bar u_1^{\gam}}{\us^{\gam+1}}\bar u_1'.
\end{equation*}
In \eqref{EP-1d-reduced}, we rewrite the equation
$\displaystyle{\bar u_1'=\frac{\bar E \bar u_1^{\gam}}{\bar u_1^{\gam+1}-\us^{\gam+1}}}$
as
$$\bar E=(\bar u_1^{\gam+1}-\us^{\gam+1})\frac{\bar u_1'}{\bar u_1^{\gam}}.$$
Hence one can rewrite $\bar a$ as
\begin{equation*}
  -\bar a=\frac{(\gam \bar u_1^{\gam+1}+\us^{\gam+1})}{\us^{\gam+1}}\cdot \frac{\bar u_1'}{\bar u_1}.
\end{equation*}
Then the statements (a) and (b) easily follow from the Lemma \ref{lemma-1d-full EP}. Therefore, the inequalities stated in (c) easily follow for $m=1,2,3$.

Using the expressions of $(\der_1 \bar a_{11}, \bar a)$ given in the above, we get
\begin{equation*}
  -2\bar a+\der_1\bar a_{11}=\left((\gam-1)\left(\frac{\bar u_1}{\us}\right)^{\gam+1}+2\right)\frac{\bar u_1'}{\bar u_1}.
\end{equation*}
Therefore, the inequality for $m=0$ stated in (c) directly follows from Lemma \ref{lemma-1d-full EP}. Finally, the statement (d) follows from the fact  $\bar u_1'=0$ at $x_1=l_{\rm max}$ as stated in Lemma \ref{lemma-1d-full EP}.
\end{proof}
\end{lemma}

In this section, we assume that the condition \eqref{condition: r in sec4} continues to hold so that, for each $P=(\tphi, \tpsi, \tPsi)\in \iterV(r_2)\times \iterP(r_3)$, the coefficients $(a_{ij}^P, a^P, b_1^P, b_0^P)$ given by Definition \ref{definition:approx coeff and fs} satisfy Lemma \ref{lemma on L_1}. Note that $a_{12}^P=a_{21}^P$ and $a_{22}^P=1$ in $\ol{\Om_L}$.

Let us set
\begin{equation}
\label{definition: L*}
 L_*:=\min\left\{\frac{L+l_{\rm max}}{2}, \frac 32 L\right\}.
\end{equation}
For global higher order derivative estimates of $(v,w)$ in $\Om_L$, we shall introduce an \emph{extension} of the differential operator $\mfrak{L}_1^P$ onto $\Om_{L_*}$ so that the extended operator is strictly elliptic near the boundary $\der\Om_{L_*}\cap\{x_1=L_*\}$ with respect to the first component $v$.

\begin{lemma}\label{corollary-coeff extension}
For any given constant $\om_0>0$, there exits a constant $\delta_1\in(0, \bar{\delta}]$ depending only on the data and $\om_0$ so that if
\begin{equation}\label{condition for r-coeff ext}
  \max\{r_2, r_3\}\le 2\delta_1,
\end{equation}
then, for each $P\in \iterV(r_1)\times \iterP(r_3)$,
one can define an extension $(\alp_{11}^P, \alp_{12}^P, \alp^P)$ of $(a_{11}^P, a_{12}^P, a^P)$ onto $\Om_{L_*}$
with satisfying the following properties:
\begin{itemize}
\item[(a)] $\displaystyle{(\alp_{11}^P, \alp_{12}^P, \alp^P)=(a_{11}^P, a_{12}^P, a^P)}$ in $\Om_L$
\item[(b)] $\displaystyle{\der_2\alp_{11}^P=\der_2 \alp^P=0}$ on $\Gamw$
\item[(c)] $\displaystyle{\alp_{12}^P=\der_2^2\alp_{12}^P=0}$ on $\Gamw$
\item[(d)] There exists a constant $C>0$ such that
    \begin{equation*}
      \|(\alp_{11}^P,\alp_{12}^P, \alp^P)-(\bar{\alp}_{11},0, \bar{\alp})\|_{H^3(\Om_{L_*})}\le C\|(\tphi, \tpsi, \tPsi)\|_{H^4(\Om_L)}
    \end{equation*}
    for
    \begin{equation*}
     (\bar{\alp}_{11}, \bar{\alp}):=(\alp_{11}^{P_0}, \alp^{P_0})\quad\tx{with $P_0=(0,0,0)$}.
    \end{equation*}
\item[(e)] On the boundary $\der\Om_{L_*}\cap\{x_1=L_*\}$, it holds that
\begin{equation*}
\alp_{11}^P\ge \om_0,\quad\tx{and}\quad
\begin{pmatrix}
\alp_{11}^P &\alp_{12}^P\\
\alp_{12}^P&1
\end{pmatrix}\ge \om_0\mathbb{I}_2.
\end{equation*}

\item[(f)]
For the constant $\lambda_{L_*}>0$ from Lemma \ref{lemma-coefficient at bg}(c) corresponding to $L_*$, it holds that
\begin{equation}
\label{essential condi for H-k est}
 -2\alp^P-(2m-1)\der_{1}\alp_{11}^P-4|\der_2 \alp_{12}^P|\ge \frac{\lambda_{L_*}}{2}\quad\tx{in $\ol{\Om_{{L_*}}}$}\,\,\,\tx{for $m=0, 1, 2, 3$.}
\end{equation}

\end{itemize}

\begin{proof}
{\textbf{Step 1.}} For simplicity, let us denote $(a_{11}^P, a_{12}^P, a^P)$ by $(a_{11}, a_{12}, a)$. By using \eqref{def of coeff pert}, we represent $(a_{11}, a)$ as
\begin{equation*}
  a_{11}=\bar a_{11}+da_{11}\quad\tx{and}\quad a=\bar a+da\quad\tx{in $\Om_L$.}
\end{equation*}
First of all, we define an extension of $(\bar a_{11}, \bar a)$ onto $\Om_{L_*}$.
\begin{lemma}
\label{lemma-coefficient extension}
Fix $L\in(0, l_{\rm max})$. For any given constant $\om_0>0$, one can define an extension $(\bar{\alp}_{11}, \bar{\alp})$ of $(\bar a_{11}, \bar a)$ onto $\Om_{L_*}$ to satisfy the following properties:
\begin{itemize}
\item[(i)] $(\bar{\alp}_{11}, \bar{\alp})=(\bar a_{11}, \bar a)$ in $\ol{\Om_L}$
\item[(ii)] $\bar{\alp}_{11},\,\, \bar{\alp}\in C^{\infty}([0,\infty)\times [-1,1])$
\item[(iii)] $\bar{\alp}_{11}=2\om_0>0$ in $[L_*,\infty)\times [-1,1]$
\item[(iv)] $-2\bar{\alp}-(2m-1)\der_{1}\bar{\alp}_{11}\ge \lambda_{L_*}$ in $[0,\infty)\times [-1,1]$ for $m=0, 1, 2, 3$ where $\lambda_{L_*}>0$ is the constant from Lemma \ref{lemma-coefficient at bg}(c) corresponding to $L_*$.
\end{itemize}

\begin{proof}
By Lemma \ref{lemma-coefficient at bg}(c), there exists a constant $\lambda_{L_*}>0$ such that
\begin{equation*}
-2\bar a-(2m-1)\der_1 \bar a_{11}\ge \lambda_{L_*}\quad\tx{in $\ol{\Om_{L_*}}$ for all $m=0,1,2,3$.}
\end{equation*}

Set as $r_0:=\frac 14 (L_*-L)$, and rewrite the above inequality as
\begin{equation*}
  -2\bar a-(2m-1)\der_1 \bar a_{11}\ge \lambda_{L_*}\quad\tx{in $\ol{\Om_{L+4r_0}}$ for all $m=0,1,2,3$}.
\end{equation*}

Fix a cut-off function $\chi^{(1)}\in C^{\infty}(\R)$ so that
\begin{equation*}
\chi^{(1)}(x_1, x_2)=\begin{cases}
1\quad&\mbox{for $x_1\le L+2r_0$}\\
0\quad&\mbox{for $x_1\ge L+3r_0$}
\end{cases}  \quad\tx{and}\quad (\chi^{(1)})'\le 0\,\,\tx{on $\R$}.
\end{equation*}
Given a constant $\om_0>0$, for $\rx=(x_1, x_2)\in [0, \infty)\times [-1,1]$, define $\bar{\alp}_{11}$ by
\begin{equation}
\label{extension of kappa}
  \bar{\alp}_{11}(\rx)=\bar a_{11}(\rx)\chi^{(1)}(x_1)+2\om_0(1-\chi^{(1)}(x_1)).
\end{equation}
Fix another cut-off function $\chi^{(2)}\in C^{\infty}(\R)$ so that
\begin{equation*}
\chi^{(2)}(x_1, x_2)=\begin{cases}
1\quad&\mbox{for $x_1\le L+r_0$}\\
0\quad&\mbox{for $x_1\ge L+2r_0$}
\end{cases}
\quad\tx{and}\quad  (\chi^{(2)})'\le 0\,\,\tx{on $\R$}.
\end{equation*}
For $\rx=(x_1, x_2)\in [0, \infty)\times [-1,1]$, define
\begin{equation*}
  \bar{\alp}(\rx)=\bar a(\rx)-(1-\chi^{(2)}(x_1))\alp_0
\end{equation*}
where $\alp_0>0$ is a constant to be specified later.
Then we have
\begin{equation*}
\begin{split}
  &-2\bar{\alp}-(2m-1)\der_{1}\bar{\alp}_{11}\\
  &=(-2\bar a-(2m-1)\der_{1}\bar a_{11})\chi^{(1)}-2\bar a(1-\chi^{(1)})+2\alp_0(1-\chi^{(2)})
  -(2m-1)(\bar a_{11}-2\om_0)(\chi^{(1)})'\\
  &\ge \lambda_{L_*}\chi^{(1)}+2\alp_0(1-\chi^{(2)})
  -(2m-1)(\bar a_{11}-2\om_0)(\chi^{(1)})'.
  \end{split}
\end{equation*}
From the definitions of $\chi^{(1)}$ and $\chi^{(2)}$, it follows that $\chi^{(1)}+(1-\chi^{(2)})\ge 1$ on $\R$. If $(\chi^{(1)})'(x_1)\neq 0$, then $x_1$ lies on the interval $[L+2r_0, L+3r_0]$ so we have $\displaystyle{(1-\chi^{(2)})(x_1)=0}$.  Moreover, there exists a constant $\beta_0>0$ depending only on $(\gam, \zeta_0, S_0, J, E_0, \om_0)$ so that
\begin{displaymath}
|(\bar a_{11}-2\om_0)(\chi^{(1)})'|\le \beta_0\quad \tx{on $\{x_1\ge 0\}$.}
\end{displaymath}
Hence we have
\begin{equation*}
\begin{split}
-2\bar{\alp}-(2m-1)\der_{1}\bar{\alp}_{11}
 &\ge  \lambda_{L_*}\chi^{(1)}+(2\alp_0-\beta_0)(1-\chi^{(2)})\\
 &\ge \min\{ \lambda_{L_*}, 2\alp_0-\beta_0\}
 \end{split}
\end{equation*}
for all $x_1\ge 0$. Choosing the constant $\alp_0>0$ to satisfy $$\min\{ \lambda_{L_*}, 2\alp_0-\beta_0\}=\lambda_{L_*},$$
the proof of Lemma \ref{lemma-coefficient extension} is completed.
\end{proof}
\end{lemma}

{\textbf{Step 2.}} For a function $\xi$ defined in $\ol{\Om_L}$, define its extension $\mcl{E}\xi$ onto $\ol{\Om_{L_*}}$ by
\begin{equation}
\label{extension onto Om L_*}
  \mcl{E}\xi(\rx)=
  \begin{cases}
  \xi(\rx)&\quad \tx{if $\rx=(x_1, x_2)\in\ol{\Om_L}$}\\
 \displaystyle{ \sum_{j=0}^3 c_j\xi\left(L+\frac{1}{2^j}(L-x_1), x_2\right)}&\quad\tx{if $\rx=(x_1,x_2)\in \ol{\Om_{L_*}}\setminus \ol{\Om_L}$}
  \end{cases}
\end{equation}
where $(c_j)_{j=0}^3\in \R^4$ satisfies
\begin{equation*}
  \sum_{j=0}^3 \left(-\frac{1}{2^j}\right)^kc_j=1\quad\tx{for $k=0, 1, 2, 3$}.
\end{equation*}
Clearly, $\mcl{E}\xi$ is an $H^4$-extension of $\xi$ on $\Om_{L_*}$, and the operator $\mcl{E}:H^4(\Om_L)\rightarrow H^4(\Om_{L_*})$ is linear and bounded. And, the norm of $\mcl{E}$ depends only on $L$. More importantly, the extension is given in $x_1$-direction only so that if $\xi$ satisfies the slip boundary condition $\der_2\xi=0$ on $\Gam_w$, then $\mcl{E}\xi$ satisfies the slip boundary condition
\begin{equation*}
  \der_2(\mcl{E}\xi)=0\quad\tx{on $\der\Om_{L_*}\cap\{|x_2|=1\}$}.
\end{equation*}
For $P=(\tphi, \tpsi, \tPsi)\in \iterV(r_2)\times\iterP(r_3)$, let us set $\mcl{E}P:=(\mcl{E}\tphi, \mcl{E}\tpsi, \mcl{E}\tPsi)$. And, define
\begin{equation*}
  (d\alp_{11}^P, \alp_{12}^P, d\alp^P):=({a}_{11}^{\mcl{E}P}, {a}_{12}^{\mcl{E}P}, {a}^{\mcl{E}P})-(\bar a_{11}, 0, \bar a)\quad\tx{in $\ol{\Om_{L_*}}$}
\end{equation*}
where the term $({a}_{11}^{\mcl{E}P}, {a}_{12}^{\mcl{E}P}, {a}^{\mcl{E}P})$ is given by Definition \ref{definition:coefficients-nonlinear}. Finally, define
\begin{equation}
\label{definition of extended coefficient}
  \alp_{11}^P:=\bar{\alp}_{11}+d\alp_{11}^P,\quad
  \alp_{12}^P:=\mfrak{a}_{12}^{\mcl{E}P},\quad
  \alp^P:=\bar{\alp}+d\alp^P\quad\tx{in $\ol{\Om_{L_*}}$}.
\end{equation}
Then we can complete the proof of Lemma \ref{corollary-coeff extension} by applying Lemma \ref{lemma-coefficient extension} and the generalized Sobolev inequality.

\end{proof}

\end{lemma}

\newcommand \itersetent{\mcl{I}_{\rm ent}(r_1)}
\newcommand \itersetdelta{\mcl{I}_{\rm vor}(2\bar{\delta})\times \mcl{I}_{\rm pot}(2\bar{\delta})}
As we shall see later, the property (f) stated in Lemma \ref{corollary-coeff extension} is the most important ingredient to establish a priori estimates of the solutions $(v,w)$ to the boundary value problem \eqref{lbvp-main general} in $\Om_L$. More precisely, the inequality \eqref{essential condi for H-k est} for $m=0,1,2,3$ is the essential condition to achieve a priori $H^{m+1}$-estimate of $v$ in $\Om_L$.

Fix $P\in\iterseta$. For the differential operator $\mfrak{L}_1^P$ given by Definition \ref{definition:approx coeff and fs}, define
\begin{equation*}
  \mcl{L}^P v:=\mfrak{L}_1^P(v,0)=a_{11}^P\der_{11}v+2a_{12}^P\der_{12}v+\der_{22}v+a^P\der_1 v\quad\tx{in $\Om_L$.}
\end{equation*}
Given a function $f\in L^2(\Om_L)$, suppose that $V\in H^2(\Om_L)$ solves the problem:
\begin{equation*}
  \begin{cases}
\mcl{L}^P V=f\quad\tx{in $\Om_L$},\\
V=0\quad\tx{on $\Gamen$},\quad \der_2V=0\quad\tx{on $\Gamw$}.
  \end{cases}
\end{equation*}
Then we can apply \cite[Theorem 1.1, Chapter 1]{KZ} and the property (f)(with $m=0$) stated in Lemma \ref{corollary-coeff extension} with minor adjustments to obtain the estimate
\begin{equation}
\label{simple H1 estimate of v from single eqn}
  \|V\|_{H^1(\Om_L)}+\|\der_1 V\|_{L^2(\Gamen\cup\Gamex)}\le C\|f\|_{L^2(\Om_L)}.
\end{equation}
Note that there is no restriction on the background momentum density $J$ and the nozzle length $L$ required for this estimate. Simply, the positive acceleration of the background solution stated in Lemma \ref{lemma-1d-full EP} naturally yields the property Lemma \ref{corollary-coeff extension}(f). It is a surprising discovery that the operator $\mcl{L}^P$ satisfies the technical conditions given in \cite[Theorem 1.1 in Chapter 1]{KZ}.

Now we explain how Lemma \ref{corollary-coeff extension}(f) yields the estimate \eqref{simple H1 estimate of v from single eqn} in details. According to the proof of \cite[Theorem 1.1 in Chapter 1]{KZ}, we should start with
\begin{equation*}
\int_{\Om_L} e^{-\mu x_1}\der_1 V\mcl{L}^P V\,d\rx=\int_{\Om_L} e^{-\mu x_1}f\der_1 V\,d\rx
\end{equation*}
for a constant $\mu>0$ to be determined. By integrating the left-hand side by parts, and applying Lemma \ref{lemma on L_1}(h), Lemma \ref{corollary-coeff extension}(b) and (f)(with $m=0$), one can directly check that if the constant $\mu$ is fixed sufficiently small, then it holds that
\begin{equation*}
  \int_{\Om_L} e^{-\mu x_1}\der_1 V\mcl{L}^P V\,d\rx\le -\lambda
  \left(\int_{\Om_L}|DV|^2\,d\rx+\int_{\Gamen\cup\Gamex}(\der_1V)^2\,dx_1\right)
\end{equation*}
for some constant $\lambda>0$.

Note that, however, if we fix $G$ as $G=e^{-\mu x_1}$ in the proof of Proposition \ref{proposition-H1-apriori-estimate}, then we can get the result of \eqref{lower bound of -(I1+I2)}, only for a sufficiently large momentum density $J$ due to the coupling term $T_{\rm mix}$, given by \eqref{definition-Jc}. By choosing the energy weight function $G$ as $G=\bar{\rho}^{\eta}$ with a carefully chosen $\eta>0$, we can establish the main result of this paper for a bigger class of parameters.

\subsection{The well-posedness of the boundary value problem \eqref{lbvp-main general}}
\label{section:lbvp singular perturbation}
For $P=(\tphi, \tpsi, \tPsi)$, we assume that $P$ is fixed in $\iterseta$ throughout \S \ref{section:lbvp singular perturbation}. For simplicity, let us denote all the coefficients $(a_{ij}^P, a^P, b_1^P, b_0^P)$ of $\mfrak{L}_1^P$ by $(a_{ij}, a, b_1, b_0)$.

\subsubsection{Singular perturbation problems and related lemmas}
As explained in \S \ref{subsection-lbvp-preliminary}, we shall establish the existence of a solution to \eqref{lbvp-main general} by solving the following problem.
\begin{problem}
\label{problem:singular perturbation}
Let a constant $\eps>0$ be fixed sufficiently small. Given functions $f_1\in H^3(\Om_L)$ and $f_2\in H^2(\Om_L)$ with satisfying the compatibility conditions \eqref{compatibility conditions for f1 and f2}, find $(v,w)$ that solves the following problem:
\begin{equation}
\label{bvp-sing-pert}
  \begin{split}
  &\begin{cases}
  \eps \der_{111}v+\mfrak{L}_1^P(v, w)=f_1\quad&\tx{in $\Om_L$},\\
  \mfrak{L}_2(v, w)=f_2\quad&\tx{in $\Om_L$},
  \end{cases}\\
  &\begin{cases}
  v=0,\quad \der_1v=0\quad&\tx{on $\Gamen$},\\
  \der_2 v=0\quad&\tx{on $\Gamw$},\\
 \der_{11}v=0\quad&\tx{on $\Gamex$},
  \end{cases}\\
  &\begin{cases}
  \der_1 w=0\quad&\tx{on $\Gamen$},\\
  \der_2w=0\quad&\tx{on $\Gamw$},\\
  w=0\quad&\tx{on $\Gamex$}.
  \end{cases}
  \end{split}
\end{equation}
\end{problem}

\begin{remark}
Since there is a third order differential equation in Problem \ref{problem:singular perturbation}, we add two boundary conditions $\der_1v=0$ on $\Gamen$, and $\der_{11}v=0$ on $\Gamex$ to establish the well-posedness.
\end{remark}

First of all, we prove several technical lemmas.
Throughout Lemmas \ref{lemma:wp of singular pert prob-main}--\ref{viscous-estimate-intermediate1}, $(v,w)$ is assumed to be a smooth solution to Problem \ref{problem:singular perturbation}.

\begin{lemma}
\label{lemma:wp of singular pert prob-main}
Assume that the background solution $(\bar u_1, \bar E)$ and the nozzle length $L$ are fixed to satisfy all the conditions stated in Lemma \ref{proposition-H1-apriori-estimate}. And, assume that the condition \eqref{condition for r-coeff ext} holds.
Then, one can fix  a constant $\bar{\eps}>0$, and further reduce the constant $\delta_1>0$ from \eqref{condition for r-coeff ext} depending only on the data so that whenever the inequality $0<\eps<\bar{\eps}$ and the condition \eqref{condition for r-coeff ext} hold, if $(v,w)$ is a smooth solution to Problem \ref{problem:singular perturbation}, then it satisfies
\begin{equation}
\label{a priori estimate1 of vm and wm}
\begin{split}
  &\sqrt{\eps}\|\der_{11}v\|_{L^2(\Om_L)}
  +\|\der_1v\|_{L^2(\Gamen)}+\|Dv\|_{L^2(\Gamex)}
  +\|v\|_{H^1(\Om_L)}+\|w\|_{H^1(\Om_L)}
  \\
  &\le C\left(\|f_1\|_{L^2(\Om_L)}+\|f_2\|_{L^2(\Om_L)}\right).
  \end{split}
\end{equation}

\end{lemma}

\begin{lemma}
\label{lemma-a priori H2 estimate of wm}
Under the same assumptions as Lemma \ref{lemma:wp of singular pert prob-main}, it holds that
\begin{equation}
  \|w\|_{H^2(\Om_L)}\le C(\|f_1\|_{L^2(\Om_L)}+\|f_2\|_{L^2(\Om_L)})
\end{equation}
for some constant $C>0$ depending only on the data.
\end{lemma}

\begin{lemma}
\label{lemma for pre H2 estimate of vm, part 1}
Let $\ls$ be the constant given from Lemma \ref{lemma-1d-full EP}. Under the same assumptions as Lemma \ref{lemma:wp of singular pert prob-main}, for any given constant $t\in(0, \frac{\ls}{8}]$,
there exists a constant $C_t>0$ depending only on the data and $t$ to satisfy
\begin{equation}
\label{viscous-estimate-intermediate1}
  \|D\der_1 v\|_{L^2(\Om_L\cap \{2t<x_1<\frac{l_s}{2}-2t\})}
  \le C_t\left(\|f_1\|_{L^2(\Om_L)}+\|f_2\|_{L^2(\Om_L)}\right).
\end{equation}

\end{lemma}

\begin{lemma}
\label{lemma for pre H2 estimate of vm, part 2}
One can reduce the constant $\delta_1>0$ further from the one given in Lemma \ref{corollary-coeff extension} with depending only on the data so that if the inequality $0<\eps\le \bar{\eps}$ and the condition \eqref{condition for r-coeff ext} hold, then it holds that
\begin{equation}
\label{viscous-estimate-intermediate2}
\begin{split}
  &\sqrt{\eps} \|\der_{111}v\|_{L^2(\Om_L\cap\{x_1>\frac{l_s}{4}\})}
  +\|\der_{12}v\|_{L^2(\Gamex)}
  +\|D\der_1v\|_{L^2(\Om_L\cap\{x_1>\frac{\ls}{4}\})}
  \\
  &\le C\left(\|f_1\|_{H^1(\Om_L)}+\|f_2\|_{L^2(\Om_L)}\right).
\end{split}
\end{equation}
\end{lemma}

Once Lemmas \ref{lemma:wp of singular pert prob-main}--\ref{lemma for pre H2 estimate of vm, part 2} are proved, then we apply these lemmas and the method of Galerkin's approximations to prove the unique existence of a {\emph{weak solution}} to the boundary value problem \eqref{lbvp-main general}.

\begin{definition}
\label{definition of weak solution}
For $v,w,V,W\in H^1(\Om_L)$, define two bilinear operators $\mcl{B}_1^P$ and $\mcl{B}_2$ by
\begin{align*}
\mcl{B}_1^P[(v,w),V]:=&
\int_{\Om_L}-\left(a_{11}\der_1v\der_1V+a_{12}(\der_1v\der_2V
+\der_2v\der_1V)+\der_2v\der_2V+a\der_1vV\right)\\
&\phantom{aa}-(\der_1a_{11}\der_1v+\der_2a_{12}\der_1v+\der_1a_{12}\der_2 v)V+(b_1\der_1 w+b_0w)V\,d\rx,\\
\mcl{B}_2[(v,w),W]:=&-\int_{\Om_L} \nabla w\cdot \nabla W+(\frac{1}{\gam S_0\bar{\rho}^{\gam-1}}w-\frac{\bar u_1}{\gam S_0\bar{\rho}^{\gam-1}}\der_1 v)W\,d\rx.
\end{align*}

It is said that $(v,w)\in H^1(\Om_L)\times H^1(\Om_L)$ is {\emph{a weak solution}} to the boundary value problem \eqref{lbvp-main general} if it satisfies
\begin{equation}
\label{weak-viscous-eqns}
    \begin{cases}
 \mcl{B}_1^P[(v,w),V]=\int_{\Om_L} f_1 V\,d\rx,\\
\mcl{B}_2[(v,w),W]=\int_{\Om_L} f_2 W\,d\rx
\end{cases}
\end{equation}
  for any test functions $V, W\in C^{\infty}(\ol{\Om_L})$ with $V$ vanishing near $\Gamen\cup \Gamex$, and $W$ vanishing near $\Gamex$.
\end{definition}

\begin{proposition}
\label{proposition-wp of bvp with approx coeff}
Assume that the background solution $(\bar u_1, \bar E)$ and the nozzle length $L$ are fixed to satisfy all the conditions stated in Lemma \ref{proposition-H1-apriori-estimate}. For the constant $\delta_1>0$ given in Lemma \ref{lemma for pre H2 estimate of vm, part 2}, suppose that the condition
\begin{equation}\label{condition for r-final}
  \max\{r_2, r_3\}\le \delta_1
\end{equation}
holds.
Then, for any given $P=(\tphi, \tpsi, \tPsi)\in \iterseta$, the associated linear boundary value problem \eqref{lbvp-main general} has a unique weak solution $(v,w)\in H^1(\Om_L)\times H^2(\Om_L)$ in the sense of Definition \ref{definition of weak solution}. And, the weak solution satisfies the estimate
\begin{equation}
\label{estimate 1 of v and w}
\|v\|_{H^1(\Om_L)}+\|w\|_{H^2(\Om_L)}\le C\left(\|f_1\|_{L^2(\Om_L)}+\|f_2\|_{H^1(\Om_L)}\right)
\end{equation}
for a constant $C>0$ fixed depending only on the data.
Furthermore, for any constant $r\in(0, L)$, one can fix a constant $C_r>0$ depending only on the data and $r$ to satisfy the estimate
\begin{equation}
\label{estimate 2 of v and w}
\begin{split}
&\|v\|_{H^2(\Om_L\cap\{x_1<L-r\})}
+\|w\|_{H^3(\Om_L\cap\{x_1<L-r\})}\\
&\le C_r\left(\|f_1\|_{H^1(\Om_L)}+\|f_2\|_{H^1(\Om_L)}\right).
\end{split}
\end{equation}
Therefore, $(v,w)$ is a strong solution to the boundary value problem \eqref{lbvp-main general}.
\end{proposition}

\subsubsection{Proofs of Lemmas \ref{lemma:wp of singular pert prob-main}--\ref{lemma for pre H2 estimate of vm, part 2}}

\begin{proof}[Proof of Lemma \ref{lemma:wp of singular pert prob-main}]
Fix a constant $\eps>0$, and suppose that $(v,w)$ is a smooth solution to Problem \ref{problem:singular perturbation}. For a function $G=G(x_1)$, 
\begin{equation}
\label{eqn-s1}
  \eps\int_{\Om_L} G\der_{111}v\der_1 v\,d\rx+\mfrak{J}(v,w)=\int_{\Om_L} Gf_1\der_1 v+f_2 w\,d\rx
\end{equation}
for
\begin{equation*}
  \mfrak{J}(v,w)=\int_{\Om_L}G\der_1 v\mfrak{L}_1(v,w)+w\mfrak{L}_2(v,w)\,d\rx.
\end{equation*}
Integrating by parts twice with using the boundary conditions $\der_1 v=0$ on $\Gamen$ and $\der_{11}v=0$ on $\Gamex$ for $v$ in \eqref{bvp-sing-pert} yields
\begin{equation*}
  \int_{\Om_L}G\der_{111}v\der_1 v\,d\rx
  =-\int_{\Gamex}\frac{G'}{2}(\der_1 v)^2\,dx_2+\int_{\Om_L} \frac{G''}{2}(\der_1 v)^2-G(\der_{11}v)^2\,d\rx.
\end{equation*}
If we fix $G(x_1)$ as \eqref{definition of G} for $\eta$ given by \eqref{choice of eta}, then one has
\begin{equation}
\label{eqn-s2}
  \eps\int_{\Om_L} G\der_{111}v\der_1 v\,d\rx\le -k_0\eps \int_{\Om_L}(\der_{11}v)^2\,d\rx+k_1\eps\int_{\Gamex}(\der_1 v)^2\,d\rx+k_2\eps\int_{\Om_L}(\der_1 v)^2\,d\rx
\end{equation}
for
\begin{equation*}
  k_0=\inf_{\Om_L}G,\quad k_1=\frac 12\sup_{\Om_L}|G'|,\quad k_2=\frac 12\sup_{\Om_L}|G''|.
\end{equation*}
Note that the constants $k_0$, $k_1$ and $k_2$ are fixed depending only on the data. In particular, $k_0$ is a strictly positive constant.

Next, we use the estimates \eqref{lower bound of -(I1+I2)} and \eqref{weak estimate of approx problem remainder}, given in the proof of Lemma \ref{proposition-H1-apriori-estimate}, to get
\begin{equation}
\label{eqn-s3}
\begin{split}
  \mfrak{J}(v,w)\le
  &-k_3\int_{\Om_L}|Dv|^2+|Dw|^2+w^2\,d\rx-k_4\left(\int_{\Gamen} (\der_1 v)^2\,dx_2+\int_{\Gamex} |Dv|^2\,dx_2\right)
  \end{split}
  \end{equation}
  for
  \begin{equation*}
    k_3=\lambda_0 -C_*(r_2+r_3),\quad k_4=\lambda_{\rm bd} -C_*(r_2+r_3).
  \end{equation*}
  Then the estimate \eqref{a priori estimate1 of vm and wm} follows from \eqref{eqn-s1}--\eqref{eqn-s3} and the Cauchy-Schwarz inequality.

  \end{proof}

\begin{proof}[Proof of Lemma \ref{lemma-a priori H2 estimate of wm}]

According to Definition \ref{definition:approx coeff and fs}, the equation $\mfrak{L}_2(v,w)=f_2$ can be rewritten as
\begin{equation*}
  \Delta w-{c}_0w=f_2+{c}_1\der_1 v\quad\tx{in $\Om_L$}.
\end{equation*}
Note that the coefficients ${c}_0$ and ${c}_1$ are smooth and independent of $x_2$.
By Lemma \ref{lemma:wp of singular pert prob-main}, one has
\begin{equation*}
  \|f_2+{c}_1\der_1 v\|_{L^2(\Om_L)}\le C\left(\|f_1\|_{L^2(\Om_L)}+\|f_2\|_{L^2(\Om_L)}\right).
\end{equation*}
Hence one can easily prove Lemma \ref{lemma-a priori H2 estimate of wm} by applying \cite[Theorems 8.8 and 8.12]{GilbargTrudinger} and the method of reflections.

\end{proof}

Rewrite the equation $\displaystyle{\eps\der_{111}v+\mfrak{L}_1^P(v,w)=f_1}$ stated in \eqref{bvp-sing-pert} as
\begin{equation}
\label{sing pert eqn}
  \mcl{M}_{\eps}v=F_1\quad\tx{in $\Om_L$}
\end{equation}
for
\begin{equation}
\label{definition-bootstrap}
  \begin{split}
  &\mcl{M}_{\eps}v:=\eps\der_{111}v+a_{11}\der_{11}v+2a_{12}\der_{12}v+\der_{22}v+a\der_1 v,\\
\tx{and}\quad  &F_1:=f_1-b_1\der_1w-b_0w.
  \end{split}
\end{equation}

\begin{proof}[Proof of Lemma \ref{lemma for pre H2 estimate of vm, part 1}] In this proof, the ellipticity of the operator $\mcl{L}_1^P$ in $\Om_L\cap\{x_1\le\frac{\ls}{2}\}$ plays an important role.

For a fixed constant $t\in(0, \frac{l_s}{8})$, define a cut-off function $\chi\in C^{\infty}(\R)$ satisfying the following properties:
\begin{equation}
\label{property of cut-off function in singular perturb.}
  \chi(x_1)=\begin{cases}
  1\quad&\mbox{for $x_1\in [2t, \frac{\ls}{2}-2t]$},\\
  0\quad&\mbox{for $x_1\in (-\infty, t]\cup[\frac{\ls}{2}-t, \infty)$},
  \end{cases}\quad 0\le \chi\le 1\,\,\tx{on $\R$}.
\end{equation}
Such a function $\chi$ can be fixed so that, for each $k\in \mathbb{N}$, it holds that
\begin{equation*}
\|\chi\|_{C^k(\R)}\le C_k\left(1+\frac{1}{t^k}\right)
\end{equation*}
for some constant $C_k>0$ depending only on $k$.
Note that
\begin{equation}
\label{integral for vm with viscosity1}
  \int_{\Om_L} \chi^2\der_{11}v \mcl{M}_{\eps}v\,d\rx=\int_{\Om_L} F_1\chi^2\der_{11}v\,d\rx.
\end{equation}
By integrating by parts  with using the boundary conditions
$\der_2 v=0$ on $\Gamw$, and $\chi=0$ on $\Gamen\cup\Gamex$, we obtain that
\begin{equation*}
I_1+I_2=\int_{\Om_L} F_1\chi^2\der_{11}v\,d\rx
\end{equation*}
for
\begin{equation*}
  \begin{split}
  I_1&:=\int_{\Om_L} \left(a_{11}(\der_{11}v)^2+2a_{12}\der_{11}v\der_{12}v+(\der_{12}v)^2\right)\chi^2\,d\rx,\\
  I_2&:=\int_{\Om_L}-\eps\chi\chi' (\der_{11}v)^2+a\chi^2\der_1 v\der_{11}v-2\chi\chi'\der_2v\der_{12}v  \,d\rx.
  \end{split}
\end{equation*}
It follows from the statement (${\tx g}_2$) of Lemma  \ref{lemma on L_1} that there exists a constant $\kappa_0>0$ to satisfy
\begin{equation*}
  I_1\ge \kappa_0\int_{\Om_L}|D\der_1 v|^2\chi^2\,d\rx.
\end{equation*}
By applying Lemma \ref{lemma:wp of singular pert prob-main}, the term $I_2$ can be estiamtes as
\begin{equation*}
  |I_2|\le \frac{\kappa_0}{4}\int_{\Om_L} |D\der_1 v|^2\chi^2\,d\rx+C_t\left(\|f_1\|_{L^2(\Om_L)}+\|f_2\|_{L^2(\Om_L)}\right)^2
\end{equation*}
for some constant $C_t>0$ fixed depending only on the data and $t$. So we obtain
\begin{equation*}
{\tx{LHS of \eqref{integral for vm with viscosity1}}}\ge
\frac 34\kappa_0\int_{\Om_L} |D\der_1 v|^2\chi^2\,d\rx-C_t\left(\|f_1\|_{L^2(\Om_L)}+\|f_2\|_{L^2(\Om_L)}\right)^2.
\end{equation*}
Applying Lemma \ref{lemma:wp of singular pert prob-main} also yields
\begin{equation*}
  \left|\int_{\Om_L} F_{1}\der_{11}v \chi^2\, d\rx\right|\le \frac{\kappa_0}{4}\int_{\Om_L} |D\der_1 v|^2\chi^2\,d\rx+C\left(\|f_1\|_{L^2(\Om_L)}+\|f_2\|_{L^2(\Om_L)}\right)^2.
\end{equation*}
Finally, the estimate \eqref{viscous-estimate-intermediate1} is obtained by combining the two previous estimates.
\end{proof}

\begin{proof}[Proof of Lemma \ref{lemma for pre H2 estimate of vm, part 2}]
{\textbf{Step 1.}} Let us fix a smooth cut-off function $\zeta(x_1)$ satisfying the following properties:
\begin{equation*}
 \zeta(x_1)=\begin{cases}
  0\quad&\mbox{for}\quad x_1\le \frac{l_s}{8},\\
  1\quad&\mbox{for}\quad x_1\ge \frac{\ls}{4},
  \end{cases}\quad
  0\le \zeta\le 1,\quad
  \tx{and} \quad \zeta'\ge 0\quad\tx{on $\R$}.
\end{equation*}
It is clear that, for each $k\in \mathbb{N}$, there exists a constant $C_k>0$ depending only on $k$ to satisfy
\begin{equation*}
\|\zeta\|_{C^k(\R)}\le C_k\left(1+\frac{1}{\ls^k}\right).
\end{equation*}
Multiplying the both sides of \eqref{sing pert eqn} by $\zeta^2 \der_{111}v$, we have
\begin{equation}
\label{estimate-H2-viscous}
  \int_{\Om_L} \mcl{M}_{\eps}v
  \der_{111}v\,d\rx
  =\int_{\Om_L} F_{1}\zeta^2 \der_{111}v \,d\rx.
\end{equation}
By an integration by parts with using the conditions
\begin{equation*}
  \zeta=0\,\,\tx{on $\Gamen$},\quad \der_{11}v=0\,\,\tx{on $\Gamex$},\quad a_{12}=0\,\,\tx{on $\Gamw$},
\end{equation*}
we can directly check that
\begin{equation}
\label{estimate3}
{\tx{LHS of \eqref{estimate-H2-viscous}}}
=I_1+I_2+I_3
\end{equation}
for
\begin{align*}
&I_1=\int_{\Om_L}\eps (\der_{111}v)^2\zeta^2
-\left(\frac 12\der_1a_{11}+a-2\der_2 a_{12}\right)(\der_{11}v)^2\zeta^2\,d\rx+\int_{\Gamex} \frac{(\der_{12}v)^2}{2}\,dx_2,\\
&I_2=-\int_{\Om_L}(2\der_1 a_{12}\der_{11}v\der_{12}v+\der_1 a\der_1v\der_{11}v)\zeta^2\,d\rx,\\
&I_3=-\int_{\Om_L}\left(a_{11}(\der_{11}v)^2+4a_{12}\der_{11}v\der_{12}v+
(\der_{12}v)^2+2\der_{11}v\der_{22}v+2a\der_1v\der_{11}v\right)\zeta\zeta'\,d\rx.
\end{align*}
From Lemma \ref{corollary-coeff extension}(f) with $m=1$, it follows that
\begin{equation}
\label{estimate4}
  I_1\ge \int_{\Om_L}\left(\eps (\der_{111}v)^2+\frac{\lambda_{L_*}}{4} (\der_{11}v)^2\right)\zeta^2 \,d\rx+\int_{\Gamex} \frac{(\der_{12}v)^2}{2}\,dx_2.
\end{equation}
By applying Lemma \ref{corollary-coeff extension}(d), we can estimate $I_2$ as
\begin{equation*}
  |I_2|\le C_*\int_{\Om_L}\left((r_2+r_3+\epsilon)(\der_{11}v)^2+(r_2+r_3)(\der_{12}v)^2
  +\frac{1}{\epsilon}(\der_1 v)^2\right)\zeta^2\,d\rx
\end{equation*}
for any constant $\epsilon>0$. Since $\displaystyle{{\rm spt}\zeta'\subset \left[\frac{\ls}{8}, \frac{\ls}{4}\right]}$, we can apply Lemma \ref{lemma for pre H2 estimate of vm, part 1} to estimate $I_3$ as
\begin{equation}
\label{estimate5}
  |I_3|\le C\left(\|f_1\|_{L^2(\Om_L)}+\|f_2\|_{L^2(\Om_L)}\right)^2.
\end{equation}
If two constants $\delta_1$ and $\epsilon$ are chosen to satisfy
\begin{equation}
\label{condition for delta-1}
  C_*\max\{\delta_1,\epsilon\}\le \frac{\lambda_{L_*}}{16},
\end{equation}
and the condition \eqref{condition for r-coeff ext} holds, then we get
\begin{equation}
\label{final estimate in step 2}
\begin{split}
  {\tx{LHS of \eqref{estimate-H2-viscous}}}
  \ge &
  \int_{\Om_L}\left(\eps (\der_{111}v)^2+\frac{\lambda_{L_*}}{16} (\der_{11}v)^2-C_*\delta_1(\der_{12}v)^2\right)\zeta^2\,d\rx\\
  &+\int_{\Gamex} \frac{(\der_{12}v)^2}{2}\,dx_2-C\left(\|f_1\|_{L^2(\Om_L)}+\|f_2\|_{L^2(\Om_L)}\right)^2.
\end{split}
\end{equation}

In order to estimate the right-hand side of \eqref{estimate-H2-viscous}, we integrate by parts with using the conditions $\zeta=0$ on $\Gamen$, and $\der_{11}v=0$ on $\Gamex$. Then, we get
\begin{equation*}
 \tx{RHS of \eqref{estimate-H2-viscous}}=-\int_{\Om_L} \der_1F_1\der_{11}v\zeta^2+2F_1\der_{11}v\zeta\zeta'\,d\rx.
\end{equation*}
So we can apply Lemmas \ref{lemma-a priori H2 estimate of wm} and \ref{lemma for pre H2 estimate of vm, part 1} to obtain the estimate
\begin{equation*}
  |{\tx{RHS of \eqref{estimate-H2-viscous}}}|\le \frac{\lambda_{L_*}}{32} \int_{\Om_L}(\der_{11}v)^2\zeta^2\,d\rx
  +C\left(\|f_1\|_{H^1(\Om_L)}+\|f_2\|_{L^2(\Om_L)}\right)^2.
\end{equation*}
And, we combine this estimate with \eqref{final estimate in step 2} to get
\begin{equation}
\label{pre-H2 estimate}
\begin{split}
  &\int_{\Om_L}\left(\eps (\der_{111}v)^2+\frac{\lambda_{L_*}}{32} (\der_{11}v)^2\right)\zeta^2\,d\rx+\int_{\Gamex} \frac{(\der_{12}v)^2}{2}\,dx_2\\
&  \le
  C_*\delta_1\int_{\Om_L}(\der_{12}v)^2\zeta^2\,d\rx+
  C\left(\|f_1\|_{H^1(\Om_L)}+\|f_2\|_{L^2(\Om_L)}\right)^2.
  \end{split}
\end{equation}

{\textbf{Step 2.}} Since $\|a_{12}\|_{C^0(\ol{\Om_L})}$ is small by Lemma \ref{corollary-coeff extension}(d), we can fix a constant $\mu>0$ so that the matrix $\displaystyle{{\mathbb{A}}:=\begin{pmatrix}\mu&a_{12}\\ a_{12}&1\end{pmatrix}}$ satisfies
\begin{equation}
\label{positivity of A}
  \mathbb{A}\ge \frac{\mu}{2}\mathbb{I}_2\quad\tx{in $\ol{\Om_L}$.}
\end{equation}
Rewrite
$\displaystyle{
  \int_{\Om_L} \zeta^2\der_{11}v \mcl{M}_{\eps}v\,d\rx=\int_{\Om_L} F_1\zeta^2\der_{11}v\,d\rx}$
as
\begin{equation}
\label{integral-for-visous-sol}
\begin{split}
 & \int_{\Om_L} (\eps \der_{111}v+
 \mu\der_{11}v+2a_{12}\der_{12}v+\der_{22}v+a\der_1 v)\der_{11}v\zeta^2\,d\rx\\
& =\int_{\Om_L} (F_1+(\mu-a_{11})\der_{11}v)\der_{11}v\zeta^2\,d\rx.
 \end{split}
\end{equation}
Integrating by parts twice gives
\begin{equation*}
\tx{LHS of \eqref{integral-for-visous-sol}}=J_1+J_2,
\end{equation*}
for
\begin{equation*}
  \begin{split}
  J_1&=\int_{\Om_L}\left(\mu(\der_{11}v)^2
  +2a_{12}\der_{12}v\der_{11}v+(\der_{12}v)^2\right)\zeta^2\,d\rx,\\
  J_2&=\int_{\Om_L} \left(2\der_2 v\der_{12}v-\eps (\der_{11}v)^2\right)\zeta\zeta'\,d\rx-\int_{\Gamex} \der_2 v\der_{12}v\,dx_2.
  \end{split}
\end{equation*}
Then, it directly follows from \eqref{positivity of A} that
\begin{equation*}
  J_1\ge \frac{\mu}{2}\int_{\Om_L}|D\der_1 v|^2\zeta^2\,d\rx.
\end{equation*}
By Lemma \ref{lemma:wp of singular pert prob-main} and the estimate \eqref{pre-H2 estimate}, we have
\begin{align*}
&|J_2|\le C\left(\|f_1\|_{L^2(\Om_L)}+\|f_2\|_{L^2(\Om_L)}\right)^2,\\
\tx{and}\quad&|{\tx{RHS of \eqref{integral-for-visous-sol}}}|\le C\left(\delta_1\int_{\Om_L}(\der_{12}v)^2\zeta^2\,d\rx
  +\left(\|f_1\|_{L^2(\Om_L)}+\|f_2\|_{L^2(\Om_L)}\right)^2\right).
\end{align*}
Now we combine all the estimates given in the above to get
\begin{equation*}
 \frac{\mu}{2}\int_{\Om_L}|D\der_1 v|^2\zeta^2\,d\rx\le  \hat{C}\left(\delta_1\int_{\Om_L}(\der_{12}v)^2\zeta^2\,d\rx
  +\left(\|f_1\|_{L^2(\Om_L)}+\|f_2\|_{L^2(\Om_L)}\right)^2\right)
\end{equation*}
for some constant $\hat{C}>0$ fixed depending only on the data.
Reducing $\delta_1$ to satisfy \eqref{condition for delta-1} and the inequality
$$
\hat C\delta_1\le \frac{\mu}{4}
$$
yields the estimate \eqref{viscous-estimate-intermediate2}. Hence the proof of the lemma is completed.
\end{proof}

\subsubsection{Proof of Proposition \ref{proposition-wp of bvp with approx coeff}}
The main idea is to find a solution  $\{(v^{(\eps)}, w^{(\eps)})\}_{\eps>0}$ to Problem \ref{problem:singular perturbation} for any small $\eps>0$, then take the limit as $\eps$ tends to $0+$. Most importantly, we show that the limit is a weak solution to the boundary value problem \eqref{lbvp-main general} in the sense of Definition \ref{definition of weak solution}.
\medskip

{\textbf{Step 1.}} Fix $\eps\in(0,\bar{\eps}]$ and $P=(\tphi, \tpsi, \tPsi)\in \iterseta$. And, suppose that two constants $r_2$ and $r_3$ satisfy the condition \eqref{condition for r-coeff ext} for the constant $\delta_1$ given from Lemma \ref{lemma for pre H2 estimate of vm, part 2}. In addition, let us assume that
\begin{equation}\label{partial smoothness}
(\tphi, \tpsi, \tPsi, f_1, f_2)(\cdot, x_2)\in [C^{\infty}([0,L])]^5\quad\tx{for all $x_2\in[-1,1]$}.
\end{equation}
\smallskip

{\textbf{Step 2.}}{\emph{Galerkin's approximations}}: Define
\begin{equation*}
  \Gam:=\{x_2\in \R: |x_2|<1\},
\end{equation*}
and $\langle \cdot, \cdot\rangle$ to be the standard scalar product in $L^2(\Gam)$, that is,
\begin{equation*}
  \langle \xi, \eta \rangle:=\int_{\Gam} \xi(x_2)\eta(x_2)\,dx_2.
\end{equation*}
Next, consider an eigenvalue problem:
\begin{equation}
\label{evp-for galerkin}
  -\eta''=\lambda\eta\,\,\tx{on $\Gam$}, \quad
  \eta'=0\,\,\tx{on $\der\Gam=\{\pm 1\}$}.
\end{equation}
Let $\mfrak{E}:=\{\eta_k\}_{k=0}^{\infty}$ be the set of all eigenfunctions of \eqref{evp-for galerkin}. One can take the set $\mfrak{E}$ with satisfying the following properties:
\begin{itemize}
\item[(i)] For each $k=0,1,2,\cdots$, let $\lambda_k(\ge 0)$ be the eigenvalue associated with $\eta_k$. Then, it holds that
    \begin{equation*}
      0=\lambda_0<\lambda_1<\lambda_2<\cdots \rightarrow \infty.
    \end{equation*}
\item[(ii)] The set $\mfrak{E}$ forms the orthonormal basis of $L^2(\Gam)$.
\item[(iii)] The set $\mfrak{E}$ forms an orthogonal basis of $H^m(\Gam)$ for each of $m=1$ and 2.
\end{itemize}
In fact, $\mfrak{E}=\{\cos k\pi x_2:k=0,1,2\cdots\}$ is such a set.

Fix $m\in \mathbb{N}$. As an $m$-dimensional approximation of a solution $(v,w)$ to Problem \ref{problem:singular perturbation}, let us set
\begin{equation}
\label{definition of vm and wm}
\begin{split}
 v_m(x_1, x_2):=\sum_{j=0}^m \vartheta_j(x_1) \eta_j(x_2),\quad
 w_m(x_1, x_2):=\sum_{j=0}^m \Theta_j(x_1)\eta_j(x_2)
  \end{split}
\end{equation}
for $\rx=(x_1, x_2)\in \Om_L$.
 We shall determine $\vartheta_j$ and $\Theta_j$ for $j=0,1,\cdots, m$ so that $(v_m, w_m)$ solves the following problem:
\begin{equation}
\label{galerkin-eqns}
\begin{cases}
  \langle \eps \der_{111}v_m+\mfrak{L}_1^P(v_m, w_m), \eta_k\rangle=\langle f_1, \eta_k \rangle\\
  \langle \mfrak{L}_2(v_m, w_m), \eta_k\rangle=\langle f_2, \eta_k\rangle
  \end{cases}\quad\tx{for $0<x_1<L$}
\end{equation}
for all $k=0,1,...,m$,
\begin{equation}
\label{galerkin-bcs}
\begin{split}
  &\begin{cases}
  v_m=0, \quad \der_1 v_m=0\quad&\tx{on $\Gamen$},\\
\der_{11}v_m=0\quad&\tx{on $\Gamex$},
  \end{cases}\\
  &\begin{cases}
  \der_1 w_m=0\quad&\tx{on $\Gamen$},\\
  w_m=0\quad&\tx{on $\Gamex$}.
  \end{cases}
  \end{split}
\end{equation}

\begin{lemma}\label{lemma-wp of Galerking approx}
Assume that the background solution $(\bar u_1, \bar E)$ and the nozzle length $L$ are fixed to satisfy all the conditions stated in Lemma \ref{proposition-H1-apriori-estimate}. For the constant $\delta_1>0$ given in Lemma \ref{lemma for pre H2 estimate of vm, part 2}, suppose that the condition \eqref{condition for r-coeff ext} holds. And, assume that $\eps\in(0, \bar{\eps}]$ for the constant $\bar{\eps}>0$ given from Lemma \ref{lemma:wp of singular pert prob-main}. For a fixed $P=(\tphi, \tpsi, \tPsi)\in \iterseta$, assume that the condition \eqref{partial smoothness} is satisfied. Then, for each $m\in \mathbb{N}$, the problem of \eqref{galerkin-eqns} and \eqref{galerkin-bcs} has a unique smooth solution $(v_m, w_m)$ in the form of \eqref{definition of vm and wm}. Moreover, there exists a constant $C>0$ depending only on the data, and for each $t\in(0, \frac{\ls}{8}]$, there exists a constant $C_t>0$ depending only on the data and $t$ so that the following estimates hold:
\begin{equation}
\label{galerkin estimate-1}
\begin{split}
  &\sqrt{\eps}\|\der_{11}v_m\|_{L^2(\Om_L)}
  +\|\der_1v_m\|_{L^2(\Gamen)}+\|Dv_m\|_{L^2(\Gamex)}
  +\|v_m\|_{H^1(\Om_L)}
  \\
  &\le C\left(\|f_1\|_{L^2(\Om_L)}+\|f_2\|_{L^2(\Om_L)}\right),
  \end{split}
\end{equation}

\begin{equation}
\label{galerkin estimate-2}
\begin{split}
  \|w_m\|_{H^2(\Om_L)}\le C\left(\|f_1\|_{L^2(\Om_L)}+\|f_2\|_{L^2(\Om_L)}\right),
  \end{split}
\end{equation}

\begin{equation}
\label{galerkin estimate-3}
  \|D\der_1 v_m\|_{L^2(\Om_L\cap \{2t<x_1<\frac{l_s}{2}-2t\})}
  \le C_t\left(\|f_1\|_{L^2(\Om_L)}+\|f_2\|_{L^2(\Om_L)}\right),
\end{equation}

\begin{equation}
\label{galerkin estimate-4}
\begin{split}
  &\sqrt{\eps} \|\der_{111}v_m\|_{L^2(\Om_L\cap\{x_1>\frac{l_s}{4}\})}
  +\|\der_{12}v_m\|_{L^2(\Gamex)}
  +\|D\der_1v_m\|_{L^2(\Om_L\cap\{x_1>\frac{\ls}{4}\})}
  \\
  &\le C\left(\|f_1\|_{H^1(\Om_L)}+\|f_2\|_{L^2(\Om_L)}\right).
\end{split}
\end{equation}

\end{lemma}

This lemma can be proved by applying the Fredholm alternative theorem, the Arzel\`{a}-Ascoli theorem and Lemmas \ref{lemma:wp of singular pert prob-main}--\ref{lemma for pre H2 estimate of vm, part 2}. We provide a detailed proof after we complete the proof of Proposition \ref{proposition-wp of bvp with approx coeff}.
\medskip

\medskip

{\textbf{Step 3.}} Given two test functions $V\in C^{\infty}(\ol{\Om_L})$ and $W\in C^{\infty}(\ol{\Om_L})$ with $V$ vanishing near $\Gamen\cup\Gamex$ and $W$ vanishing near $\Gamex$, define
 \begin{equation}
 \begin{split}
      &V_m(x_1, x_2):=\sum_{j=0}^m \langle V(x_1,\cdot),\eta_j\rangle \eta_j(x_2),\quad
      W_m(x_1, x_2):=\sum_{j=0}^m \langle W(x_1,\cdot),\eta_j\rangle \eta_j(x_2).
       \end{split}
    \end{equation}
By integrating by parts, one can directly check that
 \begin{equation*}
    \begin{cases}
    \begin{split}
\int_{\Om_L}&\eps\der_1v_m\der_{11}V_m\,d\rx+\mcl{B}_1^P[(v_m, w_m), V_m]
=\int_{\Om_L} f_1 V_m\,d\rx\end{split},\\
\begin{split}
\mcl{B}_2[(v_m, w_m), W_m]=\int_{\Om_L} f_2 W_m\,d\rx
\end{split}
\end{cases}
    \end{equation*}
for $\mcl{B}_1^P$ and $\mcl{B}_2$ given in Definition \ref{definition of weak solution}. Since the sequence $\{(v_m, w_m)\}$ satisfies the uniform estimates \eqref{galerkin estimate-1}--\eqref{galerkin estimate-4}, it has a subsequence, that we shall still denote as $\{(v_m, w_m)\}$, and there exists $(v^{\eps},w^{\eps})\in H^1(\Om_L)\cap H^2(\Om_L)$ so that
\begin{itemize}
\item[-] $(v_m, w_m)$ converges to $(v^{\eps},w^{\eps})$ weakly in $H^1(\Om_L)\times H^2(\Om_L)$,
\item[-] and that $\der_{11}v_m$ converges to $\der_{11}v^{\eps}$ weakly in $L^2(\Om_L)$.
\end{itemize}
Then, the weak limit $(v^{\eps},w^{\eps})$ satisfies that
 \begin{equation*}
    \begin{cases}
    \begin{split}
\int_{\Om_L}&\eps\der_1v^{\eps}\der_{11}V\,d\rx+\mcl{B}_1^P[(v^{\eps},w^{\eps}), V]
=\int_{\Om_L} f_1 V\,d\rx\end{split},\\
\begin{split}
\mcl{B}_2[(v^{\eps},w^{\eps}), W]=\int_{\Om_L} f_2 W\,d\rx.
\end{split}
\end{cases}
    \end{equation*}
Moreover, it follows from \eqref{galerkin estimate-1}--\eqref{galerkin estimate-4} that $(v^{\eps},w^{\eps})$ satisfies all the estimates \eqref{a priori estimate1 of vm and wm}--\eqref{viscous-estimate-intermediate2}, in which the estimate constants $C$ are fixed independent of $\eps\in(0, \bar{\eps}]$. Therefore, we can take a sequence $\{\eps_n\}$ so that
\begin{itemize}
\item[-] $\displaystyle{\eps_n\rightarrow 0\,\,\tx{as $n\to \infty$}}$,
\item[-] $(v^{\eps_n}, w^{\eps_n})$ weakly converges to $(v,w)$ for some $(v,w)\in H^1(\Om_L)\times H^2(\Om_L)$,
\item[-] and $D\der_1v^{\eps_n}$ weakly converges to $D\der_1v$ in $L^2_{\rm loc}(\Om_L)$.
\end{itemize}
Then it is clear that $(v,w)$ is a weak solution to \eqref{lbvp-main general} in the sense of Definition \ref{definition of weak solution}.
Moreover, it directly follows from the uniform estimates \eqref{galerkin estimate-1} and \eqref{galerkin estimate-2} of $\{(v^{\eps}, w^{\eps})\}_{\eps\in(0, \bar{\eps}]}$ that $(v,w)$ satisfies the estimate \eqref{estimate 1 of v and w}.
\medskip

{\textbf{Step 4.}} Regarding $v$ as a weak solution to
\begin{equation}
\label{bvp for v-near the entrance}
  \begin{cases}
  \sum_{i,j=1}^2\der_j(a_{ij}\der_iv)=F_1&\quad\tx{in $\Om_L\cap\{x_1<\gs^P(x_2)\}$}\,\,\tx{(see)}\\
  v=0&\quad\tx{on $\Gamen$}\\
  \der_2 v=0&\quad\tx{on $\Gamw$}
  \end{cases}
\end{equation}
where $\gs^P$ is given in Lemma \ref{lemma on L_1}(g), and $F_1$ is given by
\begin{equation*}
F_1:=f_1-b_1\der_1w-b_0w+(\der_1 a_{11}+\der_2 a_{12}+a)\der_1v+\der_1 a_{12}\der_2 v,
\end{equation*}
one has
\begin{equation}
\label{H2-estimate of weak limit 1}
  \|v\|_{H^2(\Om_L\cap\{x_1<\frac{\ls}{2}\})}\le C(\|f_1\|_{L^2(\Om_L)}+\|f_2\|_{L^2(\Om_L)})
\end{equation}
because the equation stated in \eqref{bvp for v-near the entrance} is strictly elliptic in $\Om_L\cap\{x_1<\frac 34 \ls\}$ owing to Lemma \ref{lemma on L_1}($g_2$).
\smallskip

It follows from the uniform estimates \eqref{galerkin estimate-3} and \eqref{galerkin estimate-4} of $\{v^{\eps}\}_{\eps\in(0, \bar{\eps}]}$ and the weak convergence property stated in the previous step that
\begin{equation}
\label{estimate for v-order 2}
  \|D\der_1 v\|_{L^2(\Om_L\cap\{x_1>\frac{\ls}{8}\})}\le C\left(\|f_1\|_{H^1(\Om_L)}+\|f_2\|_{L^2(\Om_L)}\right).
\end{equation}
Since the coefficients $\{a_{ij}\}$ of the operator $\mfrak{L}_1^P$ are bounded in $C^0(\ol{\Om_L})$ due to Lemma \ref{lemma on L_1}(see the estimate \eqref{estimate-coefficient-difference}), we can fix a constant $\beta_0>0$ depending only on the data to satisfy
\begin{equation*}
\mathbb{A}^{(\beta_0)}:=  \begin{pmatrix}
  a_{11}+\beta_0&a_{12}\\
  a_{12}&1
  \end{pmatrix}\ge \frac{\min\{\beta_0, 1\}}{2}\mathbb{I}_2\quad \tx{in $\ol{\Om_L}$}.
\end{equation*}
Now, we rewrite $\mcl{B}_1^P[(v,w), V]=\int_{\Om_L}f_1 V\,d\rx$ in \eqref{weak-viscous-eqns} as
\begin{equation*}
 \begin{split}
&-\int_{\Om_L}
(a_{11}+\beta_0)\der_1v\der_1V+a_{12}(\der_1v\der_2V
+\der_2v\der_1V)+\der_2v\der_2V\,d\rx\\
&=\int_{\Om_L} (F_1+\beta_0\der_{11}v)V\,d\rx.
\end{split}
\end{equation*}
Using the uniform positivity of the matrix $\mathbb{A}^{(\beta_0)}$ and the estimate \eqref{estimate for v-order 2}, we can apply a standard elliptic estimate result(see \cite[Theorem 8.10]{GilbargTrudinger}) to conclude that, for any constant $r\in(0, L)$, there exists a constant $C_r>0$ fixed depending only on the data and $r$ such that
\begin{equation}
\label{H2 estimate of v-away from the exit}
  \|v\|_{H^2(\Om_L\cap\{\frac{\ls}{4}<x_1<L-r\})}\le C_r\left(\|f_1\|_{H^1(\Om_L)}+\|f_2\|_{L^2(\Om_L)}\right)).
\end{equation}
\smallskip

Since $w$ is in $H^2(\Om_L)$, it can be regarded as a strong solution to
\begin{equation}
\label{bvp for w}
\begin{split}
  &\Delta w=F_2\quad\tx{in $\Om_L$},\\
  &\der_1 w=0\,\,\tx{on $\Gamen$},\quad \der_2 w=0\,\,\tx{on $\Gamw$},\quad w=0\,\,\tx{on $\Gamex$}
  \end{split}
\end{equation}
for $F_2$ given by
\begin{equation*}
  F_2:=f_2+{c}_0 w+{c}_1\der_1 v.
\end{equation*}
From \eqref{H2-estimate of weak limit 1} and \eqref{H2 estimate of v-away from the exit}, it follows that, for any constant $r\in(0, L)$, there exists a constant $C_r>0$ fixed depending only on the data and $r$ such that
\begin{equation*}
\|F_2\|_{H^1(\Om_L\cap\{x_1<L-\frac r2\})}\le C_r(\|f_1\|_{H^1(\Om_L)}+\|f_2\|_{H^1(\Om_L)}).
\end{equation*}
Then we can establish a priori $H^3$-estimate of $w$ away from $\Gamex$ by applying \cite[Theorem 8.10]{GilbargTrudinger}, thus the estimate \eqref{estimate 2 of v and w} is established.

\medskip

{\textbf{Step 5.}} In Steps 1--4, we have proved Proposition \ref{proposition-wp of bvp with approx coeff} for $P\in \iterseta$, $f_1\in H^3(\Om_L)$ and $f_2\in H^2(\Om_L)$ satisfying the assumption \eqref{partial smoothness}. We now explain how to extend the proof for all $P\in \iterseta$, $f_1\in H^3(\Om_L)$ and $f_2\in H^2(\Om_L)$.
\smallskip

Set
\begin{equation*}
  \Om_L^{\rm ext}:=(-\frac L4, \frac{5L}{4})\times (-1,1).
\end{equation*}
Let $\mcl{E}:H^4(\Om_L)\rightarrow H^4(\Om_L^{\rm ext})$ be an operator such that, for any given  function $u\in H^4(\Om_L)$, $\mcl{E}u$ satisfies the following properties:
\begin{itemize}
\item[-] $\mcl{E}u=u$ in $\Om_L$
\item[-] $\der_{x_2}^k\mcl{E}u=0$ on $\Gamw$ if $\der_{x_2}^ku=0$ on $\Gamw$ for $k\in \mathbb{N}$.
\end{itemize}
One can directly construct such an operator $\mcl{E}$ so that $\mcl{E}$ is bounded and linear (e.g.see \cite[Appendix A]{bae2023supersonic} ) with the norm $\|\mcl{E}\|$ being bounded depending only on $L$.

Let $\chi:\R\rightarrow \R$ be a smooth function that satisfies the following conditions:
\begin{itemize}
\item[-] $\chi(x_1)\ge 0$ for all $x_1\in \R$,
\item[-] $\chi(x_1)=\chi(-x_1)$ for all $x_1\in \R$,
\item[-] ${\rm spt}\chi\subset [-1,1]$,
\item[-] $\int_{\R}\chi(x_1)\,dx_1=1$.
\end{itemize}
For a constant $\tau>0$, let us define $\chi^{(\tau)}:\R\rightarrow \R$ by
\begin{equation*}
  \chi^{(\tau)}(x_1):=\frac{1}{\tau}\chi\left(\frac{x_1}{\tau}\right).
\end{equation*}
Then $\{u^{(\tau)}:=(\mcl{E}u)*\chi^{(\tau)}\}_{\tau\in(0, \frac{L}{10}]}$ yields a partially smooth (with respect to $x_1$) approximation of $u$, and it converges to $u$ in $H^4(\Om_L)$ as $\tau$ tends to $0+$.
\smallskip

Let us fix $P=(\tphi, \tpsi, \tPsi)\in \iterseta$ for $(r_2, r_3)$ satisfying the condition \eqref{condition for r-final}.
Given a sequence $\{\tau_n:n\in \mathbb{N}\}\subset (0, \frac{L}{10}]$, let us define $P_n$ by
\begin{equation*}
  P_n:=(\tphi^{(\tau_n)}, \tpsi^{(\tau_n)}, \tPsi^{(\tau_n)}),\,\,(f_{1,n}, f_{2,n}):=(f_1^{(\tau_n)}, f_2^{(\tau_n)}).
\end{equation*}
In particular, take a sequence $\{\tau_n:n\in \mathbb{N}\}$ so that the following properties hold:
\begin{itemize}
\item[-] $P_n\in \iterV(2r_2)\times \iterP(2r_3)$ for any $ n\in \mathbb{N}$,
\item[-] $\displaystyle{\lim_{n\to \infty}\tau_n=0,\,\,\tx{thus it holds that }}$
$$
\lim_{n\to\infty}\|P_n-P\|_{H^4(\Om_L)}=\lim_{n\to\infty}\|f_{1,n}-f_1\|_{H^3(\Om_L)}=\lim_{n\to\infty}\|f_{2,n}-f_2\|_{H^2(\Om_L)}=0.
$$
\end{itemize}

For each $n\in \mathbb{N}$, we repeat Steps 1--4 for $(P_n, f_{1,n}, f_{2,n})$ to get a sequence $\{(v_n, w_n)\}$ so that
\begin{itemize}
\item[-] $(v_n, w_n)$ is the weak solution to the boundary value problem \eqref{lbvp-main general} associated with $P_n$,
\item[-] $(v_n, w_n)$ satisfies the estimates \eqref{estimate 1 of v and w} and \eqref{estimate 2 of v and w}.
\end{itemize}
Then there exist a subsequence $\{(v_{n_j}, w_{n_j})\}$ and $(v,w)\in H^1(\Om_L)\times H^2(\Om_L)$ that satisfy the following properties:
\begin{itemize}
\item[-] $v\in H^1(\Om_L)\cap H^2_{\rm loc}(\Om_L)$ and $w\in H^2(\Om_L)\cap H^3_{\rm loc}(\Om_L)$,
\item[-] the subsequence $\{(v_{n_j}, w_{n_j})\}$ weakly converges to $(v,w)$ in $H^1(\Om_L)\times H^2(\Om_L)$,
\item[-] the sequence $\{v_{n_j}\}$ weakly converges to $v$ in $H^2(\Om_L\cap\{x_1<L-r\})$ for any $r\in(0,L)$,
\item[-] the sequence $\{w_{n_j}\}$ weakly converges to $w$ in $H^3(\Om_L\cap\{x_1<L-r\})$ for any $r\in(0,L)$.
\end{itemize}
Therefore, it follows that $(v,w)$ is a weak solution to the boundary value problem \eqref{lbvp-main general} associated with $P\in \iterseta$, and that it satisfies the estimates \eqref{estimate 1 of v and w} and \eqref{estimate 2 of v and w}. This completes the proof of Proposition \ref{proposition-wp of bvp with approx coeff}.
 \hfill \qed

\begin{proof}[Proof of Lemma \ref{lemma-wp of Galerking approx}]

 Let us fix $\eps\in(0, \bar{\eps}]$ and $m\in \mathbb{N}$.
 \medskip

{\textbf{Step 1.}}The estimates stated in \eqref{galerkin estimate-1}--\eqref{galerkin estimate-4} can be achieved by minor modifications of the proofs of Lemmas \ref{lemma:wp of singular pert prob-main}--\ref{lemma for pre H2 estimate of vm, part 2}, so we skip to prove them. For details, readers can refer to \cite[Appendix A]{BDXX}.
\medskip

{\textbf{Step 2.}} Fix $m\in \mathbb{N}$. Let us define ${\bf X}_j:[0,L]\rightarrow \R^{m+1}$ for $j=1,\cdots, 5$ by
\begin{equation*}
  \begin{split}
  &{\bf X}_1:=(\vartheta_0, \cdots, \vartheta_m),\quad
  {\bf X}_2:={\bf X}'_1,\quad
  {\bf X}_3:={\bf X}''_1,\\
  &{\bf X}_4:=( \Theta_0,\cdots,\Theta_m),\quad
  {\bf X}_5:={\bf X}'_4.
  \end{split}
\end{equation*}
Set ${\bf X}: [0, L]\rightarrow \R^{5(m+1)\times 1}$ as
$\displaystyle{
  {\bf X}:=({\bf X}_1,  {\bf X}_2,{\bf X}_3,
 {\bf X}_4,
 {\bf X}_5)^T.}$

For each $k$, we rewrite the first equation in \eqref{galerkin-eqns} as
\begin{equation*}
  \vartheta_k'''=\frac{1}{\eps}\left(-\langle\mfrak{L}_1^P(v_m, w_m), \eta_k\rangle +\langle f_1, \eta_k \rangle\right).
\end{equation*}
This implies that ${\bf X}_3$ satisfies
\begin{equation*}
  {\bf X}_3'=\frac{1}{\eps}\left(\sum_{n=1}^5\mathbb{C}_n^P{\bf X}_n+{\bf F}_3\right)
\end{equation*}
for $\mathbb{C}^P_n:[0,L]\rightarrow \R^{(m+1)\times (m+1)}$ determined depending on the coefficients of $\mfrak{L}_1^P$. Then it is clear that \eqref{galerkin-eqns} yields a linear ODE system for ${\bf X}$ in the following form:
\begin{equation*}
  {\bf X}'=\mathbb{A}^P_{\eps}{\bf X}+{\bf F}_{\eps}
\end{equation*}
for $\mathbb{A}^P_{\eps}:[0, L]\rightarrow \R^{5(m+1)\times 5(m+1)}$ and ${\bf F}_{\eps}:[0, L]\rightarrow \R^{5(m+1)}$.

Next, we define a projection mapping $\Pi:\R^{5(m+1)\times 1}\rightarrow \R^{5(m+1)\times 1}$ by
\begin{equation*}
  \Pi{\bf X}:=({\bf X}_1, {\bf X}_2, {\bf 0}, {\bf 0}, {\bf X}_5)^T.
\end{equation*}
Then one can directly check that ${\bf X}$ yields a solution $(v_m, w_m)$ to the problem \eqref{galerkin-eqns}--\eqref{galerkin-bcs} if and only if ${\bf X}$ solves the integral equation
\begin{equation}
\label{X-eqn}
  {\bf X}(x_1)=\Pi\int_0^{x_1}(\mathbb{A}^P_{\eps}{\bf X}+{\bf F}_{\eps})(t)\,dt
  +({\rm Id}-\Pi)\int_L^{x_1}(\mathbb{A}^P_{\eps}{\bf X}+{\bf F}_{\eps})(t)\,dt.
\end{equation}
Define a linear operator $\mfrak{K}: C^1([0, L];\R^{5(m+1)\times 1})\rightarrow C^1([0, L];\R^{5(m+1)\times 1})$ by
\begin{equation*}
  \mfrak{K}{\bf X}(x_1):=\Pi\int_0^{x_1}\mathbb{A}^P_{\eps}{\bf X}(t)\,dt
  +({\rm Id}-\Pi)\int_L^{x_1}\mathbb{A}^P_{\eps}{\bf X}(t)\,dt
\end{equation*}
so that we rewrite \eqref{X-eqn} as
\begin{equation*}
  ({\rm Id}-\mfrak K){\bf X}=\Pi\int_0^{x_1}{{\bf F}_{\eps}}(t)\,dt
  +({\rm Id}-\Pi)\int_L^{x_1}{{\bf F}_{\eps}}(t)\,dt.
\end{equation*}
Due to the assumption \eqref{partial smoothness}, it follows that the operator $\mfrak{K}$ is compact, and that if ${\bf X}\in C^1([0, L];\R^{5(m+1)\times 1})$ solves \eqref{X-eqn}, then it is smooth for all $x_1\in[0,L]$.

Therefore we can apply the estimate \eqref{galerkin estimate-1} and the Fredholm alternative theorem to conclude that the problem of \eqref{galerkin-eqns}--\eqref{galerkin-bcs} has a unique smooth solution $(v_m, w_m)$. This completes the proof.
\end{proof}

\section{Higher order derivative estimates}
\label{section:higher order der est}

Throughout this section, we fix $P=(\tphi, \tpsi, \tPsi)\in \iterseta$ for $(r_2, r_3)$ satisfying the condition \eqref{condition for r-final}. By Proposition \ref{proposition-wp of bvp with approx coeff}, the boundary value problem \eqref{lbvp-main general} associated with $P$ has a unique strong solution $(v,w)$ that satisfies the estimates \eqref{estimate 1 of v and w} and \eqref{estimate 2 of v and w}. In this section, we establish a global $H^4$ estimate of $(v,w)$ by employing the idea developed in the proof of \cite[Theorems 1.5 and 1.8]{KZ} and Lemma \ref{corollary-coeff extension}. Once the $H^4$-estimate of $(v,w)$ is achieved, it naturally follows that $(v,w)$ is a $C^2$-classical solution.

\newcommand \extoperator{\mcl{L}_*}
\newcommand \Gamext{\Gam_{L_*}}
\newcommand \Gamwext{\Gamw^*}
\newcommand \Omext{\Om_{L_*}}

\subsection{Global $H^2$-estimate of $v$}\label{subsection:global 1}
Let us fix a positive constant $\om_0$ as
\begin{equation*}
  \om_0:=2\bar{a}_{11}(0).
\end{equation*}
For the constant $L_*$ given by \eqref{definition: L*}, let us set
\begin{equation*}
\Om_{L_*}:=(0,L_*)\times (-1,1).
\end{equation*}
In $\Om_{L_*}$, let the coefficient functions $(\alp_{11}^P, \alp_{12}^P,\alp)$ be given by \eqref{definition of extended coefficient} so that they satisfy all the properties stated in Lemma \ref{corollary-coeff extension} with the constant $\om_0$ given as in the above. Next, we define a linear differential operator $\extoperator^P$ by
    \begin{equation*}
      \extoperator^P V:=\alp_{11}^P\der_{11}V+\alp_{12}^P\der_{12}V+\der_{22}V+\alp^P\der_1V
      \quad\tx{in $\Om_{L_*}$}.
    \end{equation*}
Given a function $F:\Omext\rightarrow \R$, we introduce an auxiliary boundary value problem:
\begin{equation}
\label{bvp-mixed type-extended}
  \begin{cases}
\extoperator^P V=F&\quad \tx{in $\Om_{L_*}$},\\
  V=0&\quad\tx{on $\Gamen$},\\
  \der_2 V=0&\quad\tx{on $\Gamwext:=\der \Om_{L_*}\cap \{|x_2|=1\}$},\\
  \der_1 V=0&\quad\tx{on $\Gamext:=\der\Om_{L_*}\cap\{x_1=L_*\}$}.
  \end{cases}
\end{equation}
\begin{definition}
[A weak solution to \eqref{bvp-mixed type-extended}]
\label{definition of weak solution-V problem}
For two functions $V, \xi\in H^1(\Om_{L_*})$, define a bilinear operator
\begin{align*}
\mcl{B}_{*}^P[V,\xi]:=
\int_{\Om_{L_*}}&-(\alp_{11}\der_1V\der_1\xi+\alp_{12}(\der_1V\der_2\xi
+\der_2V\der_1\xi)+\der_2V\der_2\xi+\alp\der_1V\xi)\\
&-(\der_1\alp_{11}\der_1V+\der_2\alp_{12}\der_1V+\der_1\alp_{12}\der_2 V)\xi\,d\rx.
\end{align*}
A function $V\in H^1(\Om_L)$ is said to be {\emph{a weak solution}} to the boundary value problem \eqref{bvp-mixed type-extended} if it satisfies
\begin{equation*}
\mcl{B}_*^P[V,\xi]=\int_{\Om_L} F \xi\,d\rx
\end{equation*}
  for any test function $\xi\in C^{\infty}(\ol{\Om_L})$ vanishing near $\Gamen \cup \Gamext$.
\end{definition}

Next, we set up a singular perturbation problem for a sufficiently small constant $\eps>0$:
\begin{equation}
\label{bvp-sing-pert-extended}
  \begin{cases}
 \eps \der_{111}V+\extoperator^P V=F&\quad \tx{in $\Om_{L_*}$},\\
  V=0,\quad \der_1 V=0&\quad\tx{on $\Gamen$},\\
  \der_2 V=0&\quad\tx{on $\Gamwext$},\\
  \der_1 V=0&\quad\tx{on $\Gamext$}.
  \end{cases}
\end{equation}

\begin{lemma}
\label{lemma:aux bvp for V-global H2 est}
 For any given $F\in L^2(\Omext)$, the boundary value problem \eqref{bvp-mixed type-extended} has a unique weak solution. And, the solution satisfies the estimate
\begin{equation}
\label{H1 estimate of V}
  \|V\|_{H^1(\Omext)}\le C\|F\|_{L^2(\Omext)}
\end{equation}
for some constant $C>0$ fixed depending only on the data. Furthermore, one can fix a constant $\delta_2\in(0, \delta_1]$ (for the constant $\delta_1$ from \eqref{condition for r-final}) depending only on the data so that if $P\in \iterseta$ for two positive constants $r_2$ and $r_3$ satisfying
\begin{equation*}
  r_2+r_3\le \delta_2,
\end{equation*}
and if $F\in H^1(\Om_{L_*})$, then the weak solution $V$ satisfies the estimate
\begin{equation}
\label{H2 estimate of V}
  \|V\|_{H^2(\Omext)}\le C\|F\|_{H^1(\Omext)},
\end{equation}
that is, $V$ is the strong solution to \eqref{bvp-mixed type-extended}.
\begin{proof}
{\textbf{Step 1.}}
The proof of this lemma is much simpler than the one of  Proposition \ref{proposition-wp of bvp with approx coeff}.
This lemma can be proved by employing the idea of \cite[Theorem 1.1, Chapter 1]{KZ}. Given a constant $\eps>0$, if $V_{\eps}$ is a smooth solution to \eqref{bvp-mixed type-extended}, then it satisfies
\begin{equation}
\label{ext problem integral}
  \int_{\Omext} (\eps \der_{111}V_{\eps}+\extoperator^P V_{\eps})e^{-\mu x_1}\der_1 V_{\eps}\,d\rx=\int_{\Omext}e^{-\mu x_1}F \der_1 V_{\eps}\,d\rx
\end{equation}
for any constant $\mu>0$.
\begin{itemize}
\item[-] Integrating the left-hand side of \eqref{ext problem integral} by parts,
\item[-]applying Lemma \ref{corollary-coeff extension}(f) for $m=0$,
\item[-]and applying the Cauchy-Schwarz inequality,
\end{itemize}
it can be directly checked that
\begin{equation}
\label{estimate V coerc}
\begin{split}
  &{\tx{LHS of \eqref{ext problem integral}}}\\
  &\le \int_{\Omext} e^{-\mu x_1} \left(\left(-\eps+\frac{\eps^2}{2}\right)(\der_{11} V_{\eps})^2-\left(\frac{\lambda_{L_*}}{2}-C\mu\right)\frac{(\der_1 V_{\eps})^2}{2}-\mu\frac{(\der_2 V_{\eps})^2}{2}\right)\,d\rx\\
  &\phantom{\le}- e^{-\mu L_*}\int_{\Gam_{L_*}}\frac{(\der_2 V_{\eps})^2}{2}\,dx_2
  \end{split}
\end{equation}
for the constant $\lambda_{L_*}>0$ from Corollary \ref{corollary-coeff extension}(f). Fix the constant $\mu$ as $\displaystyle{\mu=\frac{\lambda_{L_*}}{4C}}$  so that if $\eps\in(0,1]$, then it directly follows from \eqref{estimate V coerc} that
\begin{equation*}
\begin{split}
  &{\tx{LHS of \eqref{ext problem integral}}}\\
  &\le -\int_{\Omext} e^{-\mu x_1} \left(\frac{\eps}{2}(\der_{11} V_{\eps})^2+\frac{\lambda_{L_*}}{4}|DV_{\eps}|^2\right)\,d\rx-e^{-\mu L_*}\int_{\Gam_{L_*}}\frac{(\der_1 V_{\eps})^2}{2}\,dx_2.
  \end{split}
\end{equation*}
By using this estimate, we can easily derive from \eqref{ext problem integral} that
\begin{equation}
\label{a priori estimate1 of Vm for new sing pert}
  \sqrt{\eps}\|\der_{11}V_{\eps}\|_{L^2(\Omext)}
  +\|V_{\eps}\|_{H^1(\Omext)}+\|\der_1V_{\eps}\|_{L^2(\Gamext)}\le C\|F\|_{L^2(\Omext)}.
\end{equation}
Then we can prove the unique existence of a weak solution to \eqref{bvp-mixed type-extended} along with the estimate \eqref{H1 estimate of V} by following the idea in the proof of Proposition \ref{proposition-wp of bvp with approx coeff}.
\begin{note}
According to the estimate \eqref{estimate V coerc}, for any $\eps\in(0,1]$, the singular perturbation problem \eqref{bvp-sing-pert-extended} has a unique solution in $H^1(\Om_L)$ for any $P\in \iterseta$. Here, we assume that the constants $r_2$ and $r_3$ satisfy the condition \eqref{condition for r-final}.
\end{note}
\medskip

{\textbf{Step 2.}}
Note that $\extoperator^P$ is uniformly elliptic in $\Omext\cap \{x_1<\frac 34 \ls\}$ due to Lemma \ref{corollary-coeff extension}(a) and Lemma \ref{lemma on L_1}(g). Also, it follows from Lemma \ref{corollary-coeff extension}(see the statements (d) and (e)) that there exists a constant $\Le\in(L, L_*)$ so that the principal coefficients $(\alp_{ij}^P)$ of $\extoperator^P$ satisfy
    \begin{equation}
    \label{uniform elliticity-ext domain}
      \alp_{11}^P\ge \frac{\om_0}{2},\quad\tx{and}\quad
      \begin{pmatrix}
\alp_{11}^P &\alp_{12}^P\\
\alp_{12}^P&1
\end{pmatrix}\ge \frac{\om_0}{2}\mathbb{I}_2\quad\tx{in $\Omext\cap\{x_1\ge \Le\}$}.
    \end{equation}
Denote
 \begin{equation*}
      l_0:=\frac{L_*-\Le}{20}.
    \end{equation*}
In addition to the uniform ellipticity of $\extoperator^P$ in $\Omext\cap\{x_1\le \frac 34 \ls\,\,\tx{or}\,\, x_1\ge \Le\}$, it follows from Lemma \ref{corollary-coeff extension}(b) that the compatibility condition $\alp_{12}=0$ holds on $\Gamw$. Then we can apply the method of reflection, \cite[Theorem 8.12]{GilbargTrudinger} and the estimate \eqref{H1 estimate of V} to establish the estimate
\begin{equation}
\label{H2-estimate of aux problem 1}
\|V\|_{H^2(\Omext\cap\{x_1<\frac{\ls}{2}\,\,\tx{or}\,\,x_1> \Le+2l_0\})}\le C\|F\|_{L^2(\Om_{L_*})}.
\end{equation}
\medskip

{\textbf{Step 3.}} We now suppose that $F\in H^1(\Omext)$. Let us fix $\eps\in(0,1]$, and let $V_{\eps}$ be a smooth solution to \eqref{bvp-sing-pert-extended}.
\medskip

 As stated in Step 2, the operator $\extoperator^P$ is uniformly elliptic in $\Omext\cap\{x_1\le \frac 34 \ls\,\,\tx{or}\,\, x_1\ge \Le\}$ with the ellipticity constant being bounded below by a positive constant fixed depending only on the data. In addition, the compatibility condition $\alp_{12}=0$ holds on $\Gamw$ (see Lemma \ref{corollary-coeff extension}(b)). So we can follow the proof of Lemma \ref{lemma for pre H2 estimate of vm, part 1} to obtain the following lemma:
\begin{lemma}\label{lemma:aux sing perturb problem-H2-pre1}
For a constant $r>0$ satisfying
\begin{equation*}
  r\le \min \left\{\frac{\ls}{8}, l_0 \right\},
\end{equation*}
let us set a domain $\mcl{D}_r$ as
\begin{equation*}
\mcl{D}_r:=\Omext\cap\left\{2r<x_1<\frac{\ls}{2}-2r\right\}\cup \left\{\Le+2r<x_1<L_*-2r\right\}.
\end{equation*}
Then, there exists a constant $C_r>0$ depending only on the data and $r$ (but independent of $\eps$) so that the following estimate holds:
\begin{equation}
\label{pre H2 estimate-sing pert-ext1}
  \|D\der_1 V_{\eps}\|_{L^2(\mcl{D}_r)}
\le C_{r}\|F\|_{L^2(\Omext)}.
\end{equation}

\end{lemma}

Let $\chi\in C^{\infty}_0(\R)$ be a smooth cut-off function that satisfies the following properties:
\begin{equation*}
\begin{split}
 &\chi(x_1)=\begin{cases}0\quad&\mbox{for $x_1\le \frac{\ls}{40}$},\\
 1\quad&\mbox{for $\frac{\ls}{32}\le x_1 \le \Le+18l_0$},\\
 0\quad&\mbox{for $x_1\ge \Le+19l_0$}(=L_*-l_0)
 \end{cases}\quad \tx{and}\quad  0\le  \chi \le 1\quad\tx{on $\R$}.
 \end{split}
\end{equation*}
Such a function can be fixed to satisfies the estimate
$$
\left|\frac{d^k\chi}{dx_1^k}\right|\le C_k\left(1+\frac{1}{\ls^k}+\frac{1}{l_0^k}\right)
$$
with a constant $C_k>0$ depending only on $k$, for each $k\in \mathbb{N}$.

Next, we use the idea from the proof of Lemma \ref{lemma for pre H2 estimate of vm, part 2}, and we apply Lemma \ref{corollary-coeff extension}(f) with $m=1$ and Lemma \ref{lemma:aux sing perturb problem-H2-pre1} to derive from
\begin{equation*}
  \int_{\Omext}(\eps \der_{111}V_{\eps}+\extoperator^P V_{\eps})\chi^2\der_{111}V_{\eps}\,d\rx
  =\int_{\Omext} F\chi^2\der_{111}V_{\eps}\,d\rx
\end{equation*}
that
\begin{equation*}
  \int_{\Omext}\left(\eps(\der_{111}V_{\eps})^2
  +|D\der_{1}V_\eps|^2\right)\chi^2\,d\rx \le C\|F\|^2_{H^1(\Omext)}
\end{equation*}
provided that the constant $r_2+r_3$ is appropriately small depending on the data. Then the following lemma is obtained.
\begin{lemma}
\label{lemma for pre H2 estimate of Vm-ext sing pert}
One can fix a constant $\delta_2\in(0, \delta_1]$ (for the constant $\delta_1$ from \eqref{condition for r-final}) depending only on the data so that if $P\in \iterseta$ for two positive constants $r_2$ and $r_3$ satisfying
\begin{equation*}
  r_2+r_3\le \delta_2,
\end{equation*}
and if $F\in H^1(\Om_{L_*})$, then any smooth solution $V_{\eps}$ to \eqref{bvp-sing-pert-extended} satisfies the estimate
\begin{equation*}
  \sqrt{\eps}\|\der_{111}V_{\eps}\|_{L^2(\Omext^{(1)})}
  +\|D\der_1V_{\eps}\|_{L^2(\Omext^{(1)})}\le C\|F\|_{H^1(\Omext)}
\end{equation*}
for $\Omext^{(1)}:=\Omext\cap \{\frac{\ls}{32}<x_1<\Le+18l_0\}$.
\end{lemma}
By using Lemma \ref{lemma for pre H2 estimate of Vm-ext sing pert}, we can follow the argument given from Step 4 in the proof of Proposition \ref{proposition-wp of bvp with approx coeff} to show that the weak solution $V$ to \eqref{bvp-mixed type-extended} satisfies the estimate
\begin{equation*}
  \|V\|_{H^2(\Omext\cap\{\frac{\ls}{8}<x_1<\Le+8l_0\})}\le C\|F\|_{H^1(\Omext)}.
\end{equation*}
By combining this estimate with \eqref{H2-estimate of aux problem 1}, we finally obtain the estimate \eqref{H2 estimate of V}. This completes the proof of Lemma \ref{lemma:aux bvp for V-global H2 est}.

\end{proof}
\end{lemma}

Now, we return to the boundary value problem \eqref{lbvp-main general}. By Proposition \ref{proposition-wp of bvp with approx coeff}, it has a unique strong solution $(v,w)\in [H^1(\Om_L)\cap H^2_{\rm loc}(\Om_L)]\times H^2(\Om_L)$. Note that
\begin{equation*}
  (a_{11}^P, a_{12}^P, a^P)=(\alp_{11}^P, \alp_{12}^P, \alp^P) \quad\tx{in $\Om_L$}.
\end{equation*}
So we can regard $v$ as a strong solution to
\begin{equation*}
  \begin{cases}
  \mcl{L}_*^P v= F_1 \quad&\mbox{in $\Om_L$},\\
  v=0\quad&\mbox{on $\Gamen$},\quad
  \der_2 v=0\quad \mbox{on $\Gamw$}
  \end{cases}
\end{equation*}
for $F_1$ given by \eqref{definition-bootstrap}. Note that it follows from \eqref{estimate 1 of v and w} and \eqref{definition-bootstrap} that
\begin{equation*}
  \|F_1\|_{H^1(\Om_L)}\le C(\|f_1\|_{H^1(\Om_L)}+\|f_2\|_{H^1(\Om_L)}).
\end{equation*}
\begin{lemma}\label{lemma:V vs v}
Given a function $F\in H^1(\Omext)$, suppose that
\begin{equation*}
  F=F_1\quad\tx{in $\Om_{L}$}.
\end{equation*}
For the constant $\delta_2\in(0, \delta_1]$ from Lemma \ref{lemma:aux bvp for V-global H2 est}, suppose that $P\in \iterseta$ for two positive constants $r_2$ and $r_3$ satisfying
\begin{equation}
\label{condition for r-auxiliary}
  r_2+r_3\le \delta_2.
\end{equation}
If $V\in H^2(\Omext)$ is a strong solution to \eqref{bvp-mixed type-extended}, then it holds that
\begin{equation*}
  V=v\quad\tx{in $\Om_L$}.
\end{equation*}

\begin{proof}
Fix a small constant $t>0$, and set $z:=v-V$ in $\Om_L\cap\{x_1<L-t\}$. Then, $z\in H^2(\Om_L\cap\{x_1<L-t\})$ satisfies
\begin{equation*}
  \begin{split}
  \extoperator^P z=0\quad&\tx{in $\Om_L\cap\{x_1<L-t\}$},\\
  z=0\quad&\tx{on $\Gamen$},\\
  \der_2z=0\quad&\tx{on $
  \Gamw\cap \{x_1<L-t\}$}.
  \end{split}
\end{equation*}
So we can integrate by parts to get
\begin{equation}
\label{estimate of xi}
\begin{split}
0&=
\int_{\Om_L\cap\{x_1<L-t\}}
\extoperator^P z\der_1 z\,d\rx=T_1+T_2
\end{split}
\end{equation}
for
\begin{equation*}
  \begin{split}
T_1:=&\int_{\Om_L\cap\{x_1<L-t\}}\left(-\der_1\alp_{11}^P+2\alp^P-2\der_2\alp^P_{12}\right)
\frac{(\der_1 z)^2}{2}\,d\rx-\int_{\{(L-t, x_2):|x_2|<1\}}\frac{(\der_2 z)^2}{2}\,dx_2,\\
T_2:=&\left(\int_{\{(L-t, x_2):|x_2|<1\}}-\int_{\Gamen}\right)\alp_{11}^P\frac{(\der_1 z)^2}{2}\,dx_2.
\end{split}
\end{equation*}
First of all, it follows from Lemma \ref{corollary-coeff extension}(f) that $T_1\le 0$ holds. From the statements ($\tx{g}_3$) and ($\tx{g}_4$) of Lemma \ref{lemma on L_1}, it directly follows that $T_2\le 0$ holds for $0<t<\frac{L-\ls}{20}$. So we obtain that
\begin{equation*}
\der_1 z=0\quad\tx{in $\Om_L$}.
\end{equation*}
Then we conclude that $z\equiv 0$ in $\Om_L$ owing to the boundary condition $z=0$ on $\Gamen$.

\end{proof}

\end{lemma}

\begin{corollary}\label{corollarly:global H2 estimate of v}
Suppose that $P\in \iterseta$ for two constants $r_2$ and $r_3$ satisfying the condition \eqref{condition for r-auxiliary}. Let $(v,w)$ be a strong solution to \eqref{lbvp-main general}, which is obtained by Proposition \ref{proposition-wp of bvp with approx coeff}. Then, one has
\begin{equation}
\label{estimate-H2 for v H3 for w-global}
  \|v\|_{H^2(\Om_L)}+\|w\|_{H^3(\Om_L)}\le C\left(\|f_1\|_{H^1(\Om_L)}+\|f_2\|_{H^1(\Om_L)}\right).
\end{equation}
\begin{proof}
For the function $F_1$ given by \eqref{definition-bootstrap}, let us define a function $F$ by
\begin{equation}
\label{extension of F}
  F:=\mcl{E}F_1\quad\tx{in $\Om_{L_*}$}
\end{equation}
for the extension operator $\mcl{E}$ given by \eqref{extension onto Om L_*}. Then, for each $k=1,2,3$, the function $F$ satisfies the estimate
\begin{equation}
\label{estimate of F}
  \|F\|_{H^k(\Om_{L_*})}\le C_k\left(\|f_1\|_{H^k(\Om_L)}+\|w\|_{H^{k+1}(\Om_L)}\right)
\end{equation}
for some constant $C_k>0$ fixed depending only on the data and $k$, as long as the norm $\|w\|_{H^{k+1}(\Om_L)}$ is finite. Then, owing to the global $H^2$-estimate of a solution to the auxiliary boundary value problem \eqref{bvp-mixed type-extended} and Lemma \ref{lemma:V vs v}, we immediately obtain the estimate
\begin{equation*}
  \|v\|_{H^2(\Om_L)}\le C\left(\|f_1\|_{H^1(\Om_L)}+\|f_2\|_{H^1(\Om_L)}\right).
\end{equation*}
Then the estimate \eqref{estimate-H2 for v H3 for w-global} is obtained by applying a bootstrap argument and a standard elliptic estimate result. Further details can be given by an analogy of Step 4 in the proof of Proposition \ref{proposition-wp of bvp with approx coeff}.

\end{proof}
\end{corollary}

\subsection{Global $H^3$-estimate of $v$}
Let us fix $P\in \iterseta$ for $(r_2, r_3)$ satisfying the condition \eqref{condition for r-auxiliary} for the constant $\delta_2>0$ given from Lemma \ref{lemma for pre H2 estimate of Vm-ext sing pert}. And, let $(v,w)$ be a strong solution to \eqref{lbvp-main general}.

\begin{lemma}\label{lemma-global H3 estimate}
One can further reduce $\delta_2>0$ depending only on the data so that the following estimate holds:
\begin{equation}
\label{H3 estimate of v}
\|v\|_{H^3(\Om_L)}\le C\left(\|f_1\|_{H^2(\Om_L)}+\|f_2\|_{H^1(\Om_L)}\right).
\end{equation}

\begin{proof}{\textbf{Step 1.}}
For the function $F$ given by \eqref{extension of F}, let $V$ be a strong solution to \eqref{bvp-mixed type-extended}. From \eqref{estimate-H2 for v H3 for w-global} and \eqref{estimate of F}, it directly follows that
\begin{equation*}
  \|F\|_{H^2(\Omext)}\le C\left(\|f_1\|_{H^2(\Om_L)}+\|f_2\|_{H^1(\Om_L)}\right).
\end{equation*}
In this proof, we shall show that
\begin{equation*}
  \|V\|_{H^3(\Om_L)}\le C\|F\|_{H^2(\Omext)}
\end{equation*}
so that the estimate \eqref{H3 estimate of v} directly follows owing to Lemma \ref{lemma:V vs v}.

\quad\\
{\textbf{Step 2.}} By Lemma \ref{lemma on L_1}($\tx{g}_3$), Lemma \ref{corollary-coeff extension}(a) and \eqref{uniform elliticity-ext domain}, the differential operator $\extoperator^P$ is uniformly elliptic in $\Omext\cap\{x_1\le \frac{5\ls}{8}\,\,\tx{or}\,\,x_1\ge \Le\}$. So we can apply standard elliptic estimates (e.g. \cite[Theorems 8.10 and 8.13]{GilbargTrudinger}) and the method of reflections to establish the estimate
\begin{equation}
\label{H3 estimate of V from elliptic part}
  \|V\|_{H^3(\Omext\cap\left(\{x_1<\frac{\ls}{2}\}\cup\{\Le+2l_0<x_1<\Le+18l_0\}\right))}\le C\|F\|_{H^1(\Omext)}.
\end{equation}
According to \cite[Theorems 8.10 and 8.13]{GilbargTrudinger}, a priori $H^3$-estimate of $V$ can be achieved if the coefficients $\{\alp_{ij}^P\}$ and $\alp^P$ are in $C^{1,1}(\ol{\Omext})$. But one should note that this is a sufficient condition rather than a necessary condition. The estimate \eqref{H3 estimate of V from elliptic part} can be still obtained for $H^3$-coefficients(see Lemma \ref{corollary-coeff extension}(d)) by adjusting the proofs of \cite[Theorems 8.10 and 8.13]{GilbargTrudinger} with using an argument of quotient differences, the Sobolev inequality and the generalized H\"{o}lder inequality.

\quad\\
{\textbf{Step 3.}} In order to estimate the $H^3$-norm of $V$ across the approximated sonic boundary $x_1=\gs^P(x_2)$, we shall return to the singular perturbation problem \eqref{bvp-sing-pert-extended}, and employ the idea given in the proofs of Lemmas \ref{lemma for pre H2 estimate of vm, part 1} and \ref{lemma for pre H2 estimate of vm, part 2} to establish a result similar to \cite[Theorem 1.5]{KZ}.

Let us fix a smooth cut-off function $\zeta\in C^{\infty}_0(\R)$ such that
\begin{equation}
\label{definition-zeta}
\begin{split}
& {\zeta}(x_1)=\begin{cases}0\quad&\mbox{for $x_1\le \frac{\ls}{8}$},\\
 1\quad&\mbox{for $\frac{\ls}{4}\le x_1 \le \Le +14l_0$},\\
 0\quad&\mbox{for $x_1\ge \Le+16 l_0$},
 \end{cases}\quad \tx{and}\quad
 0\le  {\zeta}\le 1\quad\tx{on $\R$}.
 \end{split}
\end{equation}
Such a function $\zeta$ can be fixed with satisfying the estimate
\begin{equation*}
\left|\frac{d^k\zeta}{dx_1^k}\right|\le C_k\left(1+\frac{1}{\ls^k}+\frac{1}{l_0^k}\right)
\end{equation*}
for a constant $C_k>0$ depending only on $k$, for each $k\in \mathbb{N}$.

Given $\eps\in(0, 1]$, let $V_{\eps}$ be a smooth solution to \eqref{bvp-sing-pert-extended} with $F$ being given by \eqref{extension of F}. Then we have
\begin{equation}
\label{integral eqn for H3 estimate of Vm}
  \int_{\Omext} \der_1(\eps\der_{111}V_{\eps}+\extoperator^P V_{\eps})\zeta^2\,\der_1^4V_{\eps}\,d\rx=\int_{\Omext}\der_1 F\zeta^2 \der_1^4V_{\eps}  \,d\rx.
\end{equation}
For simplicity, let us set
\begin{equation}
\label{definition of W}
W:=\der_1^2 \Veps.
\end{equation}
Denote $\alp_{ij}^P$ and $\alp^{P}$ by $\alp_{ij}$ and $\alp$, respectively. Integrating by parts gives
\begin{equation*}
  \tx{LHS of \eqref{integral eqn for H3 estimate of Vm}}
  =\int_{\Omext} \eps (\der_{11}W)^2\zeta^2+(A_0+A_1) \zeta^2+ (A_2+A_3+A_4)(\zeta^2)'+A_5(\zeta^2)''\,d\rx
\end{equation*}
for
\begin{equation*}
  \begin{split}
  &A_0:=(-\frac 32\der_1 \alp_{11}-\alp+2\der_2\alp_{12})(\der_1W)^2,\\
  &A_1:=-(\der_{1}^2\alp_{11}W
  +2\der_{1}^2\alp_{12}\der_{12}\Veps
  +\der_{1}^2\alp\der_1\Veps
  +2\der_1\alp W+4\der_1\alp_{12}\der_{2}W)\der_1W,\\
&A_2:=-\left((\alp+\der_1\alp_{11})W
+\der_1\alp\der_1\Veps
+2\der_1\alp_{12}\der_{12}\Veps\right)W,\\
&A_3:=-\frac 12 \alp_{11}(\der_1W)^2-\frac 32 (\der_{2}W)^2,\\
&A_4:=\alp_{12}W\der_{2}W,\\
&A_5:=-\der_{12}\Veps\der_{2}W.
  \end{split}
\end{equation*}

\quad\\
{\textbf{Step 4.}}
Applying Lemma \ref{corollary-coeff extension}(f) with $m=2$, we have
\begin{equation}
\label{essential part 1}
  \int_{\Omext} A_0\zeta^2 \,d\rx
  \ge \frac{\lambda_{L_*}}{2}\int_{\Omext} (\der_1W)^2\zeta^2\,d\rx.
\end{equation}
By the definition \eqref{definition of extended coefficient} of $\alp_{11}$, we have
\begin{equation*}
\int_{\Omext}\der_{11}\alp_{11}W \der_1 W\zeta^2\,d\rx
= \int_{\Omext}((\bar{\alp}_{11})''+\der_{11}d\alp_{11})
W \der_1 W\zeta^2\,d\rx.
\end{equation*}
 Note that $\bar{\alp}_{11}$ is smooth in $\Omext$. Therefore, for any $t>0$, it holds that
\begin{equation*}
  |\int_{\Omext} (\bar{\alp}_{11})''\der_{11}\Veps\der_1W\zeta^2\,d\rx|\le \int_{\Om_{\hat{L}}} \left(\frac{C}{t}(\der_{11}\Veps)^2+t(\der_1W)^2\right)\zeta^2\,d\rx.
\end{equation*}
Applying the generalized H\"{o}lder inequality, the Sobolev inequality, the Poincar\'{e} inequality and Lemma \ref{corollary-coeff extension}(d), we obtain that
\begin{equation*}
\begin{split}
&\left|\int_{\Omext}\der_{11}(d \alp_{11})W \der_1W\zeta^2\,d\rx\right|\\
  &\le \|\der_{11}(d\alp_{11})\|_{L^4(\Omext)}\|\der_1W\zeta\|_{L^2(\Omext)}
  \|\zeta W\|_{L^4(\Omext)}\\
  &\le C(r_2+r_3) \|\der_1W\zeta\|_{L^2(\Omext)}\|\zeta W\|_{H^1(\Omext)}\\
  &\le C(r_2+r_3)\left(\int_{\Omext} (\der_1W)^2\zeta^2\,d\rx \right)^{\frac 12} \left(\int_{\Omext}W^2(\zeta^2+|{\zeta}'|^2)+\zeta^2|D(\der_{11}\Veps)|^2\,d\rx\right)^{\frac 12}\\
  &\le t\int_{\Omext} (\der_1W)^2\zeta^2\,d\rx +\frac{C(r_2+r_3)^2}{t}
\int_{\Omext}W^2(\zeta^2+|{\zeta}'|^2)
+\zeta^2|DW|^2\,d\rx
  \end{split}
\end{equation*}
for any $t>0$.
Similarly, we also have
\begin{equation*}
\begin{split}
&\left|\int_{\Omext} \der_{11}\alp_{12}\der_{12}\Veps \der_1W \zeta^2\,d\rx\right|\\
&  \le t\int_{\Omext}
(\der_1W)^2\zeta^2\,d\rx +\frac{C(r_2+r_3)^2}{t}
\int_{\Omext}(\zeta^2+|{\zeta}'|^2)(\der_{12}\Veps)^2
+|D(\der_{12}\Veps)|^2\zeta^2\,d\rx.
\end{split}
\end{equation*}

Solving the equation $\eps\der_{111}\Veps+\extoperator^P\Veps=F$ stated in \eqref{bvp-sing-pert-extended} for $\der_{22}\Veps$, and differentiating the result with respect to $x_1$, one has
\begin{equation*}
  \der_{122}\Veps=\der_1F-\eps\der_1^2 W
  -\der_1(\alp_{11}\der_{11}\Veps+2\alp_{12}\der_{12}\Veps+\alp\der_1 \Veps).
\end{equation*}
This expression yields
\begin{equation*}
\begin{split}
  &\int_{\Omext}(\der_{122}\Veps)^2\zeta^2\,d\rx\\
&\le C\int_{\Omext} (\der_1F)^2\zeta^2+\eps^2(\der_1^2W)^2\zeta^2
+\zeta^2\left((\der_1\Veps)^2+|D\der_1\Veps|^2+|DW|^2\right)\,d\rx.
  \end{split}
\end{equation*}
Combining all the previous four inequalities together yields that, for any $t>0$,
\begin{equation}
\label{essential part 4}
\begin{split}
  &\left|\int_{\Omext}(\der_{1}^2\alp_{11} W
  +2\der_{1}^2\alp_{12}\der_{12}\Veps)\der_1W\zeta^2\,d\rx\right|\\
  &\le t\int_{\Omext} (\der_1W)^2\zeta^2\,d\rx+ \frac{C{(r_2+r_3)}^2}{t}\int_{\Omext}\left(\eps^2(\der_1^2W)^2+|DW|^2\right)\zeta^2\,d\rx\\
  &\phantom{\le}+
  \frac{C{(r_2+r_3)}^2}{t}\int_{\Omext}((\der_1\Veps)^2+|D\der_1\Veps|^2+(\der_1F)^2)(\zeta^2+(\zeta')^2)\,d\rx.
  \end{split}
\end{equation}
Also, it can be directly checked that, for any $t>0$, we have
\begin{equation}
\label{essential part 5}
\begin{split}
&\left|\int_{\Omext} (\der_{11}\alp\der_1\Veps
  +2\der_1\alp W+4\der_1\alp_{12}\der_{2}W)\der_1 W\zeta^2\,d\rx\right|\\
&\le t\int_{\Omext}(\der_1W)^2\zeta^2\,d\rx+\frac Ct\|\der_1\Veps\zeta\|^2_{H^1({\Omext})}
+\frac{C{(r_2+r_3)}^2}{t}\int_{{\Omext}}(\der_2W)^2\zeta^2\,d\rx.
\end{split}
\end{equation}
We combine all the estimates \eqref{essential part 1}--\eqref{essential part 5} together, and choose the constant $t>0$ sufficiently small to get
\begin{equation*}
\begin{split}
  &\int_{\Omext} (A_0+A_1)\zeta^2\,d\rx\\
  &\ge \frac{\lambda_{L_*}}{4}\int_{\Omext} (\der_1W)^2\zeta^2\,d\rx
  -C{(r_2+r_3)}^2\int_{\Omext}\left(|DW|^2+\eps^2(\der_{11} W)^2\right)\zeta^2\,d\rx\\
  &\phantom{\ge \ge\ge} -C\int_{\Omext}(\zeta^2+|\zeta'|^2)\left((\der_1\Veps)^2+|D\der_1\Veps|^2\right)+(\der_1F)^2\,d\rx.
  \end{split}
\end{equation*}
For the domain $\Omext^{(1)}$ given in Lemma \ref{lemma for pre H2 estimate of Vm-ext sing pert}, it hold that $\displaystyle{{\rm spt}\, \zeta\subset \Omext^{(1)}}$ so we can apply Lemma \ref{lemma for pre H2 estimate of Vm-ext sing pert} to further estimate as
\begin{equation}
 \label{estimate of A0 part1}
\begin{split}
  \int_{\Omext} (A_0+A_1)\zeta^2\,d\rx\ge & \frac{\lambda_{L_*}}{4}\int_{\Omext} (\der_1W)^2\zeta^2\,d\rx-C\|F\|_{H^1(\Omext)}^2\\
  &-C{(r_2+r_3)}^2\int_{\Omext}\left(|DW|^2+\eps^2(\der_{11} W)^2\right)\zeta^2\,d\rx.
  \end{split}
\end{equation}

\quad\\
{\textbf{Step 5.}} Set
\begin{equation*}
\mcl{I}_1:=\int_{\Omext} A_2(\zeta^2)' \,d\rx,\quad
\mcl{I}_2:=  \int_{\Omext} (A_3+A_4) (\zeta^2)' +A_5(\zeta^2)''\,d\rx.
\end{equation*}
\\
Since $\displaystyle{{\rm spt}\, \zeta'\subset \Omext^{(1)}}$, we can apply Lemmas \ref{corollary-coeff extension}(d) and \ref{lemma for pre H2 estimate of Vm-ext sing pert} to get
\begin{equation}
\label{estimate-A2-final1}
|\mcl{I}_1| \le C\|F\|^2_{H^1(\Omext)}.
\end{equation}

\quad\\
Next, we shall adjust the proof of Lemma \ref{lemma for pre H2 estimate of vm, part 1} to estimate the term $\mcl{I}_2$.
Let us fix a smooth function $\xi$ with double bumps to satisfy the following properties:
\begin{equation*}
\begin{split}
\xi(x_1)=\begin{cases}0\quad&\mbox{for $x_1\le \frac{\ls}{10}$},\\
 1\quad&\mbox{for $\frac{\ls}{8}\le x_1 \le \frac{5\ls}{8}$},\\
 0\quad&\mbox{for $\frac{3\ls}{4}\le x_1\le \Le+8l_0$},\\
  1\quad&\mbox{for $\Le+12l_0\le x_1 \le \Le+16l_0$},\\
 0\quad&\mbox{for $x_1\ge \Le+17l_0$},
 \end{cases}\quad
 \tx{and}\quad
 0\le  \xi(x_1)\le 1\quad\tx{on $\R$}.
 \end{split}
\end{equation*}
One can fix $\xi\in C^{\infty}(\R)$ so that, for each $k\in \mathbb{N}$, it holds that
\begin{equation*}
\left|\frac{d^k\xi}{dx_1^k}\right|\le C_k\left(1+\frac{1}{\ls^k}+\frac{1}{l_0^k}\right)
\end{equation*}
with a constant $C_k>0$ depending only on $k$.
For such a function $\chi$, we use the following integral representation to estimate the term $\mcl{I}_2$:
\begin{equation}
\label{estiate-A2-1}
   \int_{\Omext} \der_1(\eps\der_{111}\Veps+\extoperator^P\Veps)\xi^2\der_{1}^3\Veps \,d\rx=\int_{\Omext}\der_1 F\xi^2\der_{1}^3\Veps\,d\rx.
\end{equation}

First of all, we integrate by parts to get
\begin{equation}
\label{estiate-A2-2}
\int_{\Omext}\der_1(\eps\der_{111}\Veps)\xi^2\der_{1}^3\Veps\,d\rx
  =-\int_{\Omext}\eps (\der_{111}\Veps)^2\xi\xi'\,d\rx.
\end{equation}
Since $\rm{spt}\,\xi'\subset \Omext^{(1)}$, we can apply Lemma
\ref{lemma for pre H2 estimate of Vm-ext sing pert} to obtain the estimate
\begin{equation*}
\label{estimate of sing pert term in elliptic domain}
  |\int_{\Omext}\der_1(\eps\der_{111}\Veps)\der_{1}^2\Veps\xi^2\,d\rx|\le C\|F\|^2_{H^1(\Omext)}.
\end{equation*}

Set
\begin{equation*}
  \begin{split}
\mcl{J}_1&:=\int_{\Omext} (\der_1\alp_{11}\der_1^2\Veps+2\der_1\alp_{12}\der_{12}\Veps
  +\der_1\alp\der_1\Veps)\xi^2\der_1^3\Veps\,d\rx\\
  \mcl{J}_2&:=\int_{\Omext} \extoperator^P(\der_1 \Veps)\xi^2\der_1^2\Veps\,d\rx.
  \end{split}
\end{equation*}
By using Lemma \ref{corollary-coeff extension}(d) and Lemma \ref{lemma for pre H2 estimate of Vm-ext sing pert}, for any $t>0$, one can directly check that
\begin{equation}
\label{estiate-A2-3}
  |\mcl{J}_1|\le t\int_{\Omext}(\der_{1} W)^2\xi^2\,d\rx+\frac Ct\|F\|^2_{H^1(\Omext)}.
\end{equation}

Continuing to use the definition \eqref{definition of W} of $W$, a direct computation yields that
\begin{equation*}
  \der_{22}(\der_1 \Veps)\xi^2\der_{1}^3\Veps=\der_2(\der_{12}\Veps\xi^2\der_{1}W)
  -\der_1(\der_{12} \Veps \xi^2\der_{2}W)+(\der_{2}W)^2\xi^2+2\der_{12}\Veps\der_{2}W\xi\xi'.
\end{equation*}
Integrating by parts with using this expression gives
\begin{equation*}
  \mcl{J}_2=\int_{\Omext} (\sum_{i,j=1}^2\alp_{ij}\der_iW\der_jW
+\alp\der_{11}\Veps\der_{1}W)\xi^2
  +2\der_{12}\Veps \der_2 W\xi\xi'\, d\rx
\end{equation*}
for $\alp_{22}\equiv 1$ in $\Omext$. It follows from Lemma \ref{lemma on L_1}($\tx{g}_3$), Lemma \ref{corollary-coeff extension}(a) and the estimates in \eqref{uniform elliticity-ext domain} that
\begin{equation*}
\begin{split}
\mcl{J}_2 \ge \int_{\Omext} (\lambda_1|DW|^2+\alp\der_{11}\Veps\der_{1}W)\xi^2
+2\der_{12}\Veps \der_2 W\xi\xi'\,d\rx
  \end{split}
\end{equation*}
for
$$\lambda_1=\min\{\lambda_0,\om_0\},$$
where the positive constants $\lambda_0$ and $\om_0$ from Lemma \ref{corollary-coeff extension}(a) and \eqref{uniform elliticity-ext domain}, respectively.
Then we apply Lemma \ref{lemma for pre H2 estimate of Vm-ext sing pert} to get
\begin{equation}
\label{estiate-A2-4}
 \mcl{J}_2\ge
  \frac{\lambda_1}{2}\int_{\Omext} |DW|^2\xi^2\,d\rx-C\|F\|^2_{H^1(\Omext)}.
\end{equation}
\quad\\

Let us set
\begin{equation*}
\Omext^{(2)}:=\Omext\cap \left\{(x_1,x_2):x_1\in \left(\frac{\ls}{8},\frac{5\ls}{8}\right)\cup (\Le+12l_0, \Le+16l_0)\right\}.
\end{equation*}
We combine all the estimates given in \eqref{estiate-A2-2}--\eqref{estiate-A2-4} with choosing  a constant $t>0$ sufficiently small in \eqref{estiate-A2-3} so that we can derive from \eqref{estiate-A2-1} that
\begin{equation}
\label{estimate-A2-final2}
\|DW\|_{L_2(\Omext^{(2)})} \le C\|F\|_{H^1(\Omext)}.
\end{equation}
Since $\rm spt \,\zeta'\cup \rm spt \,\zeta''\subset \Omext^{(2)}\subset \Omext^{(1)}$, we can apply Lemma \ref{lemma for pre H2 estimate of Vm-ext sing pert} and the estimate \eqref{estimate-A2-final2} to obtain the estimate
\begin{equation}
\label{estimate-A2 and A3}
 |\mcl{I}_2|\le  C\|F\|_{H^1(\Omext)}^2.
\end{equation}

\quad\\
{\textbf{Step 6.}} Back to \eqref{integral eqn for H3 estimate of Vm}, we combine the estimates \eqref{estimate of A0 part1}, \eqref{estimate-A2-final1} and \eqref{estimate-A2 and A3} to get
\begin{equation*}
\begin{split}
  {\tx{LHS of \eqref{integral eqn for H3 estimate of Vm}}}
  \ge
  &\int_{\Omext} \left(\eps(1-C\eps (r_2+r_3)^2)(\der_1^2W)^2
  +\left(\frac{\lambda_{L_*}}{4}-C(r_2+r_3)^2\right)(\der_{1}W)^2\right)\zeta^2\,d\rx\\
  &-C(r_2+r_3)^2\int_{\Omext}|\der_2 W|^2\zeta^2\,d\rx-C\|F\|_{H^1(\Omext)}^2.
\end{split}
\end{equation*}
Next, we integrate the right-hand side of \eqref{integral eqn for H3 estimate of Vm} by parts, and apply \eqref{estimate-A2-final2} to get the estimate
\begin{equation*}
  \begin{split}
  {|\tx{RHS of \eqref{integral eqn for H3 estimate of Vm}}|}
  \le t\int_{\Omext} (\der_{1}W)^2\zeta^2\,d\rx+\frac Ct\int_{\Omext} (\der_{11}F)^2\zeta^2\,d\rx+C\|F\|^2_{H^1(\Omext)}
  \end{split}
\end{equation*}
for any constant $t>0$.
From the two previous estimates, it follows that one can reduce the constant $\delta_2>0$ from the one given from Lemma \ref{lemma for pre H2 estimate of Vm-ext sing pert} so that if the inequality
$$r_2+r_3\le \delta_2$$
holds, then it follows that
\begin{equation}
\label{H3 estimate pre}
  \begin{split}
  &\int_{\Omext} \left(\frac{\eps}{2}(\der_1^2 W)^2
  +\frac{\lambda_{L_*}}{8}(\der_{1}W)^2\right)\zeta^2\,d\rx\\
  &\le C(r_2+r_3)^2\int_{\Omext} |\der_2 W|^2\zeta^2\,d\rx+C\|F\|_{H^2(\Omext)}^2.
  \end{split}
\end{equation}

In order to close the estimate \eqref{H3 estimate pre}, we shall adjust the arguments given in Step 2 in the proof of Lemma \ref{lemma for pre H2 estimate of vm, part 2}, and in Step (5-2) of this proof.

By Lemma \ref{corollary-coeff extension}(d), one can fix a constant $\beta>0$ depending only on the data so that the matrix $\mathbb{B}=\begin{pmatrix}\beta&\alp_{12}\\
\alp_{12}&1\end{pmatrix}$ satisfies
\begin{equation}\label{definition of beta}
\mathbb{B}\ge \frac{\beta}{2}\mathbb{I}_2\quad\tx{in $\ol{\Omext}$}.
\end{equation}
Next, we rewrite the equation $\displaystyle{\eps\der_{111}\Veps+\extoperator^P\Veps=F}$ as
\begin{equation*}
  \eps\der_{111}\Veps+\beta\der_{11}\Veps+2\alp_{12}\der_{12}\Veps+\der_{22}\Veps
  +\alp\der_1\Veps=F-(\alp_{11}-\beta)\der_{11}\Veps\quad\tx{in $\Omext$.}
\end{equation*}
By following the idea given in Step 5, with replacing the cut-off function $\xi$ by $\zeta$, we can derive the following estimate from the above equation:
\begin{equation}
\label{H3 estimate of Veps-2}
  \int_{\Omext}|DW|^2\zeta^2\le (\beta-\alp_{11})^2(\der_{1}W)^2\zeta^2\,d\rx+C\|F\|_{H^1(\Omext)}^2.
\end{equation}
We substitute this estimate into the right-hand side of \eqref{H3 estimate pre}, and further reduce the constant $\delta_2$ to obtain the estimate
\begin{equation*}
  \begin{split}
  \int_{\Omext} \frac{\eps}{2}(\der_1^2 W)^2\zeta^2
  +\frac{\lambda_{L_*}}{16}(\der_{1}W)^2\zeta^2\,d\rx
  \le C\|F\|_{H^2(\Omext)}^2.
  \end{split}
\end{equation*}
Then we use the estimate \eqref{H3 estimate of Veps-2} again to finally obtain the following result:
\begin{equation}
\label{H3 estimate of Veps}
  \sqrt{\eps}\|\der_{1}^2 W \zeta\|_{L^2(\Omext)}
  +\|DW \zeta\|_{L^2(\Omext)}\le C\|F\|_{H^2(\Omext)}.
\end{equation}

\quad\\
{\textbf{Step 7.}} Now we are ready to estimate the strong solution $V$ to \eqref{bvp-mixed type-extended}.
Even though the estimate \eqref{H3 estimate of Veps} is derived with assuming that $\Veps$ is smooth, it can be directly verified by the method of Galerkin's approximation that the estimate is valid for a $H^1$-weak solution $\Veps$ to \eqref{bvp-sing-pert-extended}. Therefore, we can pass to the limit (of a sequence) as $\eps$ tends to $0+$, and apply \eqref{estimate-H2 for v H3 for w-global}, \eqref{estimate of F}, and \eqref{H3 estimate of Veps} to get the estimate
\begin{equation}
\label{H3 estimate of V-1}
  \|D\der_{11}V\|_{L^2(\Omext\cap\{\frac{\ls}{4}<x_1<\Le+14l_0\})}\le C\left(\|f_1\|_{H^2(\Om_L)}+\|f_2\|_{H^1(\Om_L)}\right).
\end{equation}

For the constant $\beta>0$ given to satisfy \eqref{definition of beta}, we rewrite the equation $\extoperator^P V=F$ as
\begin{equation*}
  \beta\der_{11}V+2\alp_{12}\der_{12}V+\der_{22}V
  =F-\alp\der_1V-(\alp_{11}-\beta)\der_{11}V\quad\tx{in $\Omext$}.
\end{equation*}
By applying a standard elliptic estimate result with using \eqref{H3 estimate of V from elliptic part} and \eqref{H3 estimate of V-1}, we obtain the estimate
\begin{equation*}
  \|V\|_{H^3(\Omext\cap\{x_1<\Le+12l_0\})}\le C\left(\|f_1\|_{H^2(\Om_L)}+\|f_2\|_{H^1(\Om_L)}\right).
\end{equation*}
This estimate combined with Lemma \ref{lemma:V vs v} proves Lemma \ref{lemma-global H3 estimate}.

\end{proof}
\end{lemma}

\subsection{Proof of Proposition \ref{theorem-wp of lbvp for system with sm coeff}} \label{subsection:H4 estimate}

Assume that the background solution $(\bar u_1, \bar E)$ and the nozzle length $L$ are fixed to satisfy all the conditions stated in Lemma \ref{proposition-H1-apriori-estimate}. Let us fix $P\in \iterseta$ for $(r_2, r_3)$ satisfying the condition
\begin{equation*}
  r_2+r_3\le \delta_2
\end{equation*}
for the constant $\delta_2>0$ from Lemma \ref{lemma-global H3 estimate}. Then, it follows from Proposition \ref{proposition-wp of bvp with approx coeff}, Corollary \ref{corollarly:global H2 estimate of v} and Lemma \ref{lemma-global H3 estimate} that Problem \ref{problem-lbvp for iteration} has a unique strong solution $(v,w)$ that satisfies the following estimates:
\begin{equation*}
  \begin{split}
  \|w\|_{H^3(\Om_L)}&\le C\left(\|f_1\|_{H^1(\Om_L)}+\|f_2\|_{H^1(\Om_L)}\right),\\
  \|v\|_{H^3(\Om_L)}&\le C\left(\|f_1\|_{H^2(\Om_L)}+\|f_2\|_{H^1(\Om_L)}\right).
  \end{split}
\end{equation*}
Therefore, we get to prove Proposition \ref{theorem-wp of lbvp for system with sm coeff} if we verify the estimate \eqref{estimate of v and w}, the global $H^4$-estimate of $(v,w)$.

\quad\\
{\textbf{Step 1.}}
Regarding $w$ as a solution to \eqref{bvp for w}, we apply the method of reflection and \cite[Theorem 8.13]{GilbargTrudinger} to achieve the estimate $$\|w\|_{H^4(\Om_L)}\le C\|F_2\|_{H^2(\Om_L)}.$$
And, we combine this estimate with Proposition \ref{proposition-wp of bvp with approx coeff} and Lemma \ref{lemma-global H3 estimate} to show that
\begin{equation*}
  \|w\|_{H^4(\Om_L)}\le C\left(\|f_1\|_{H^2(\Om_L)}+\|f_2\|_{H^2(\Om_L)}\right).
\end{equation*}
From \eqref{estimate of F} and the estimate given in the right above, it follows that
\begin{equation*}
  \|F\|_{H^3(\Omext)}\le C\left(\|f_1\|_{H^3(\Om_L)}+\|f_2\|_{H^2(\Om_L)}\right).
\end{equation*}
\quad\\
{\textbf{Step 2.}} The overall procedure to estimate $\|v\|_{H^4(\Om_L)}$ is similar to the proof of Lemma \ref{lemma-global H3 estimate}. So we give a brief description of how to establish an a priori estimate of $\|v\|_{H^4(\Om_L)}$ without details. Since we have $v=V$ in $\Om_L$ for the solution $V$ to the auxiliary boundary value problem \eqref{bvp-mixed type-extended} in the extended domain $\Omext$, it suffices to estimate $\|V\|_{H^4(\Om_L)}$ by taking the following steps:

\begin{itemize}
\item[(1)] As explained in Step 2 in the proof of Lemma \ref{lemma-global H3 estimate}, it follows from Lemma \ref{lemma on L_1}($\tx{g}_3$), Lemma \ref{corollary-coeff extension}(a) and \eqref{uniform elliticity-ext domain} that the differential operator $\extoperator^P$ is uniformly elliptic in $\Omext\cap\{x_1\le \frac{5\ls}{8}\,\,\tx{or}\,\,x_1\ge \Le\}$. So we can check that
\begin{equation}
\label{H4 estimate of V-pre}
  \|V\|_{H^4(\Omext\cap(\{x_1<\frac{\ls}{2}\}\cup\{\Le+2l_0<x_1<\Le+18l_0\}))}\le C\|F\|_{H^2(\Omext)}.
\end{equation}

\item[(2)] Next, we fix a smooth cut-off function $\til{\zeta}\in C^{\infty}(\R)$ such that
\begin{itemize}
\item[-]${\rm spt}\til{\zeta}\subset\{x_1:\zeta(x_1)=1\}$,
\item[-]and $\til{\zeta}(x_1)=1$ for $ \frac 38 \ls\le x_1\le\Le+10 l_0$
\end{itemize}
for the cut-off function $\zeta$ given in Step 3 in the proof of Lemma \ref{lemma-global H3 estimate}, with satisfying the condition \eqref{definition-zeta}.
\end{itemize}
Similarly to \eqref{integral eqn for H3 estimate of Vm}, we suppose that $\Veps$ is a smooth solution to \eqref{bvp-sing-pert-extended} for $\eps\in(0,1]$. Then, we have
\begin{equation}
\label{integral eqn for H4 estimate}
  \int_{\Omext} \der_1^2(\eps \der_{111}\Veps+\extoperator^P\Veps)\til{\zeta}^2\der_1^5\Veps \,d\rx=\int_{\Omext}\der_{1}^2F \til{\zeta}^2\der_1^5\Veps \,d\rx.
\end{equation}
By
\begin{itemize}
\item[-] applying Lemma \ref{corollary-coeff extension}(f) with $m=3$,
\item[-] adjusting the arguments given in Steps 3--6 in the proof of Lemma \ref{lemma-global H3 estimate},
\item[-] and using the estimate \eqref{H3 estimate of Veps},
\end{itemize}
we can show from \eqref{integral eqn for H4 estimate} that one can fix a constant $\delta_3\in(0, \delta_2]$ depending only on the data so that if $(r_2, r_3)$ satisfy the condition
\begin{equation*}
  r_2+r_3\le \delta_3,
\end{equation*}
then, $\Veps$ satisfies the estimate
\begin{equation*}
  \|D\der_1^3 \Veps \til{\zeta}\|_{L^2(\Omext)}\le C\|F\|_{H^3(\Omext)}.
\end{equation*}
Then, we pass to the limit as $\eps$ tends to $0+$, and apply \eqref{estimate of F} and Corollary \ref{corollarly:global H2 estimate of v} to establish the estimate
\begin{equation}
\label{H4 estimate of V}
  \|D\der_1^3 V \|_{L^2(\Omext\cap\{\frac 38\ls<x_1<\Le+10l_0\})}
  \le C\left(\|f_1\|_{H^3(\Om_L)}+\|f_2\|_{H^2(\Om_L)}\right)
\end{equation}
for the solution $V$ to the boundary value problem \eqref{bvp-mixed type-extended}. Then we can adjust the argument in Step 7 in the proof of Lemma \ref{lemma-global H3 estimate} and apply the estimates \eqref{H4 estimate of V-pre} and \eqref{H4 estimate of V} to show that
\begin{equation}
\label{H4 estimate of V-2}
  \|\der_1V\|_{H^3(\Omext\cap\{x_1<\Le+9l_0\})}\le C\left(\|f_1\|_{H^3(\Om_L)}+\|f_2\|_{H^2(\Om_L)}\right).
\end{equation}
Finally, we rewrite the equation $\extoperator^PV=F$ as
\begin{equation*}
  \der_{22}V=F-\left(\alp_{11}\der_{11}V+2\alp_{12}\der_{12}V+\alp\der_1V\right)\quad\tx{in $\Omext$}.
\end{equation*}
By using Lemma \ref{corollary-coeff extension}(d) and the estimate \eqref{H4 estimate of V-2}, we can directly derive from the above representation that
\begin{equation*}
  \|\der_2^4V\|_{L^2(\Omext\cap\{x_1<\Le+9l_0\})}\le
  C\left(\|f_1\|_{H^3(\Om_L)}+\|f_2\|_{H^2(\Om_L)}\right).
\end{equation*}
The proof is completed if we choose $ \bar{\delta}=\delta_3$.
\hfill \qedsymbol

\section{Proof of Theorem \ref{theorem-HD}}
\label{seciton:pf of main theorem}
Now we are ready to prove Theorem \ref{theorem-HD}, the main theorem of this paper.
\newcommand \itersetinrs {\iterseta}
\newcommand \itersetforvor{\iterT}
\newcommand \itersetforpot{\iterP}
\smallskip

\begin{proof}
Let the constant $\bar{\delta}>0$ be from Proposition \ref{theorem-wp of lbvp for system with sm coeff}. Suppose that the constants $r_1,r_2$ and $r_3$ from the iterations sets $\iterT(r_1)$ and $\iterseta$, given by Definition \ref{definition: iteration sets}, satisfy
\begin{equation}
\label{condition for r-0}
  \max\{r_1, r_2+r_3\}\le \bar{\delta}.
\end{equation}

\medskip

{\textbf{Step 1.}} {\emph{Claim 1: One can fix the constants $r_1$, $r_2$ and $r_3$ so that, for any given $\tilT\in \iterT(r_1)$ and $P=(\tpsi, \tPsi)\in \iterseta$, Problem \ref{LBVP1 for iteration}(see \S \ref{subsection:framework}) has a unique solution $(\phi, \psi, \Psi)\in \iterT(r_1)\times \iterseta$.}}
\smallskip

Note that Corollary \ref{corollary:wp of mixed system in iteration} directly follows from Proposition \ref{theorem-wp of lbvp for system with sm coeff}. Therefore, for any given $\tilT\in \iterT(r_1)$ and $P=(\tpsi, \tPsi)\in \iterseta$, the linear boundary value problem \eqref{lbvp main} associated with $(\tilT, P)$ has a unique solution $(\psi, \Psi)\in H^4(\Om_L)\times H^4(\Om_L)$, and it satisfies the estimate
\begin{equation}
       \label{estimate:system}
      \|\psi\|_{H^4(\Om_L)}+\|\Psi\|_{H^4(\Om_L)}
      \le C_1\left(r_1+r_2+r_3^2+\mcl{P}(S_0, E_{\rm en}, \om_{\rm en})\right)
    \end{equation}
    for $\mcl{P}(S_0, E_{\rm en}, \om_{\rm en})$ defined by \eqref{definition-perturbation of bd}.

Due to the $H^3$-estimate and the compatibility conditions for the function $f_0^{(\til T,P)}$ stated in Lemma \ref{lemma on L_1}(f), the elliptic boundary value problem \eqref{lbvp for phi} has a unique solution $\phi\in H^5(\Om_L)$ that satisfies the estimate
\begin{equation}
    \label{estimate:phi}
      \|\phi\|_{H^5(\Om_L)}\le C_2r_1.
    \end{equation}
In \eqref{estimate:system} and \eqref{estimate:phi}, the estimate constants $C_1$ and $C_2$ are fixed depending only on the data. For the rest of the proof, any estimate constant is regarded to be fixed depending only on the data unless otherwise specified.

Let us set
\begin{equation}
\label{definition of C_*}
  C_*:=\max\{C_1, C_2\}.
\end{equation}

\begin{itemize}
\item[(i)] If the constant $r_1$ satisfies the inequality
\begin{equation}
\label{condition for r-1}
r_1\le \frac{r_2}{C_*},
\end{equation}
then we have
\begin{equation*}
\|\phi\|_{H^5(\Om_L)}\le r_1.
\end{equation*}

\item[(ii)] If the constants $r_1$, $r_2$, $r_3$ and $\mcl{P}(S_0, E_{\rm en}, \om_{\rm en})$ satisfy
\begin{equation}
\label{condition for r-2}
\begin{split}
&C_*(r_1+r_2)\le \frac{r_3}{3},\quad C_*r_3\le \frac 13\quad\tx{and}\quad C_*\mcl{P}(S_0, E_{\rm en}, \om_{\rm en})\le \frac{r_3}{3},
\end{split}
\end{equation}
then we have
\begin{equation}
\label{potential estimate fixed pt}
\|(\psi, \Psi)\|_{H^4(\Om_L)}\le r_3.
\end{equation}
\end{itemize}

We rewrite the equation $\der_2\mfrak{L}_1^P(\psi, \Psi)=\der_2 f_1^P$ in $\Om_L$ as
\begin{equation*}
\der_{222}\psi=\der_2 f_1^P-\der_2\left(a_{11}^P+2a_{12}^P+a^P\der_1\psi+b_1^P\der_1\Psi+b_0^P\Psi\right) \quad\tx{in $\Om_L$}.
\end{equation*}
Owing to the trace inequality, this expression is valid up to the boundary $\der \Om_L$.
So we can directly check by using the properties (c) and (f) stated in Lemma \ref{lemma on L_1}, and  the slip boundary condition $\der_2\psi=0$ on $\Gamw$ that $\psi$ satisfies the compatibility condition
$$\der_2^3\psi=0\quad\tx{on $\Gamw$}.$$
In addition, one can similarly check that
\begin{equation*}
  \der_2^k\Psi=0\,\,\tx{on $\Gamw$ for $k=1,3$},\quad\tx{and}\quad
  \der_2^{m-1}\phi=0\,\, \tx{on $\Gamw$ for $m=1,3,5$}.
\end{equation*}
Therefore if we fix $(r_1, r_2, r_3)$ to satisfy the conditions \eqref{condition for r-1} and \eqref{condition for r-2}, then we can conclude that
\begin{equation*}
(\phi, \psi, \Psi)\in \iterT(r_1)\times \iterseta.
\end{equation*}
This verifies Claim 1.
\medskip

{\textbf{Step 2.}} For a fixed $\tilT\in \iterT(r_1)$, let us consider a nonlinear system for $(\phi, \psi, \Psi)$
\begin{equation}
\label{nlsystem-2nd order-o}
\begin{cases}
-\Delta \phi=f_0^{(\tilT, (\phi, \psi, \Psi))}\quad &\tx{in $\Om_L$},\\
\mfrak{L}_1^{(\phi, \psi, \Psi)}(\psi, \Psi)
=f_1^{(\phi, \psi, \Psi)}\quad &\tx{in $\Om_L$},\\
\mfrak{L}_2(\psi, \Psi)=f_2^{(\tilT, (\phi, \psi, \Psi))}\quad &\tx{in $\Om_L$},
\end{cases}
\end{equation}
with the boundary conditions
\begin{equation}
\label{BCs:nlsystem-2nd order-o}
\begin{split}
\quad \der_1\phi=0,\quad \psi(0, x_2)=\int_{-1}^{x_2}\om_{\rm en}(t)\,dt,\quad \der_1\Psi=E_{\rm en}-E_0\quad&\tx{on $\Gamen$},\\
\phi=0,\quad \der_2\psi=0,\quad \der_2\Psi=0\quad&\tx{on $\Gamw$,}\\
\phi=0,\quad \Psi=0\quad&\tx{on $\Gamex$}.
\end{split}
\end{equation}
This problem is obtained as an approximated nonlinear boundary value problem for $(\phi, \psi, \Psi)$ by replacing $T$ with $\tilT\in \iterT(r_1)$ in \eqref{equation for psi}--\eqref{equation for phi}.
And, by applying the Schauder fixed point theorem, we can show that, for each $\tilT\in \iterT(r_1)$, there exists at least one solution $(\phi, \psi, \Psi)\in \iterseta$ to the associated nonlinear problem of \eqref{nlsystem-2nd order-o} and \eqref{BCs:nlsystem-2nd order-o}.
Further details can be given by following the proof of \cite[Theorems 1.6 and 1.7]{BDXX}, so we skip details. The only difference is that we should apply Proposition \ref{theorem-wp of lbvp for system with sm coeff} to guarantee a continuity of an iteration mapping with respect to an appropriately fixed norm.
\medskip

{\textbf{Step 3.}} Given $\tilT\in \mcl{I}_{\rm ent}(r_1)$, let $(\phi^{(1)}, \psi^{(1)}, \Psi^{(1)})$ and $(\phi^{(2)}, \psi^{(2)}, \Psi^{(2)})$ be two solutions to the problem \eqref{nlsystem-2nd order-o}--\eqref{BCs:nlsystem-2nd order-o}. And, suppose that both solutions are contained in $\iterseta$. And, let us set
\begin{equation}
\label{diff of sol-1}
 (u,v,w):=(\phi_1, \psi_1, \Psi_1)- (\phi_2, \psi_2, \Psi_2)\quad\tx{in $\Om_L$},
\end{equation}
and
\begin{equation*}
d_1:=\|u\|_{H^2(\Om_L)},\quad d_2:=\|(v,w)\|_{H^1(\Om_L)},\quad d:=d_1+d_2.
\end{equation*}
Then we have
\begin{equation}
\label{nlsystem-2nd order}
\begin{cases}
-\Delta u=F_0\quad &\tx{in $\Om_L$},\\
\mfrak{L}_1^{ (\phi_1, \psi_1, \Psi_1)}(v,w)
=F_1\quad &\tx{in $\Om_L$},\\
\mfrak{L}_2(v,w)=F_2\quad &\tx{in $\Om_L$},
\end{cases}
\end{equation}
with the boundary conditions
\begin{equation}
\label{BCs:nlsystem-2nd order-u}
\begin{split}
\der_1u=0,\quad v(0, x_2)=0,\quad \der_1w=0\quad&\tx{on $\Gamen$},\\
u=0,\quad \der_2 v=0,\quad \der_2 w=0\quad&\tx{on $\Gamw$},\\
u=0,\quad w=0\quad&\tx{on $\Gamex$},
\end{split}
\end{equation}
for
\begin{equation*}
  \begin{split}
  F_0&:=f_0^{(\tilT, (\phi_1, \psi_1, \Psi_1))}-f_0^{(\tilT, (\phi_2, \psi_2, \Psi_2))},\\
  F_1&:=f_1^{ (\phi_1, \psi_1, \Psi_1)}-f_1^{ (\phi_2, \psi_2, \Psi_2)}+(\mfrak{L}_1^{(\phi_2, \psi_2, \Psi_2)}-\mfrak{L}_1^{(\phi_1, \psi_1, \Psi_1)})(\psi_2, \Psi_2),\\
  F_2&:=f_2^{(\tilT, (\phi_1, \psi_1, \Psi_1))}-f_2^{(\tilT, (\phi_2, \psi_2, \Psi_2))}.
  \end{split}
\end{equation*}
By following the argument in Step 1 from \cite[Proof of Theorem 1.7]{BDXX} together with Lemma \ref{proposition-H1-apriori-estimate}, we can show that
\begin{align}
\label{contraction-1}
 & d_1\le C_{\sharp}r_1(d_1+d_2)\\
 \label{contraction-new}
 &  d_2\le C_{\sharp}\left(d_1+(r_1+r_2+r_3+\mcl{P}(S_0, E_{\rm en}, \om_{\rm en}))d_2\right)
\end{align}
for some constant $C_{\sharp}>0$.
The estimate \eqref{contraction-1} is easily obtained by investigating the linear elliptic boundary value problem for $u$, stated as a part of the problem \eqref{nlsystem-2nd order}--\eqref{BCs:nlsystem-2nd order-u}. The estimate \eqref{contraction-new} is obtained by applying Lemma  \ref{proposition-H1-apriori-estimate} to the boundary value problem for $(v,w)$ stated in \eqref{nlsystem-2nd order}--\eqref{BCs:nlsystem-2nd order-u}.

If the constant $r_1$ satisfies the inequality
\begin{equation}
\label{condition for r-3}
  r_1\le \frac{1}{2C_{\sharp}},
\end{equation}
then the estimate \eqref{contraction-1} immediately yields that
\begin{equation}\label{contraction-2}
  d_1\le 2C_{\sharp}r_1d_2.
\end{equation}
We combine \eqref{contraction-2} with \eqref{contraction-new}
to get
\begin{equation}\label{contraction-3}
  d_2\le C_{\sharp}\left((2C_{\sharp}+1)r_1+r_2+r_3+\mcl{P}(S_0, E_{\rm en}, \om_{\rm en})\right)d_2.
\end{equation}
Therefore, if $(r_1, r_2, r_3)$ are fixed sufficiently small and if $\mcl{P}(S_0, E_{\rm en}, \om_{\rm en})$ is sufficiently small to satisfy the condition
\begin{equation}
\label{condition for r-4}
  C_{\sharp}\left((2C_{\sharp}+1)r_1+r_2+r_3+\mcl{P}(S_0, E_{\rm en}, \om_{\rm en})\right)<1,
\end{equation}
then we obtain that $d_1=d_2=0$, and this implies that
\begin{equation*}
  (\phi_1, \psi_1, \Psi_1)=(\phi_2, \psi_2, \Psi_2)\quad\tx{in $\ol{\Om_L}$}.
\end{equation*}
\medskip

{\textbf{Step 4.}} Hereafter, we assume that all the conditions \eqref{condition for r-0}, \eqref{condition for r-1}, \eqref{condition for r-2}, \eqref{condition for r-3} and \eqref{condition for r-4} hold.

For a fixed $\tilT\in \iterT(r_1)$, let ${\bf m}(\Psi, \nabla\psi, \nabla^{\perp}\phi)$ be the momentum density field associated with $\tilT$ in the sense of Definition \ref{definition: momentum density field}. By the result obtained in the previous three steps, the vector field ${\bf m}(\Psi, \nabla\psi, \nabla^{\perp}\phi)$ is well defined in $\Om_L$. Note that it is divergence-free, that is, it satisfies the equation \eqref{div-free}. Next, we prove the well-posedness of Problem \ref{problem-transport equation}(see \S \ref{subsection:framework}).

\begin{lemma}
\label{lemma:wp for ivp, entropy}
Assuming that the conditions \eqref{condition for r-0}, \eqref{condition for r-1} and \eqref{condition for r-2} hold, one can fix a constant $\bar{r}_3>0$ depending only on the data so that if the inequality
\begin{equation}
\label{condition for r-6}
0<r_3\le \bar{r}_3
\end{equation}
holds, then, for each $\tilT\in \iterT_{\rm ent}(r_1)$, Problem \ref{problem-transport equation} has a unique solution $T\in H^4(\Om_L)$ that satisfies
\begin{equation}
\label{estimate of T}
  \|T\|_{H^4(\Om_L)}\le C_{\flat}\|S_{\rm en}-S_0\|_{C^4([-1,1])}
\end{equation}
for some constant $C_{\flat}>0$ fixed depending only on the data.
\end{lemma}
This lemma is an analogy of \cite[Lemma 3.5]{BDXX} except that there is one significant difference in its proof.

\begin{proof}
First of all, we fix a constant $\bar{r}_3$ to satisfy
\begin{equation*}
  \bar{r}_3\le \min\left\{1, \frac{C_*}{C_{\sharp}(2C_{\sharp}+3C_*+3)}\right\}
\end{equation*}
so that whenever the inequality $0<r_3\le \bar{r}_3$ holds, the conditions \eqref{condition for r-3} and \eqref{condition for r-4} automatically hold by \eqref{condition for r-0}, \eqref{condition for r-1} and \eqref{condition for r-2}.
\smallskip

For a fixed $\tilT\in \iterT(r_1)$, let ${\bf m}(\Psi, \nabla\psi, \nabla^{\perp}\phi)$ be the approximated momentum density field associated with $\tilT$ in the sense of Definition \ref{definition: momentum density field}. Since the solution $(\phi,\psi, \Psi)$ to the boundary value problem \eqref{nlsystem-2nd order-o}--\eqref{BCs:nlsystem-2nd order-o} lies in the iteration set $\itersetinrs$, by using the first inequality stated in \eqref{condition for r-2}, one can directly check the estimate
\begin{equation}
\label{estimate for m-1}
  \|{\bf m}(\Psi, \nabla\psi, \nabla^{\perp}\phi)-{\bf m}(0, {\bf 0}, {\bf 0})\|_{H^3(\Om_L)}\le Cr_3.
\end{equation}

Note that ${\bf m}(0, {\bf 0}, {\bf 0})=J\hat{\bf e}_1$. Next, we shall improve the regularity of ${\bf m}(\Psi, \nabla\psi, \nabla^{\perp}\phi)$ near the entrance boundary $\Gamen$.
\smallskip

{\emph{Claim: In $\Om_L\cap\{x_1<\frac{\ls}{2}\}$(see Lemma \ref{lemma-1d-full EP} for the definition of the constant $\ls$), the following estimate holds for some constant $C>0$ fixed depending only on the data:
\begin{equation}
\label{estimate for m-2}
\begin{split}
  &\|{\bf m}(\Psi, \nabla\psi, \nabla^{\perp}\phi)-{\bf m}(0, {\bf 0}, {\bf 0})\|_{H^4(\Om_L\cap\{x_1<\frac{\ls}{2}\})} \le Cr_3.
  \end{split}
\end{equation}
}}

{\emph{Verification of the claim:}}
Define a function $\breve{\psi}$ by
\begin{equation*}
  \breve{\psi}(x_1,x_2):=\psi(x_1, x_2)-\int_{-1}^{x_2}\om_{\rm en}(t)\,dt.
\end{equation*}
Then it follows from the boundary condition $\psi(0, x_2)=\int_{-1}^{x_2}\om_{\rm en}(t)\,dt$ that
\begin{equation*}
\breve{\psi}=0\quad\tx{on $\Gamen$}.
\end{equation*}
Next, we rewrite the equation $\mfrak{L}_1^{(\phi,\psi, \Psi)}(\psi, \Psi)=f_1^{(\phi, \psi, \Psi)}$ in terms of $\breve{\psi}$ as
\begin{equation}
\label{equation for u}
  a_{11}\der_{11}\breve{\psi}+2a_{12}\der_{12}\breve{\psi}+\der_{22}\breve{\psi}=F\quad\tx{in $\Om_L$}
\end{equation}
for
\begin{equation*}
  \begin{split}
  &F:=f_1^{(\phi, \psi, \Psi)}-\mfrak{L}_1^{(\phi, \psi, \Psi)}\left(\int_{-1}^{x_2}\om_{\rm en}(t)\,dt, \Psi\right)-a^{(\phi, \psi, \Psi)}\der_1 \breve{\psi},\\
  &(a_{11}, a_{12}):=(a_{11}^{(\phi, \psi, \Psi)}, a_{12}^{(\phi, \psi, \Psi)}).
  \end{split}
\end{equation*}
By using the compatibility conditions of $\om_{\rm en}$ stated in Condition \ref{conditon:1} and Lemma \ref{lemma on L_1}, we can directly check the following properties:
\begin{equation*}
\begin{split}
&\der_2 \breve{\psi}=0\quad\tx{on $\Gamw$},\\
&\der_2 a_{11}=0,\,\,a_{12}=\der_2^2a_{12}=0,\,\,\der_2 F=0\quad\tx{on $\Gamw$}.
\end{split}
\end{equation*}
Most importantly, it follows from Lemma \ref{lemma on L_1}($\tx{g}_3$) that
the equation \eqref{equation for u} is uniformly elliptic in $\Om_L\cap\{x_1<\frac 34 \ls\}$. Then we can use the estimates given in Lemma \ref{lemma on L_1}(e) and the estimate \eqref{potential estimate fixed pt} to show that
\begin{equation*}
  \|\psi\|_{H^5(\Om_L\cap\{x_1<\frac{\ls}{2}\})}\le Cr_3.
\end{equation*}
This together with \eqref{definition: vector field m} yields
\begin{equation}
\label{estimate for m-2}
  \|{\bf m}(\Psi, \nabla\psi, \nabla^{\perp}\phi)-{\bf m}(0, {\bf 0}, {\bf 0})\|_{H^4(\Om_L\cap\{x_1<\frac{\ls}{2}\})}
  \le Cr_3.
\end{equation}
The claim is verified.
\smallskip

\newcommand \vtheta{\vartheta}
For ${\bf m}(\rx):={\bf m}(\Psi(\rx), \nabla\psi(\rx), \nabla^{\perp}\phi(\rx))$, define a function $\vtheta:\ol{\Om_L}\rightarrow \R$ by
\begin{equation*}
  \vtheta(x_1, x_2):=\int_{-1}^{x_2} {\bf m}(x_1, t)\cdot \hat {\bf e}_1\,dt.
\end{equation*}
Then, one can directly check from \eqref{div-free} that
\begin{equation}
\label{derivative of vtheta}
  (\der_{x_1}, \der_{x_2})\vtheta=(-{\bf m}\cdot \hat {\bf e}_2, {\bf m}\cdot \hat {\bf e}_1)\quad\tx{in $\ol{\Om_L}$}.
\end{equation}
Owing to the estimate \eqref{estimate for m-1}, we can further reduce the constant $\bar{r}_3$ depending only on the data so that if $r_3\in (0,\bar{r}_3)$, then we have
    \begin{equation}
    \label{estimate of vtheta-2}
      \der_2\vtheta(\rx)\ge \frac J2\quad\tx{for all $\rx\in \ol{\Om_L}$}.
    \end{equation}
It follows from \eqref{derivative of vtheta} and the boundary conditions stated in \eqref{BCs:nlsystem-2nd order-o} that $\vartheta$ satisfies the boundary conditions
\begin{equation*}
\der_{x_1}\vartheta(x_1, \pm 1)=0\quad\tx{for all $0<x_1<L$. }
\end{equation*}
So we have
\begin{equation*}
  \vtheta(x_1,-1)=0\quad\tx{and}\quad \vartheta(x_1,1)=\vartheta(0,1)=:\bar{\vartheta}\quad\tx{for all $0\le x_1\le L$}.
\end{equation*}
And, it follows from the estimate \eqref{estimate for m-2} and the Morrey's inequality that
\begin{equation}
\label{estimate of vtheta-1}
  \|\vartheta-J(1+x_2)\|_{C^3(\ol{\Om_L\cap\{x_1<\frac{\ls}{3}\}})}\le
  Cr_3.
\end{equation}

Then we can define a {\emph{Lagrangian mapping}} $\mathscr{L}:\ol{\Om_L}\rightarrow [-1,1]$ by the implicit relation
\begin{equation}
\label{definition:vtheta}
  \vtheta(0, \mathscr{L}(\rx)):=\vtheta(\rx).
\end{equation}
Then
\begin{equation}
\label{expression of T}
  T(\rx):=(S_{\rm en}-S_0)\circ \mathscr{L}(\rx)
\end{equation}
solves Problem \ref{problem-transport equation}.

\newcommand \lmap{\mathscr{L}}
By direct calculations with using the definition of $\lmap$ given by \eqref{definition:vtheta}, and using the estimates  \eqref{estimate for m-1}, \eqref{estimate of vtheta-2} and \eqref{estimate of vtheta-1},
we can directly check that
\begin{equation*}
  \|\lmap\|_{H^3(\Om_L)}\le C.
\end{equation*}
Finally, let us consider the term $D^4\lmap(\rx)$. By a lengthy but straightforward computation, we get
\begin{equation*}
 D^4\lmap(\rx)=\frac{{\sum_{j=1}^5 \mathscr{D}_j}}{\der_{x_2}\vtheta(0, \lmap(\rx))}
\end{equation*}
for
\begin{equation*}
\begin{split}
 \mathscr{D}_1&=D^4\vtheta,\\
 \mathscr{D}_2&=-6\der_{x_2}^3\vtheta(0, \lmap)(D\lmap)^2D^2\lmap,\\
 \mathscr{D}_3&=-3\der_{x_2}^2\vtheta(0, \lmap)(D^2\lmap)^2,\\
 \mathscr{D}_4&=-4\der_{x_2}^2\vtheta(0, \lmap)D\lmap D^3\lmap,\\
 \mathscr{D}_5&=-\der_{x_2}^4\vtheta(0, \lmap)(D\lmap)^4.
 \end{split}
\end{equation*}
By using \eqref{derivative of vtheta} and applying the Sobolev inequality, we can directly estimate the terms $\mathscr{D}_j$ for $j=1,2,3,4$ as follows:
 \begin{equation*}
 \begin{split}
 &\|\mathscr{D}_1\|_{L^2(\Om_L)}\le C\|{\bf m}\|_{H^3(\Om_L)},\\
 &\|\mathscr{D}_j\|_{L^2(\Om_L)}\le C\|\vtheta\|_{C^3(\ol{\Om_L\cap\{x_1<\frac{\ls}{3}\}})}
 \|\lmap\|^3_{H^3(\Om_L)}\quad\tx{for $j=2,3,4$.}
 \end{split}
 \end{equation*}
And, applying the generalized Sobolev inequality gives
\begin{equation*}
  \|\mathscr{D}_5\|_{L^2(\Om_L)}
  \le C\|\lmap\|_{H^3(\Om_L)}^3\|\der_{x_2}^4\vtheta(0, \lmap)\|_{L^2(\Om_L)}.
\end{equation*}
For each $t\in (0, L)$, the function $\lmap(t, \cdot): [-1,1]\rightarrow [-1,1]$ is $C^1$-diffeomorphic. By \eqref{estimate for m-2} and \eqref{derivative of vtheta}, the function $\der_{x_2}^4\vtheta(0, \lmap(t,\cdot))$ belongs to $L^2((-1,1))$ due to the trace theorem. Then we apply the Fubini's theorem and use \eqref{estimate for m-2}--\eqref{estimate of vtheta-2} to get
\begin{equation*}
\begin{split}
  \|\der_{x_2}^4\vtheta(0, \lmap)\|_{L^2(\Om_L)}^2
  &=\int_{0}^L\int_{\Om_L\cap\{x_1=t\}}|\der_{x_2}^4\vtheta(0, \lmap(x_1,x_2))|^2\,dx_2dt\\
  &=\int_0^L\int_{\Om_L\cap\{x_1=t\}}\frac{|\der_{x_2}^4\vtheta(0, y)|^2}{\der_{x_2}\lmap(t, \lmap^{-1}(t,y))}\,dy\,dt\\
  &\le C
  \end{split}
\end{equation*}
where $\lmap^{-1}$ represents the inverse of $\lmap(t,\cdot)$ as a mapping of $x_2\in[-1,1]$. Then, the estimate \eqref{estimate of T} easily follows.

\end{proof}

{\textbf{Step 5.}}
In addition to the conditions \eqref{condition for r-0}, \eqref{condition for r-1}, \eqref{condition for r-2} and \eqref{condition for r-6}, if the inequality
\begin{equation}
\label{condition for r-7}
  \mcl{P}(S_{\rm en}, E_{\rm en}, \om_{\rm en})\le \frac{r_1}{2C_{\flat}}
\end{equation}
holds, then it follows from Lemma \ref{lemma:wp for ivp, entropy} that, for each $\tilT\in \iterT(r_1)$, the solution $T$ to Problem \ref{problem-transport equation} satisfies the estimate
\begin{equation*}
  \|T\|_{H^4(\Om_L)}\le \frac{r_1}{2}.
\end{equation*}
Furthermore, by using the representation \eqref{expression of T} and the compatibility conditions for $S_{\rm en}$ stated in Condition \ref{conditon:1}(i), it can be checked that the solution $T$ satisfies the compatibility conditions
\begin{equation*}
  \der_2^kT=0\quad\tx{on $\Gamw$ for $k=1,3$}.
\end{equation*}
This implies $T\in \iterT(r_1)$. Then we can apply the Schauder fixed point theorem and the uniqueness of a solution to Problem \ref{problem-transport equation} to conclude that the nonlinear boundary value problem of \eqref{equation for psi}--\eqref{BC for T} has at least one solution $(T, \phi, \psi, \Psi)\in \iterT(r_1)\times \iterV(r_2)\times \iterP(r_3)$ provided that $(r_1, r_2, r_3, \mcl{P}(S_{\rm en}, E_{\rm en}, \om_{\rm en}) )$ satisfy all the conditions \eqref{condition for r-0}, \eqref{condition for r-1}, \eqref{condition for r-2}, \eqref{condition for r-6} and \eqref{condition for r-7}. To verify this statement, one can simply repeat the argument in the proof of  \cite[Theorem 1.7]{BDXX}. For further details, see \S 3.2.2 in \cite{BDXX}.

\medskip

{\textbf{Step 6.}}{\emph{The existence of a solution to Problem \ref{problem-HD}}} : For the solution $(T,\phi, \psi, \Psi)$ to the boundary value problem of \eqref{equation for psi}--\eqref{BC for T}, let us set
\begin{equation*}
  (S, \vphi, \Phi):=(S_0, \bar{\vphi}, \bar{\Phi})+(T, \psi, \Psi)\quad\tx{in $\Om_L$}.
\end{equation*}
For the constant $\bar{r}_3$ from Lemma \ref{lemma:wp for ivp, entropy}, one can fix a constant $\hat{r}_3\in(0, \bar{r}_3]$ depending only on the data so that whenever
\begin{equation}\label{condition for r_3-equiv}
  r_3\le \hat{r}_3,
\end{equation}
we can combine the first inequality in \eqref{condition for r-2} with \eqref{condition for r_3-equiv} so that $(\vphi, \phi, \Phi, S)$ satisfy all the conditions \eqref{smallness condition-1}--\eqref{smallness condition-3}. Then it follows that $(\vphi, \phi, \Phi, S)$ is a solution to Problem \ref{problem-HD}.
\smallskip

Now, we fix three constants $r_1$, $r_2$ and $r_3$ as follows:
\begin{equation}
\label{choice of r's}
  \begin{split}
  r_3&=\frac{\mcl{P}(S_{\rm en}, E_{\rm en}, \om_{\rm en})}{\max\{\frac{1}{3C_*}, \frac{C_*+1}{12C_*C_{\flat}}\}}=:\kappa \mcl{P}(S_{\rm en}, E_{\rm en}, \om_{\rm en}),\\
  r_1&=\frac{r_3}{6C_*(2C_*+1)}=\frac{\kappa}{6C_*(2C_*+1)} \mcl{P}(S_{\rm en}, E_{\rm en}, \om_{\rm en})=:\mu_1\mcl{P}(S_{\rm en}, E_{\rm en}, \om_{\rm en}),\\
  r_2&=\frac{r_3}{3(2C_*+1)}=\frac{\kappa}{3(2C_*+1)} \mcl{P}(S_{\rm en}, E_{\rm en}, \om_{\rm en})=:\mu_2\mcl{P}(S_{\rm en}, E_{\rm en}, \om_{\rm en}),
  \end{split}
\end{equation}
for the constants $C_*$ and $C_{\flat}$ from \eqref{definition of C_*} and \eqref{estimate of T}, respectively. If the term $\mcl{P}(S_{\rm en}, E_{\rm en}, \om_{\rm en})$ satisfies the inequality
\begin{equation}
\label{condition for data perturbation-2}
\begin{split}
  &\mcl{P}(S_{\rm en}, E_{\rm en}, \om_{\rm en})\\
  &\le \min\left\{\frac{1}{3\kappa C_*}, \frac{\bar{\delta}}{\mu_1+\mu_2+\kappa}, \frac{1}{2C_{\sharp}\mu_1}, \frac{\hat{r}_3}{\kappa}, \frac{1}{2C_{\sharp}((2C_{\sharp}+1)\mu_1+\mu_2+\kappa+1)}\right\}=:\sigma_*
  \end{split}
\end{equation}
for the constants $\bar{\delta}$, $C_{\sharp}$ and $\hat{r}_3$ from \eqref{condition for r-0}, \eqref{contraction-1}(or \eqref{contraction-new}) and \eqref{condition for r_3-equiv}, respectively, then all the conditions \eqref{condition for r-0}, \eqref{condition for r-1}, \eqref{condition for r-2}, \eqref{condition for r-3}, \eqref{condition for r-4}, \eqref{condition for r-7} and \eqref{condition for r_3-equiv} hold. Since $(T, \phi, \psi, \Psi)$ is contained in the iteration set $\iterT(r_1)\times \iterV(r_2)\times \iterP(r_3)$, it follows from \eqref{choice of r's} that $(\vphi, \phi, \Phi, S)$ satisfies the estimate \eqref{solution estimate HD}.
\smallskip

 {\emph{The uniqueness of a solution}} :\\
Suppose that $(\vphi^{(1)}, \Phi^{(1)}, \phi^{(1)}, S^{(1)})$ and $(\vphi^{(2)}, \Phi^{(2)}, \phi^{(2)}, S^{(2)})$ are two solutions to Problem \ref{problem-HD}. In addition, suppose that they satisfy the estimate \eqref{solution estimate HD} with the term $\mcl{P}(S_{\rm en}, E_{\rm en}, \om_{\rm en}) $ satisfying the condition \eqref{condition for data perturbation-2}. For each $j=1$ and $2$, let us set
\begin{equation*}
  (\psi^{(j)}, \Psi^{(j)}, T^{(j)}):=
  (\vphi^{(j)}, \Phi^{(j)}, S^{(j)})-(\bar{\vphi}, \bar{\Phi}, S_0),\quad P^{(j)}:=(\phi^{(j)},\psi^{(j)}, \Psi^{(j)}).
\end{equation*}
Define
\begin{equation*}
  (\til u, \til v, \til w, \wtil{Y}):=(\phi^{(1)}, \vphi^{(1)}, \Phi^{(1)}, T^{(1)})-(\phi^{(2)}, \vphi^{(2)}, \Phi^{(2)}, T^{(2)}),
\end{equation*}
and
\begin{equation*}
\sigma:=\mcl{P}(S_{\rm en}, E_{\rm en}, \om_{\rm en}).
\end{equation*}
By adjusting the argument in Step 3, one can show that
\begin{align}
\label{estimate:final-0}
  &\|\til u\|_{H^2(\Om_L)}\le C\left(\|\wtil{Y}\|_{H^1(\Om_L)}+\sigma(\|\til u\|_{H^2(\Om_L)}+\|(\til v, \til w)\|_{H^1(\Om_L)})\right),\\
\label{estimate:final-1}
  &\|(\til v, \til w)\|_{H^1(\Om_L)}\le C\left(\|\til u\|_{H^2(\Om_L)}+\sigma\|(\til v, \til w) \|_{H^1(\Om_L)}\right).
\end{align}

By repeating the argument in the proof of \cite[Theorem 1.7]{BDXX} (see \S 3.2.2 in \cite{BDXX}), one can find a constant $\hat{\sigma}\in (0, \sigma_*]$ depending only on the data so that if $\sigma\le \hat{\sigma}$, then it holds that
\begin{equation}
\label{estimate:final-2}
  \|\wtil{Y}\|_{H^1(\Om_L)}\le C\sigma \left(\|\wtil{Y}\|_{H^1(\Om_L)}+\|(\til v, \til w)\|_{H^1(\Om_L)}+\|\til u\|_{H^2(\Om_L)}\right).
\end{equation}
Therefore, one can fix a constant $\bar{\sigma}\in(0,\hat{\sigma}]$ sufficiently small depending only on the data so if
$\sigma\le \bar{\sigma}$, then it follows from the estimates \eqref{estimate:final-0}--\eqref{estimate:final-2} that $$(\til u, \til v, \til w, \wtil{Y})=(0,0,0,0)\quad\tx{in $\Om_L$}.$$ 
This proves the uniqueness of a solution to Problem \ref{problem-HD}.

\medskip

{\textbf{Step 7.}}
For a solution $(\vphi, \Phi, \phi, S)$ to Problem \ref{problem-HD}, define its associated Mach number $M(\vphi, \Phi, \phi, S)$ by
\begin{equation*}
  M(\vphi, \Phi, \phi, S):=\frac{|\nabla\vphi+\nabla^{\perp}\phi|}{\sqrt{\gam S\varrho^{\gam-1}(S, \Phi, \nabla\vphi, \nabla^{\perp}\phi)}}
\end{equation*}
for $\varrho(S, \Phi, \nabla\vphi, \nabla^{\perp}\phi)$ given by \eqref{new system with H-decomp1}.

Note that $T(=S-S_0)\in \iterT(r_1)$ and $P:=(\phi,\psi, \Psi)(=(\phi, \vphi, \Phi)-(0, \bar{\vphi}, \bar{\Phi}))\in \iterV(r_2)\times \iterP(r_3)$. For $(a_{ij}^P)_{i,j=1,2}$ given by Definition \ref{definition:approx coeff and fs}, define a function $D:\Om_L\rightarrow \R$ by
\begin{equation*}
  D(\rx)=\det \begin{pmatrix}a_{11}^P(\rx) & a_{12}^P(\rx)\\ a_{12}^P(\rx) & 1 \end{pmatrix}\quad\tx{for $\rx=(x_1, x_2)\in \Om_L$}.
\end{equation*}
It follows from Lemma \ref{lemma on L_1}($\tx{g}_2$) that there exists a unique function $\gs^P:[-1,1]\rightarrow (0,L)$ satisfying that
\begin{equation}\label{signs of D}
  D(x_1, x_2)\begin{cases}
  >0\quad&\mbox{for $x_1<\gs^P(x_2)$},\\
  =0\quad&\mbox{for $x_1=\gs^P(x_2)$},\\
  <0\quad&\mbox{for $x_1>\gs^P(x_2)$}.
  \end{cases}
\end{equation}
By a direct computation with Definitions \ref{definition:coefficints-iter} and \ref{definition:coefficients-nonlinear}, it can be directly checked that
\begin{equation*}
  D(\rx)=\frac{\left(\gam S\varrho^{\gam-1}(S, \Phi, \nabla\vphi, \nabla^{\perp}\phi)\right)^2}{A_{22}^2(\rx, \Psi, \nabla\psi, \nabla^{\perp}\phi)}\left(1-M^2(\vphi, \Phi, \phi, S)\right).
\end{equation*}
Therefore, \eqref{signs of D} is equivalent to the following:
\begin{equation}\label{mag. of M}
M(\vphi, \Phi, \phi, S)\begin{cases}
  <1\quad&\mbox{for $x_1<\gs^P(x_2)$},\\
  =1\quad&\mbox{for $x_1=\gs^P(x_2)$},\\
  >1\quad&\mbox{for $x_1>\gs^P(x_2)$}.
  \end{cases}
\end{equation}
Let us set a function $\fsonic$ as
\begin{equation*}
  \fsonic(x_2):=\gs^P(x_2)\quad\tx{for $-1\le x_2\le 1$}.
\end{equation*}
Then, it follows from \eqref{estimate of gs}, \eqref{estimate:phi}, \eqref{potential estimate fixed pt} and \eqref{choice of r's} that the {\emph{sonic interface function}} $\fsonic$ satisfies the estimate \eqref{estimate of sonic boundary pt}. This completes the proof of Theorem \ref{theorem-HD}. \end{proof}

\appendix

\section{A further discussion on the nozzle length $L$}
\label{appendix:nozzle length}

Returning to the proof of Lemma \ref{proposition-H1-apriori-estimate}, we give examples of $(\gam, J)$ for which the nozzle length $L$ can be large while the condition \eqref{almost sonic condition1 full EP} holds.
\medskip

By \eqref{L-computation-n} and \eqref{definition of lambda}, one can express the nozzle length $L$ as
\begin{equation}
\label{new expression of L}
L:=\sqrt{\frac{h_0^3}{2}}J^{\frac{\gam-2}{\gam+1}}\sqrt{\lambda(\kappa_0,\kappa_L)}.
\end{equation}
Define
\begin{equation*}
  \mfrak{p}:=\frac{\sqrt 2}{2}h_0^{-\frac 32}\mcl{H}(\kappa)((\gam-1)\kappa^{\gam+1}+\eta(\kappa^{\gam+1}-1)+2)\kappa^{-\eta}h_0^{-\eta},
\end{equation*}
and
\begin{equation*}
  \hat{\alp}:=\frac{\alp}{J^{\frac{2-\gam+2\eta}{\gam+1}}\mfrak{p}}.
\end{equation*}
for $\alp$ given by \eqref{definition of beta*}.
\medskip

{\textbf{Example 1.}} Back to Step 4 in the proof of Lemma \ref{proposition-H1-apriori-estimate},
where the parameter $\eta$ in \eqref{definition of G} is fixed as $\eta=\frac 34\gam$, we take a closer look on how the parameter $\bJ$ is fixed.

By \eqref{new expression of alp} and \eqref{expression of alp}, we can roughly express $\hat{\alp}$ as
\begin{equation*}
  \hat{\alp}=
  1-\mu_1(\kappa,h_0,\gam)J^{\frac{\gam}{2(\gam+1)}}
  -\mu_2(\kappa,h_0,\gam)\lambda(\kappa_0, \kappa_L)J^{-\frac{\gam}{2(\gam+1)}}\left(J^{\frac{\gam}{2(\gam+1)}}+\mu_3(\kappa, h_0, \gam)\right)^2.
\end{equation*}
Next, we express $\lambda(\kappa_0, \kappa_L)$ as
\begin{equation}
\label{new expression of lambda}
\lambda(\kappa_0, \kappa_L)=\eps_0J^{\frac{\gam}{2(\gam+1)}}.
\end{equation}
Then we have
\begin{equation*}
  \hat{\alp}=1-\mu_1(\kappa,h_0,\gam)J^{\frac{\gam}{2(\gam+1)}}
  -\mu_2(\kappa,h_0,\gam)\eps_0\left(J^{\frac{\gam}{2(\gam+1)}}+\mu_3(\kappa, h_0, \gam)\right)^2.
\end{equation*}
Assume that
\begin{equation*}
  |\kappa-1|\le q_0
\end{equation*}
for some constant $q_0>0$. Then one can fix a constant $\bar{\mu}>0$ depending only on $(\gam, S_0, \zeta_0, q_0)$ to satisfy
\begin{equation*}
 \max_{{j=1,2,3,}\atop{|\kappa-1|\le q_0}}  |\mu_j(\kappa, h_0, \gamma)|\le \bar{\mu}
\end{equation*}
so we have
\begin{equation*}
  \hat{\alp}\ge 1-\bar{\mu}J^{\frac{\gam}{2(\gam+1)}}-
  \bar{\mu}\eps_0\left(J^{\frac{\gam}{2(\gam+1)}}+\bar{\mu}\right)^2.
\end{equation*}
Now we fix a constant $\bJ>0$ sufficiently small depending only on $(\gam, S_0, \zeta_0, q_0)$  so that, whenever the background momentum density $J$ satisfies $0<J\le \bJ$, it holds that
\begin{equation*}
  1-\bar{\mu}J^{\frac{\gam}{2(\gam+1)}}\ge \frac 23.
\end{equation*}
Then, we fix a constant $\eps_0>0$ sufficiently small depending only on $(\gam, S_0, \zeta_0, q_0)$ so that if $J\in(0, \bJ]$, then we finally obtain that
\begin{equation}
\label{lower bd of hat alp}
  \hat{\alp}\ge \frac 13,
\end{equation}
and this yields the result comparable to \eqref{positive lower bound 1 of alp}.

Substituting \eqref{new expression of lambda} into \eqref{new expression of L} yields
\begin{equation*}
  L=\sqrt{\frac{\eps_0 h_0^3}{2}}J^{\frac{5\gam-8}{4(\gam+1)}}.
\end{equation*}
This expression is valid for any $\gam>1$ as long as the inequality $0<J\le \bJ$ holds. Suppose that $1<\gam<\frac{8}{5}$. Note that the constant $\eps_0h_0^3$ is fixed independent of $J\in(0, \bJ]$. Therefore, the value of $L$ becomes large if $J$ is fixed sufficiently small.

\medskip

{\textbf{Example 2.}} Back to Step 4 in the proof of Lemma \ref{proposition-H1-apriori-estimate},
where the parameter $\eta$ in \eqref{definition of G} is fixed as $\eta=\frac{\gam}{4}$, repeat the argument given in Example 1 to get
\begin{equation*}
    \hat{\alp}=1-\mu_1(\kappa,h_0,\gam)J^{\frac{-\gam}{2(\gam+1)}}
  -\mu_2(\kappa,h_0,\gam)\eps_0\left(J^{\frac{-\gam}{2(\gam+1)}}+\mu_3(\kappa, h_0, \gam)\right)^2.
\end{equation*}
Then one can easily see that there exist a constant $\ubJ>1$ sufficiently large, and a small constant $\eps_0>0$ depending only on $(\gam, S_0, \zeta_0, q_0)$ so that, for any $J\in [\ubJ, \infty)$, if $\lambda(\kappa_0, \kappa_L)$ is expressed as
\begin{equation*}
  \lambda(\kappa_0, \kappa_L)=\eps_0J^{\frac{-\gam}{2(\gam+1)}},
\end{equation*}
then the estimate \eqref{lower bd of hat alp} holds true, thus the result comparable to \eqref{positive lower bound 2 of alp} is obtained. In this case, the nozzle length $L$ is expressed as
\begin{equation*}
  L=\sqrt{\frac{\eps_0 h_0^3}{2}}J^{\frac{3\gam-8}{4(\gam+1)}}.
\end{equation*}
So we conclude that the value of $L$ becomes large if $\gam>\frac{8}{3}$ and $J$ is fixed sufficiently large.

\vspace{.25in}
\noindent
{\bf Acknowledgements:}
The research of Myoungjean Bae was supported in part by  Samsung Science and Technology Foundation under Project Number SSTF-BA1502-51.
The research of Ben Duan was supported in part by NSFC No. 12271205 and No. 12171498. And, the research of Chunjing Xie was partially supported by NSFC grants 11971307, 12161141004, Fundamental Research Grants for Central Universities Natural Science Foundation of Shanghai 21ZR1433300, Program of Shanghai Academic Research Leader 22XD1421400 .

\bigskip
\bibliographystyle{acm}
\bibliography{References_EP_smooth_tr}

\end{document}